\documentclass[10pt]{amsart}
\usepackage{amsmath,amssymb,latexsym,cancel,rotating}
\usepackage{graphicx,amssymb,mathrsfs,amsmath,color,fancyhdr,amsthm}
\usepackage[all]{xy}
\usepackage{setspace}

\newtheorem{theorem}{Theorem}[section]
\newtheorem{lemma}[theorem]{Lemma}
\newtheorem{corollary}[theorem]{Corollary}
\newtheorem{definition}[theorem]{Definition}
\newtheorem{proposition}[theorem]{Proposition}
\newtheorem{conjecture}[theorem]{Conjecture}
\theoremstyle{definition}
\newtheorem{remark}[theorem]{Remark}

\def\bV{\mathbf{V}}
\def\bG{\mathbf{G}}
\def\bfEVO{\mathbf{E}_{\mathbf{V},\Omega}}
\def\bfEVH{\mathbf{E}_{\mathbf{V},H}}
\def\bfEVOf{\mathbf{E}_{\mathbf{V}\oplus\mathbf{W}^{1},\Omega^{(1)}}}
\def\bfEVOfpp{\mathbf{E}_{\mathbf{V}''\oplus\mathbf{W}^{1},\Omega^{(1)}}}
\def\bfEVOff{\mathbf{E}_{\mathbf{V}\oplus\mathbf{W}^{1}\oplus \mathbf{W}^2,\Omega^{(2)}}}

\def\bfEVHf{\mathbf{E}_{\mathbf{V}\oplus\mathbf{W},H^{(1)}}}

\def\bfEVOfid{\dot{\mathbf{E}}^{i}_{\mathbf{V}\oplus\mathbf{W}^{1},\Omega^{(1)}}}

\def\bfEVOfidpp{\dot{\mathbf{E}}^{i}_{\mathbf{V}''\oplus\mathbf{W}^{1},\Omega^{(1)}}}

\def\mD{\mathcal{D}}
\def\mP{\mathcal{P}}
\def\mQ{\mathcal{Q}}
\def\mN{\mathcal{N}}
\def\mK{\mathcal{K}}
\def\mL{\mathcal{L}}
\def\Hom{{\rm{Hom}}}
\def\BM{{\rm{BM}}}
\def\CC{{\rm{CC}}}
\def\ftnu{\boldsymbol{\nu}}
\def\Uv{{_{\mathcal{A}} \mathbf{U}_{v} }}
\def\Uvm{{_{\mathcal{A}} \mathbf{U}^{-}_{v} }}

\begin{document}
\begin{spacing}{1}
\title[]
{Lusztig sheaves, characteristic cycles and the Borel-Moore homology of Nakajima's quiver varieties}
\author[Jiepeng Fang, Yixin Lan]{Jiepeng Fang, Yixin Lan}

\address{Department of Mathematics and New Cornerstone Science Laboratory,The University of Hong Kong, Pokfulam, Hong Kong, Hong Kong SAR, P. R. China}
\email{fangjp@hku.hk (J.Fang)}

\address{Academy of Mathematics and Systems Science, Chinese Academy of Sciences, Beijing 100190, P.R.China}
\email{lanyixin@amss.ac.cn (Y.Lan)}


\date{\today}

\bibliographystyle{abbrv}


\maketitle

\begin{center}
	\small\emph{Dedicated to Professor George Lusztig with admiration}
\end{center}

\begin{abstract}
	By using characteristic cycles, we build a morphism from the canonical bases of integrable highest weight modules of quantum groups to the top Borel-Moore homology groups of Nakajima's quiver and tensor product varieties, and compare the canonical bases and the fundamental classes. As an application, we show that Nakajima's realization of irreducible highest weight modules and their tensor products can be defined over $\mathbb{Z}$. We also give a new proof of Nakajima's conjecture on the canonical isomorphism of  tensor product varieties.
\end{abstract}

\setcounter{tocdepth}{1}\tableofcontents

\section{Introduction}\label{introduction}
Given a symmetric generalized Cartan matrix $C=(a_{ij})_{i,j\in I}$, the matrix $C$ defines a bilinear form $\langle -,-\rangle: \mathbb{Z}I \times \mathbb{Z} I \rightarrow \mathbb{Z}$, and there is a Kac-Moody Lie algebra $\mathfrak{g}$ associated to $C$.  The quantum group $\mathbf{U}_{v}(\mathfrak{g})$ is the $\mathbb{Q}(v)$-algebra generated by $E_i, F_i, K_{\nu}$ for $i\in I, \nu\in \mathbb{Z}I$, subject to the following relations:
\begin{itemize}
	\item {\rm{(a)}} $K_0 = 1$ and $K_{\nu}K_{\nu'} = K_{\nu+\nu'}$ for any $\nu,\nu' \in \mathbb{Z}I$;
	\item {\rm{(b)}} $K_{\nu}E_i = v^{\langle \nu,i \rangle} E_i K_{\nu}$ for any $i \in I, \nu \in \mathbb{Z}I$;
	\item {\rm{(c)}} $K_{\nu}F_i = v^{-\langle \nu,i \rangle} F_i K_{\nu}$ for any $i \in I, \nu \in \mathbb{Z}I$;
	\item {\rm{(d)}} $E_i F_j - F_j E_i = \delta_{ij} \frac{K_{i} - K_{-i}}{v - v^{-1}}$ for any $i,j\in I$;
	\item {\rm{(e)}} $\sum\limits_{p+q=1-a_{ij}} (-1)^p E_i^{(p)} E_j E_i^{(q)} = 0$ for any $i\not=j$;
	\item {\rm{(f)}} $\sum\limits_{p+q=1-a_{ij}} (-1)^p F_i^{(p)} F_j F_i^{(q)} = 0$ for any $i\not=j$.
\end{itemize}
Here $E_i^{(n)}= E_i^n/[n]!$, $F_i^{(n)}= F_i^n/[n]!$ are the divided power of $E_i,F_{i}$, where $[n]! := \prod \limits_{s=1}^n \frac{v^s - v^{-s}}{v - v^{-1}}$.  Taking the classical limit $v\rightarrow 1$, we obtain the universal enveloping algebra $\mathbf{U}(\mathfrak{g})$ of $\mathfrak{g}$ from the quantum group.

The quantum group $\mathbf{U}_{v}(\mathfrak{g})$ admits a triangular decomposition 
$$ \mathbf{U}_{v}(\mathfrak{g}) = \mathbf{U}^{+}_{v}(\mathfrak{g}) \otimes \mathbf{U}^{0}_{v}(\mathfrak{g})\otimes\mathbf{U}^{-}_{v}(\mathfrak{g}),$$ where $\mathbf{U}^{+}_{v}(\mathfrak{g})$ is the subalgebra generated by $E_{i},i\in I$, $\mathbf{U}^{-}_{v}(\mathfrak{g})$ is generated by $F_{i},i\in I$,  and $\mathbf{U}^{0}_{v}(\mathfrak{g})$ is generated by $K_{\nu},\nu \in \mathbb{Z}I$. 

Given a dominant weight $\lambda$, the Verma module $M_{v}(\lambda)$ is realized by 

\begin{equation*}
	\begin{split}
		M_{v}(\lambda) =& \mathbf{U}_{v}(\mathfrak{g})/ ( \sum\limits_{i\in I} \mathbf{U}_{v}(\mathfrak{g}) E_i + \sum\limits_{\nu \in \mathbb{Z}[I]} \mathbf{U}_{v}(\mathfrak{g}) (K_{\nu} - v^{\langle \nu,\lambda \rangle})).
	\end{split}
\end{equation*}
The irreducible integrable highest weight module $L_{v}(\lambda)$ is the unique irreducible quotient of the Verma module $M_{v}(\lambda)$, and can be realized by
\begin{equation*}
	\begin{split}
		L_{v}(\lambda) =& \mathbf{U}_{v}(\mathfrak{g})/ ( \sum\limits_{i\in I} \mathbf{U}_{v}(\mathfrak{g}) E_i + \sum\limits_{\nu \in \mathbb{Z}[I]} \mathbf{U}_{v}(\mathfrak{g}) (K_{\nu} - v^{\langle \nu,\lambda \rangle}) + \sum\limits_{i \in I} \mathbf{U}_{v}(\mathfrak{g}) F_i^{\langle i,\lambda \rangle +1})\\
		\cong & \mathbf{U}^{-}_{v}(\mathfrak{g})/ (  \sum\limits_{i \in I} \mathbf{U}^{-}_{v}(\mathfrak{g}) F_i^{\langle i,\lambda \rangle +1}).
	\end{split}
\end{equation*}
We denote the image of $1$  by $v_{\lambda}$, it is the highest weight vector of $L_{v}(\lambda)$. Taking classical limit $v\rightarrow 1$, we obtain the irreducible highest weight $\mathfrak{g}$-module $L_{1}(\lambda)$.

Let $\mathcal{A}=\mathbb{Z}[v,v^{-1}]$, the integral form $_{\mathcal{A}} \mathbf{U}_{v}(\mathfrak{g})$ of the quantum group is the $\mathcal{A}$-subalgebra generated by $K_{\nu},\nu \in \mathbb{Z}$ and the divided powers of $E_{i},F_{i},i \in I$. Similarly, the integral form $_{\mathcal{A}} \mathbf{U}^{\pm}_{v}(\mathfrak{g})$ is the $\mathcal{A}$-subalgebra of $ \mathbf{U}^{\pm}_{v}(\mathfrak{g})$ generated by the divided powers. We also let $_{\mathcal{A}}L_{v}(\lambda) \subseteq L_{v}(\lambda)$  be the  $_{\mathcal{A}} \mathbf{U}_{v}(\mathfrak{g})$-module generated by $v_{\lambda}$. Then if we take specialization at $v=1$, we get the $\mathbb{Z}$-form $_{\mathbb{Z}}\mathbf{U}(\mathfrak{g})$ of  $\mathbf{U}(\mathfrak{g})$ from $_{\mathcal{A}} \mathbf{U}_{v}(\mathfrak{g})$, and get the irreducible highest weight $_{\mathbb{Z}}\mathbf{U}(\mathfrak{g})$-module $_{\mathbb{Z}}L_{1}(\lambda)$   from $_{\mathcal{A}}L_{v}(\lambda)$.

\subsection{Background}
\subsubsection{The canonical basis of Lusztig's integral form} 
For a given symmetric generalized Cartan matrix $C$, there is a finite quiver (without loops) $Q=(I,H,\Omega)$ associated to $C$ such that the symmetric Euler form of $Q$ coincides with the bilinear form defined by $C$.
In \cite{MR1088333},\cite{MR1227098} and \cite{MR1653038}, G.Lusztig has considered a certain class $\mQ$ of semisimple complexes of  perverse sheaves on the moduli space $\bfEVO$ of  representations of $Q$. These semisimple complexes are called Lusztig's sheaves.   
Lusztig has also defined the induction and restriction functors for these complexes, which make the Grothendieck group of Lusztig's sheaves  become a bialgebra. This bialgebra is canonically isomorphic to the integral form ${_{\mathcal{A}}\mathbf{U}_{v}^{+}}(\mathfrak{g})$ (or ${_{\mathcal{A}}\mathbf{U}_{v}^{-}}(\mathfrak{g})$) of the positive (or negative) part of the quantum group.  Moreover, the simple Lusztig sheaves forms the canonical basis. One can also see details in Section 3.

\subsubsection{Realization of integrable modules of universal enveloping algebras}
It is an important problem to give a geometric realization of integrable modules of quantum groups or enveloping algebras. Nakajima has provided a construction for integrable modules of $\mathfrak{g}$ by using Borel-Moore homology groups of quiver varities in \cite{MR1604167} and \cite{MR1865400}.  One can also see details in Section 5.

In \cite{MR1604167}, Nakajima introduces the stable condition and constructs certain variety $\mathfrak{L}(\omega)$ and $\mathfrak{M}(\omega)$ by using symplectic geometry and GIT quotients. He uses the convolution of Hecke correspondences to define the $\mathfrak{g}$-module structure on the top BM  homology group (with $\mathbb{Q}$-coefficients) $\mathbf{H}(\mathfrak{L}(\omega))$ of $\mathfrak{L}(\omega)$ and shows that $\mathbf{H}(\mathfrak{L}(\omega))$ is canonically isomorphic to $L_{1}(\lambda)$, where $\lambda$ is a dominant weight corresponding to $\omega$. 

Later in \cite{MR1865400}, Nakajima has considered a certain $\mathbb{C}^{\ast}$-action on  $\mathfrak{M}(\omega)$ and  defined his tensor product variety $\tilde{\mathfrak{Z}}$. The convolution of Hecke correspondences also gives a $\mathfrak{g}$-module structure on the top BM  homology group of $\tilde{\mathfrak{Z}}$, and Nakajima has showed  that  $\mathbf{H}(\tilde{\mathfrak{Z}})$  is isomorphic to the tensor product $L_{1}(\lambda) \otimes L_{2}(\lambda)$. Since his proof relies on crystal structure and the calculation of characters, the isomorphism $\mathbf{H}(\tilde{\mathfrak{Z}}) \cong L_{1}(\lambda) \otimes L_{2}(\lambda)$  is not canonical. Nakajima has conjectured  that there is a unique canonical isomorphism satisfies certain properties.  (See \cite[Conjecture 5.6]{MR1865400} or Conjecture 5.7 of the present paper for details.) This conjectural canonical isomorphism has been realized by the stable envelope in \cite{maulik2019quantum}.

\subsubsection{Realization of integrable modules of quantum groups} 
A quantized version of Nakajima's work is expected and there are many successful results. For example, Kang and Kashiwara have provided a categorical realization of irreducible integrable modules of quantum groups via cyclotomic KLR algberas in \cite{MR2995184}. Webster has given a categorification of tensor products of irreducible integrable modules of quantum groups by using diagrammatic approaches  in \cite{webster2015canonical}. However, one may still expect a geometric realization from Lusztig's theory.

Following Lusztig's theory of canonical basis and also inspired by Zheng's work \cite{zheng2014categorification}, the authors together with the collaborator consider certain class $\mQ_{\nu,\omega}$ of semisimple complexes on the moduli space of framed quivers $Q^{(1)}$ (over $\overline{\mathbb{F}}_{q}$) and its localization in \cite{fang2023lusztig}. They provide a categorical construction of $_{\mathcal{A}}L_{v}(\lambda)$ and realize  the canonical basis of $_{\mathcal{A}}L_{v}(\lambda)$ by simple perverse sheaves. Moreover, they study the relation between the canonical basis at $v=1$ and the fundamental classes of $\mathfrak{L}(\omega)$ by combinatoric methods and prove that the transition matrix of these two bases is upper-triangular with respect to an order from crystals.

After that, the authors in \cite{fang2023tensor} introduce the 2-framed quivers and consider Lusztig's sheaves $\mQ_{\nu,\omega^1,\omega^2}$ on the  moduli space of 2-framed quivers $Q^{(2)}$. They have showed that the Grothendieck group of the localization of $\mQ_{\nu,\omega^1,\omega^2}$ is canonically isomorphic to the tensor product of $_{\mathcal{A}}L_{v}(\lambda_{1})$ and $_{\mathcal{A}}L_{v}(\lambda_{2})$, and the simple perverse sheaves of framed quivers provide the canonical bases of tensor products defined  in \cite{lusztig1992canonical} and \cite{bao2016canonical}.

\subsection{The results of the present paper} 
The purpose of the present article is to study the relation between the canonical basis of integrable highest weight modules and the fundamental classes of Nakajima's quiver varieties via characteristic cycles, which is a continuation of  \cite{fang2023lusztig} and \cite{fang2023tensor}.

In order to apply characteristic cycle map defined by Hennecart in \cite{hennecart2024geometric}, we need to generalize the main results of \cite{fang2023lusztig} and \cite{fang2023tensor} from varieties over $\overline{\mathbb{F}}_{q}$ to complex varieties.  We consider the  $\mathbb{T}= \prod\limits_{i \in I} \mathbb{G}^{i}_{m}$-action on the moduli spaces  and  define certain localizations $\mQ_{\nu,\omega^{1}}/\mathcal{N}_{\nu}$ and $\mQ_{\nu,\omega^1,\omega^2}/\mathcal{N}_{\nu}$ of $\mathbf{G}_{\bV}\times \mathbb{T}$-equivariant Lusztig sheaves, then the Grothendieck groups $\mL(\omega^{1})$ and $\mL(\omega^1,\omega^2)$ are canonically isomorphic to irreducible highest weight modules and their tensor products. Moreover, simple perverse sheaves provide the canonical bases of these modules.
\begin{theorem}
	The Grothendieck group $\mL(\omega^{1})$ of the localization of Lusztig sheaves on framed quiver   is canonically isomorphic to the  irreducible highest weight module ${_{\mathcal{A}}L}_{v}(\lambda_{1})$,
	$$\chi^{\omega^{1}}: \mL(\omega^{1}) \rightarrow  {_{\mathcal{A}}L}_{v}(\lambda_{1}).$$
	Let $L_{0}$ be the constant sheaf on $\bfEVOf$ with $\mathbf{V}=0$, then $[L_{0}]$ is sent to the highest weight vector. Moreover, the images of simple perverse sheaves in $\mL(\omega^{1})$ form the canonical basis of ${_{\mathcal{A}}L}_{v}(\lambda_{1})$.
\end{theorem}

\begin{theorem} 
	Let $\Delta$ be the $\mathcal{A}$-linear map induced by the functor  
	$  (\mathrm{sw})_{!}(\mathbf{D} \boxtimes \mathbf{D})\mathbf{Res}^{\bV\oplus \mathbf{W}^{1} \oplus \mathbf{W}^{2}}_{\mathbf{V}^{1}\oplus\mathbf{W}^{1},\mathbf{V}^{2}\oplus \mathbf{W}^{2}} \mathbf{D}, $
	then morphism $ \Delta: \mL(\omega^1,\omega^2) \rightarrow \mL(\omega^{2}) \otimes \mL(\omega^{1})$ is an isomorphism of $\Uv$-modules. In particular, $\mL(\omega^1,\omega^2)$ is canonically isomorphic to the tensor product of $_{\mathcal{A}}L_{v}(\lambda_{2}) \otimes {_{\mathcal{A}}L}_{v}(\lambda_{1})$ via $ \tilde{\chi}^{\omega^1,\omega^2}=(\chi^{\omega^{2}} \otimes \chi^{\omega^{1}}  )\Delta$. Moreover, the images of simple perverse sheaves in $\mL(\omega^{1})$ form the canonical basis of $_{\mathcal{A}}L_{v}(\lambda_{2}) \otimes {_{\mathcal{A}}L}_{v}(\lambda_{1})$.
\end{theorem}

 Consider the specialization of the Grothendieck groups  $\mL(\omega^{1})$ and $\mL(\omega^1,\omega^2)$ at $v=-1$ and take certain sign twists determined by Euler forms of the quivers, we can apply the characteristic cycle maps defined  in \cite{hennecart2024geometric}, and build  isomorphisms $\CC^{s,\omega^{1}}$ and $\CC^{s,\omega^1,\omega^2}$ of $_{\mathbb{Z}}\mathbf{U}(\mathfrak{g})$-modules from the twisted Grothendieck groups to the top Borel-Moore homology groups of Nakajima's quiver varieties and tensor products varieties.
 
 \begin{theorem}
 	The morphism $\CC^{s,\omega^{1}}$ induces an isomorphism of $_{\mathbb{Z}}\mathbf{U}(\mathfrak{g})$-modules,
 	$$  \CC^{s,\omega^{1}}:\mL(\omega^{1})^{\psi}_{-1} \longrightarrow \mathbf{H}(\mathfrak{L}(\omega),\mathbb{Z}).$$
 	Moreover, we have the following commutative diagram
 	\[
 	\xymatrix{
 		\mL(\omega^{1})^{\psi}_{-1} \ar[rr] ^{\CC^{s,\omega^{1}}} \ar[dr]_{\chi^{\omega^{1},\psi}_{-1}} &   &  	\mathbf{H}(\mathfrak{L}(\omega^{1}),\mathbb{Z}) \ar[dl]^{\varphi^{\omega^{1}}}\\
 		& _{\mathbb{Z}}L_{1}(\lambda_{1}), &  
 	}
 	\]
 	where $\chi^{\omega^{1},\psi}_{-1}$ is the  twisted version of our canonical isomorphism in Theorem 1.1 and $\varphi^{\omega^{1}}$ is Nakajima's canonical isomorphism.
 \end{theorem}
 
 \begin{theorem}
 	The morphism $\CC^{s,\omega^1,\omega^2} =\bigoplus_{\nu \in \mathbb{N}I}\CC^{s}_{\bG_{\bV}}:  \mK(\omega^1,\omega^2)_{-1}^{\psi} \rightarrow \mathbf{H}(\tilde{\mathfrak{Z}},\mathbb{Z})$ induces an isomorphism of $_{\mathbb{Z}}\mathbf{U}(\mathfrak{g})$-modules,
 	$$  \CC^{s,\omega^1,\omega^2}:\mL(\omega^1,\omega^2)^{\psi}_{-1} \longrightarrow \mathbf{H}(\tilde{\mathfrak{Z}},\mathbb{Z}).$$
 \end{theorem}
 
 As an application, we can give a new proof of Nakajima's conjecture, which has been mentioned in Section 1.2.
 \begin{theorem}
 	Let $\delta:\mathbf{H}(\tilde{\mathfrak{Z}},\mathbb{Z}) \longrightarrow \mathbf{H}(\mathfrak{L}(\omega^{2}),\mathbb{Z})  \otimes \mathbf{H}(\mathfrak{L}(\omega^{1}),\mathbb{Z})$ be the unique isomorphism such that the following commutative diagram commutes
 	\[
 	\xymatrix{
 		\mL(\omega^1,\omega^2)^{\psi}_{-1}   \ar[d]_{\CC^{s,\omega^1,\omega^2}}  \ar[r]^-{\Delta} &  \mL(\omega^{2})^{\psi}_{-1}  \otimes \mL(\omega^{1})^{\psi}_{-1}  \ar[d]^{\CC^{s,\omega^{2}} \otimes \CC^{s,\omega^{1}}}      \\
 		\mathbf{H}(\tilde{\mathfrak{Z}},\mathbb{Z}) \ar[r]^-{\delta} & \mathbf{H}(\mathfrak{L}(\omega^{2}),\mathbb{Z})  \otimes \mathbf{H}(\mathfrak{L}(\omega^{1}),\mathbb{Z}) ,
 	}
 	\]
 	where $\Delta$ is the morphism in Theorem 1.2 and those $\CC^{s}$ are the morphisms in Theorem 1.3 and 1.4  respectively. Then after base change to $\mathbb{Q}$, $\delta$ is the unique isomorphism which satisfies the conditions in Nakajima's conjecture.
 \end{theorem}

 A immediate corollary is that Nakajima's construction of irreducible modules and their tensor products can be defined on homology groups with $\mathbb{Z}$-coefficients.  Since the characteristic cycle of a perverse sheaf is nonnegative, we also have the following corollaries after taking sign twists for BM homology groups. 
 \begin{corollary}
 	There is a commutative diagram of $_{\mathbb{Z}}\mathbf{U}_{-1}(\mathfrak{g})$-isomorphisms,
 	\[
 	\xymatrix{
 		\mL(\omega^{1})_{-1} \ar[rr] ^{\CC^{s,\omega^{1}}} \ar[dr]_{\chi^{\omega^{1}}_{-1}} &   &  	\mathbf{H}(\mathfrak{L}(\omega),\mathbb{Z})^{\psi} \ar[dl]^{\varphi^{\omega^{1},\psi}}\\
 		& _{\mathbb{Z}}L_{-1}(\lambda_{1}). &  
 	}
 	\]
 	Moreover, the transition matrix between the canonical basis at $v=-1$ and the fundamental classes is an upper triangular (with respect to the refine string order defined in Section 2.3) matrix, whose diagonal elements are $1$ and the other elements are non-negative integers. More precisely, if $Z$ is the irreducible component corresponding to the simple perverse sheaf $L$, we have the following equation 
 	$$\chi^{\omega^{1}}_{-1}([L])=\varphi^{\omega^{1},\psi}( [Z] + \sum\limits_{Z \prec Z'}  c_{Z,Z'}[Z']), $$ 	with constants $c_{Z,Z'} \in \mathbb{N} $.
 \end{corollary}

 \begin{corollary}
 	There is a commutative diagram of
 	$_{\mathbb{Z}}\mathbf{U}_{-1}(\mathfrak{g})$-isomorphisms 
 	\[
 	\xymatrix{
 		\mL(\omega^1,\omega^2)_{-1} \ar[rr] ^{\CC^{s,\omega^1,\omega^2}} \ar[dr]_{\chi^{\omega^1,\omega^2}_{-1}} &   &  	\mathbf{H}(\tilde{\mathfrak{Z}},\mathbb{Z})^{\psi} \ar[dl]^{\varphi^{\omega^1,\omega^2,\psi}}\\
 		& {_{\mathbb{Z}}L}_{-1}(\lambda_{2}) \otimes {_{\mathbb{Z}}L}_{-1}(\lambda_{1}).&  
 	}
 	\]
 	The transition matrix between the canonical basis of the tensor product ${_{\mathbb{Z}}L}_{-1}(\lambda_{2}) \otimes {_{\mathbb{Z}}L}_{-1}(\lambda_{1})$ at $v=-1$ and the fundamental classes in $	\mathbf{H}(\tilde{\mathfrak{Z}},\mathbb{Z})^{\psi}$ is an upper triangular (with respect to the string order) matrix, whose diagonal elements are $1$ and the other elements are non-negative integers. More precisely, if $Z$ is the irreducible component corresponding to the simple perverse sheaf $L$, we have the following equation 
 	$$\tilde{\chi}^{\omega^1,\omega^2}_{-1}([L])=\tilde{\varphi}^{\omega^1,\omega^2,\psi}( [Z] + \sum\limits_{Z \prec Z'}  c_{Z,Z'}[Z']), $$ 	where  $c_{Z,Z'}$ are constants in $\mathbb{N} $,   $Z'$ runs over irreducible components of $\tilde{\mathfrak{Z}}$, and $\prec$ is the refine string order associated to $Q^{(2)}$.
 \end{corollary}

 \subsection{The structure of the present paper}
 In Section 2, we recall some preliminaries. In Section 3, we recall Lusztig's theory on canonical basis of quantum groups. In Section 4, we generalize the sheaf theoretic construction of integrable modules of quantum groups to complex varieties. In Section 5, we recall the construction of integrable modules via Nakajima's quiver varieties and tensor product varieties. In Section 6, we construct the characteristic cycle maps and prove our main theorems.

\subsection*{Acknowledgements.}
J. Fang is supported by the New Cornerstone Science Foundation through the New Cornerstone Investigator Program awarded to Professor Xuhua He.  Y. Lan is supported by the National Natural Science Foundation of China [Grant No. 1288201]. This paper is a continuation of a collaborative work \cite{fang2023lusztig} with Prof. Jie Xiao, we are very grateful to his detailed discussion. We are also  grateful to Prof. Yiqiang Li and Prof. Hiraku Nakajima.  Li  pointed out that the sign twists of Chevalley generators in \cite{MR1604167} and \cite{fang2023lusztig}  are incorrect, and Nakajima explained how to deduce the correct sign twists from his work \cite{nakajima2001quiver} for us.

\subsection*{Convention}
Throughout this paper, all smooth varieties are  over $\mathbb{C}$ with complex analytic topology and all sheaves are of $\mathbb{C}$-vector spaces.\\
(1) Let $X$ be a smooth variety. We denote by $\mathcal{D}^b_c(X)$ the bounded derived category of constructible sheaves on $X$ and denote by $\mathbf{H}^{\BM}_{*}(X,\mathbb{Z})$ the Borel-Moore homologies of $X$ which are normalised such that $\mathbf{H}^{\BM}_{2\dim X}(X,\mathbb{Z})$ has a basis given by fundamental classes of irreducible components of $X$.\\
(2) Let $X$ be an algebraic variety and $G$ be a connected algebraic group acting on $X$. We denote by $\mathcal{D}^b_G(X)$ the $G$-equivariant derived category of constructible sheaves on $X$ and denote by $\mathbf{H}^{\BM,G}_{*}(X,\mathbb{Z})$ the $G$-equivariant Borel-Moore homologies of $X$ which are normalised such that $\mathbf{H}^{\BM}_{2\dim X-2\dim G}(X,\mathbb{Z})$ has a basis given by fundamental classes of irreducible components of $X$.\\
(3) Let $X$ be an algebraic variety and $S\subset X$ be a locally closed subvariety. We denote by $T^\ast X$ the cotangent bundle on $X$ and denote by $T^\ast_SX$ the conormal bundle on $S$ in $X$.

Since we only work on derived categories, we denote by $f^{\ast},f_{\ast},f^{!},f_{!}$ their derived functors for simplicity.   We denote the $n$-times shift functor in a triangulated category  by $[n]$ and we always omit the Tate twists. The subscript $\pm 1$ means the specialization of a quantized algebra or module at $v=\pm 1$, and the subscript $\mathbb{Z}$ or $\mathcal{A}$ means the $\mathbb{Z}$-form or the integral form of an algebra or module.  For example, $_{\mathbb{Z}}\mathbf{U}_{-1}(\mathfrak{g})$ means the $\mathbb{Z}$-form of the specialization of the quantum group at $v=-1$.

\section{Preliminaries}
\subsection{Equivariant characteristic cycle and induction}
In this subsection, we recall the formalism of the equivariant characteristic cycle in \cite[Section 2.2]{hennecart2024geometric} and the induction in \cite[Section 4]{hennecart2024geometric}. Let $X$ be a smooth algebraic variety and $G$ be an algebraic group acting on $X$.

\begin{definition}
    For any fixed $G$-invariant Whitney stratification $\mathcal{S}_X$ of $X$, we denote by $\mathcal{D}^b_{c,\mathcal{S}_X}(X)$ the full subcategory of $\mathcal{D}^b_c(X)$ consisting of $\mathcal{S}_X$-constructible complexes and denote by $\mathcal{D}^b_{G,\mathcal{S}_X}(X)$ the full subcategory of $\mathcal{D}^b_G(X)$ consisting of $\mathcal{S}_X$-constructible complexes. Let $\Lambda=\coprod_{S \in \mathcal{S}_{X}}T^{\ast}_{S}X$. We define the equivariant characteristic cycle to be the composition
	$$\CC_{G}:\mK_{0}(\mD^{b}_{G,\mathcal{S}_{X}}(X)) \xrightarrow{[{\rm{For}}^G]} \mK_{0}(\mD^{b}_{c,\mathcal{S}_{X}}(X)) \xrightarrow{\CC} \mathbf{H}^{\BM}_{2{\rm{dim}}X}(\Lambda,\mathbb{Z}) \cong \mathbf{H}^{\BM,G}_{2{\rm{dim}}X-2{\rm{dim}}G}(\Lambda,\mathbb{Z}),$$
	where $\mK_{0}(\mD^{b}_{c,\mathcal{S}_{X}}(X)),\mK_{0}(\mD^{b}_{G,\mathcal{S}_{X}}(X))$ are the Grothendieck groups of $\mathcal{D}^b_{c,\mathcal{S}_X}(X),\mathcal{D}^b_{G,\mathcal{S}_X}(X)$ respectively and the map $[{\rm{For}}^G]$ is induced by the forgetful functor
    \begin{align*}
    {\rm{For}}^{G}:\mD^{b}_{G,\mathcal{S}_{X}}(X) &\rightarrow \mD^{b}_{c,\mathcal{S}_{X}}(X)\\
    \mathcal{F}&\mapsto \mathcal{F}.
    \end{align*}
\end{definition}

We refer \cite[Section 2.1]{hennecart2024geometric} for the definition of the characteristic cycle $\CC:\mK_{0}(\mD^{b}_{c,\mathcal{S}_{X}}(X)) \rightarrow \mathbf{H}^{BM}_{2{\rm{dim}}X}(\Lambda,\mathbb{Z})$. We remark that we do not take the shifts $[\dim G]$ in the forgetful functor ${\rm{For}}^{G}$ which is different from \cite[Section 2.2]{hennecart2024geometric}, since we use the perverse t-structure defined in \cite{MR1299527}, also see \cite[Section 6.4]{MR4337423}, which is different from the perverse t-structure used in \cite{hennecart2024geometric}.

\subsubsection{Equivariant induction diagram}
Consider the following morphisms of smooth varieties 
$$Y \xleftarrow{p} V \xrightarrow{q} X',$$ 
where $p$ is smooth of relative dimension $d$ and $q$ is a closed embedding. Let the algebraic group $G$ act on $X'$ and let $P$ be a parabolic subgroup of $G$ acting on $Y,V$ such that and $p,q$ are $P$-equivariant. Let $X=G\times^{P}Y, W=G\times^{P} V$ be the parabolic inductions of $Y,V$ respectively, then there are $G$-equivariant morphisms of smooth varieties 
\begin{equation}\label{equivariant induction diagram}
X \xleftarrow{f} W \xrightarrow{g} X'
\end{equation}
defined by $f(m,v)=(m,p(v)),g(m,v)=mq(v)$ for any $(m,v)\in W$. Note that $f$ is smooth of relative dimension $d$, $g$ is proper and $(f,g):W \rightarrow X\times X'$ is a closed embedding. We refer \cite[Section 3.2]{hennecart2024geometric} for details.

\subsubsection{Induction of equivariant complex}
Consider the diagram (\ref{equivariant induction diagram}), we can define a functor
$$\mathbf{Ind}=g_{\ast}f^{\ast}[d]: \mD^{b}_{G}(X) \rightarrow \mD^{b}_{G}(X'). $$
Composing it with the induction equivalence $\mD^{b}_{P}(Y) \cong \mD^{b}_{G}(X)$, see \cite[Theorem 6.5.10]{MR4337423}, we obtain a functor 
$$\mathbf{Ind}: \mD^{b}_{P}(Y) \rightarrow \mD^{b}_{G}(X').$$
We refer \cite[Section 4.2]{hennecart2024geometric} for details.

\subsubsection{Induction of (equivariant) Borel-Moore homology}
For any morphism of smooth varieties $h:X_1\rightarrow X_2$, its cotangent correspondence is the diagram
$$T^\ast X_2\xleftarrow{{\mathrm{pr}}}T^\ast X_2\times_{X_2} X_1\xrightarrow{(dh)^\ast}T^\ast X_1,$$
where ${\mathrm{pr}}$ is the natural projection and $(dh)^\ast$ is given by the pullback of covectors. Consider the diagram (\ref{equivariant induction diagram}), the cotangent correspondences of the morphisms $f,g$ give the following diagrams
\begin{align*} 
T^{\ast}X \xleftarrow{{\mathrm{pr}}} T^{\ast}X&\times_{X}W \xrightarrow{(df)^{\ast}} T^{\ast}W,\\
T^{\ast}X' \xleftarrow{{\mathrm{pr}}} T^{\ast}X'&\times_{X'}W \xrightarrow{(dg)^{\ast}} T^{\ast}W.
\end{align*}
Let $Z=T^{\ast}_{W}(X \times X')$, then there is a Cartesian square
$$\xymatrix{
Z \ar[r]^-{\psi'} \ar[d]_{\phi'} & T^{\ast}X'\times_{X'} W \ar[d]^{(dg)^{\ast}} \\
T^{\ast}X\times_{X} W \ar[r]^{(df)^{\ast}} & T^{\ast}W,}$$
where $\phi':Z\rightarrow T^{\ast}X\times_{X} W$ is the natural morphism $Z\hookrightarrow T^\ast(X\times X')\times_{(X\times X')}W\xrightarrow{{\mathrm{pr}}}T^{\ast}X\times_{X} W$ and $\psi':Z\rightarrow T^{\ast}X'\times_{X'} W$ is the composition of the natural morphism $Z\hookrightarrow T^\ast(X\times X')\times_{(X\times X')}W\xrightarrow{{\mathrm{pr}}}T^{\ast}X'\times_{X'} W$ and the automorphism on $T^\ast X'$ acting by $-1$ on the fibers. We define $\phi: Z \rightarrow T^{\ast}X$ to be the composition 
$$Z\xrightarrow{\phi'}T^{\ast}X\times_{X} W\xrightarrow{{\mathrm{pr}}}T^\ast X,$$
and define $-\psi: Z \rightarrow T^{\ast}X'$ to be the composition 
$$Z\xrightarrow{\psi'}T^{\ast}X'\times_{X} W\xrightarrow{{\mathrm{pr}}}T^\ast X'\xrightarrow{-1}T^\ast X',$$
where $-1$ is the automorphism on $T^\ast X'$ acting by $-1$ on the fibers. We refer \cite[Section 3.5]{hennecart2024geometric} for details.

For any closed, conical, Lagrangian subvarieties $\Lambda \subseteq T^{\ast}X$ and $\Lambda' \subseteq T^{\ast}X'$ satisfying $(-\psi)\phi^{-1}(\Lambda) \subseteq \Lambda'$, we can define a map 
$$ \mathbf{Ind}=(-\psi)_{\ast}\phi^{!}: \mathbf{H}^{\BM}_{\ast}(\Lambda,\mathbb{Z}) \rightarrow \mathbf{H}^{\BM}_{\ast-2{\rm{dim}}X +2{\rm{dim}}X' }(\Lambda',\mathbb{Z}).$$
We refer \cite[Section 4.3]{hennecart2024geometric} for details. In particular, for a fixed $P$-invariant Whitney stratification $\mathcal{S}_Y$ of $Y$, we set $\Lambda_Y=\coprod_{S \in \mathcal{S}_{Y}}T^{\ast}_{S}Y$ and $\Lambda=G\times^P \Lambda_Y$, then there is an induction isomorphism 
$$\mathbf{H}^{\BM}_{2{\rm{dim}}X}(\Lambda,\mathbb{Z}) \cong \mathbf{H}^{\BM}_{2{\rm{dim}}Y}(\Lambda_{Y},\mathbb{Z}).$$
Composing it with the map $\mathbf{Ind}: \mathbf{H}^{\BM}_{\ast}(\Lambda,\mathbb{Z}) \rightarrow \mathbf{H}^{\BM}_{\ast-2{\rm{dim}}X +2{\rm{dim}}X' }(\Lambda',\mathbb{Z})$, we obtain a map 
 $$ \mathbf{Ind}: \mathbf{H}^{\BM}_{2{\rm{dim}}Y}(\Lambda_{Y},\mathbb{Z}) \rightarrow \mathbf{H}^{\BM}_{2{\rm{dim}}X'}(\Lambda',\mathbb{Z}).$$
We refer \cite[Section 4.5]{hennecart2024geometric} for details.
 
\subsubsection{Compatibility of induction operations}
 
Let $\mD^{b}_{G}(X,\Lambda)$ be the full subcategory of $\mD^{b}_{G}(X)$ consisting of objects whose singular supports are contained in $\Lambda$ and $\mK_{0}(\mD^{b}_{G}(X,\Lambda))$ be its Grothendieck group.
 
\begin{theorem}[{\cite[Theorem 6.3]{hennecart2024geometric}}]\label{com}
The induction operations of equivariant complexes and Borel-Moore homologies are compatible via the equivariant characteristic cycle, that is, for the diagram (\ref{equivariant induction diagram}), the following diagram
$$\xymatrix{
 		\mK_{0}( \mD^{b}_{P}(Y,\Lambda_{Y}) ) \ar[r]^-{\mathbf{Ind}}   \ar[d]_{\CC_{P}} & \mK_{0}( \mD^{b}_{G}(X',\Lambda') ) \ar[d]^{\CC_{G}} \\
 		\mathbf{H}^{\BM}_{2{\rm{dim}}Y}(\Lambda_{Y},\mathbb{Z})\ar[r]^{\mathbf{Ind}} & \mathbf{H}^{\BM}_{2{\rm{dim}}X'}(\Lambda',\mathbb{Z})}$$
commutes.
\end{theorem}

\section{Geometric realization of $_{\mathcal{A}} \mathbf{U}^{-}_{v}(\mathfrak{g})$ and canonical bases}
In this section, we recall the categorification and canonical basis theory in \cite{MR1088333} and \cite{MR1227098}.

\subsection{Lusztig sheaves}\label{Lusztig sheaf}
We refer {\cite[Section 9.1]{MR1227098}} for details of this subsection.

For the symmetric generalized Cartan matrix $C=(a_{ij})_{i,j\in I}$, we associate it with a finite graph $(I,H)$, where $I$ is the set of vertices and $H$ is the set of pairs consisting of an edge together with an orientation of it, that is, there are two maps $h\mapsto h',h\mapsto h''$ from $H$ to $I$, and a fixed point free involution $h\mapsto \bar{h}$ of $H$ such that $h'\not= h'',(\bar{h})'=h''$ for any $h\in H$, and 
$$\{h\in H\mid h'=i,h''=j\ \textrm{or}\ h'=j,h''=i\}=-2a_{ij}\ \textrm{for}\ i\not=j.$$
An orientation of this graph is a subset $\Omega\subset H$ such that $\Omega\cup\overline{\Omega}=H,\Omega\cap\overline{\Omega}=\varnothing$. Given an orientation, we obtain a finite quiver $Q=(I,H,\Omega)$ without loops, where $I$ is the set of vertices, and $\Omega$ is the set of arrows such that for any $h\in \Omega$, its source and target are $h'$ and $h''$ respectively. We assume $Q$ doesn't have loops and oriented cycles, but may have multiple arrows.

For any $I$-graded $\mathbb{C}$-vector space $\bV=\bigoplus_{i\in I}\bV_i$, we denote by $|\bV|=\sum_{i\in I}\dim \bV_i i\in \mathbb{N}I$ the dimension vector of $\bV$. For any $\nu\in \mathbb{N}I$, we fix a $I$-graded $\mathbb{C}$-vector space $\bV$ such that $|\bV|=\nu$ and define an affine space
$$\bfEVO=\bigoplus_{h\in \Omega} \Hom (\bV_{h'},\bV_{h''})$$
such that the connected algebraic group 
$$\bG_{\bV}=\prod_{i \in I} \mathrm{GL}(\bV_{i})$$ acts on $\bfEVO$ by conjugation $g.x=(g_{h''}x_hg_{h'}^{-1})_{h\in \Omega}$. 

Let $\mathcal{V}_{\nu}$ be the set of sequences $\ftnu= (\nu_{1},\nu_{2},...,\nu_{m})$ in $\mathbb{N}I$ such that each $\nu_{l}=n_li_l$ for some $n_l\in \mathbb{N},i_l\in I$ and $\sum_{l=1}^m \nu_{l} =\nu$. For any $\ftnu= (\nu_{1},\nu_{2},...,\nu_{m})\in \mathcal{V}_{\nu}$, let $\tilde{\mathcal{F}}_{\ftnu}$ be the variety of pairs $(x,f)$, where $x \in \bfEVO$ and $f=(\mathbf{V}=\mathbf{V}^0\supseteq \bV^1\supseteq...\supseteq \bV^m=0)$ is a sequence of $I$-graded subspaces of $\mathbf{V}$ satisfying $|\mathbf{V}^{l-1}/\mathbf{V}^{l}|=\nu^{l}$ and $x_h(\bV^l_{h'})\subseteq \bV^l_{h''}$ for any $h\in \Omega, l=1,...,m$, then $\bG_{\bV}$ acts on $\tilde{\mathcal{F}}_{\ftnu}$ by $g.(x,f)=(g.x,g.f)$, where $g.f=(\mathbf{V}=g\mathbf{V}^0\supseteq g\bV^1\supseteq...\supseteq g\bV^m=0)$.
\begin{lemma}[{\cite[Lemma 1.6]{{MR1088333}}}]
{{\rm{(a)}}} The variety $\tilde{\mathcal{F}}_{\ftnu}$ is smooth and irreducible. \\
{{\rm{(b)}}} The natural projection $\pi_{\ftnu}:\tilde{\mathcal{F}}_{\ftnu}\rightarrow \bfEVO$ is proper and $\bG_{\bV}$-equivariant.
\end{lemma}
By the decomposition theorem, see \cite[Theorem 5.3]{MR1299527}, the complex 
$$L_{\ftnu}=(\pi_{\ftnu})_!(\mathbb{C}_{\tilde{\mathcal{F}}_{\ftnu}}[\dim  \tilde{\mathcal{F}}_{\ftnu}])$$
is a $\bG_{\bV}$-equivariant semisimple complex on $\bfEVO$. Let $\mP_{\nu}$ be the full subcategory of $\mathcal{D}^b_{\bG_{\bV}}(\bfEVO)$ consisting of simple perverse sheaves $L$ such that $L[n]$ is a direct summand of $L_{\ftnu}$ for some $n\in \mathbb{Z},\ftnu\in \mathcal{S}_\nu$, and let $\mQ_{\nu}$ be the full subcategory of $\mathcal{D}^b_{\bG_{\bV}}(\bfEVO)$ consisting of direct sums of complexes of the form $L[n]$ for any $L\in \mP_{\nu},n\in \mathbb{Z}$. We call $\mQ_\nu$ the category of Lusztig sheaves.

\subsection{Induction and restriction functors}\label{Induction and restriction functors}
We refer \cite[Section 9.2]{MR1227098} for details of this subsection.

For any $\nu=\nu'+\nu''\in \mathbb{N}I$ and $I$-graded $\mathbb{C}$-vector spaces $\bV=\mathbf{V}' \oplus \mathbf{V}''$ such that $|\bV|=\nu,|\bV'|=\nu',|\bV''|=\nu''$, let $F$ be the closed subset of $\bfEVO$ consisting of $x$ such that $x_h(\bV_{h'}'')\subseteq \bV_{h''}''$ for any $h\in \Omega$. For any $x\in F$, there are natural linear maps $x'_h:\bV'_{h'}\rightarrow \bV'_{h''}, x''_h:\bV''_{h'}\rightarrow \bV''_{h''}$ induced by $x_h:\bV_{h'}\rightarrow \bV_{h''}$ for any $h\in \Omega$, and then we obtain elements $x'=(x'_h)_{h\in \Omega}\in \mathbf{E}_{\bV',\Omega},x''=(x''_h)_{h\in \Omega}\in \mathbf{E}_{\bV'',\Omega}$. Consider the following diagram 
\begin{equation}\label{resd}
	\mathbf{E}_{\mathbf{V}',\Omega} \times \mathbf{E}_{\mathbf{V}'',\Omega} \xleftarrow{\kappa } F \xrightarrow{\iota} \mathbf{E}_{\mathbf{V},\Omega},
\end{equation}
where $\iota(x)=x, \kappa(x)=(x',x'')$. Let $P\subseteq \bG_{\bV}$ be the stabilizer of $\bV''$ which is a parabolic subgroup, and let $U\subseteq P$ be the unipotent radical. Consider the following diagrams
\begin{equation}\label{indd}
	\mathbf{E}_{\mathbf{V}',\Omega} \times \mathbf{E}_{\mathbf{V}'',\Omega} \xleftarrow{p_{1}} \bG_{\bV} \times^{U} F \xrightarrow{p_{2}} \bG_{\bV} \times^{P} F \xrightarrow{p_{3}} \mathbf{E}_{\mathbf{V},\Omega},
\end{equation}
where $p_{1}(g,x) =\kappa (x), p_{2}(g,x)=(g,x), p_{3}(g,x)=g_{\cdot}\iota(x)$. 

The induction functor functor is defined by 
\begin{align*}
	\mathbf{Ind}^{\mathbf{V}}_{\mathbf{V}',\mathbf{V}''}: 
    &\mathcal{D}^b_{\bG_{\bV'}}(\mathbf{E}_{\mathbf{V}',\Omega}) \boxtimes \mathcal{D}^b_{\bG_{\bV''}}(\mathbf{E}_{\mathbf{V}'',\Omega})\rightarrow  \mathcal{D}^b_{\bG_{\bV}}(\bfEVO),\\
    &(A\boxtimes B) \mapsto (p_{3})_{!}(p_{2})_{\flat}(p_{1})^{\ast}(A\boxtimes B)[d_{1}-d_{2}],
\end{align*}
where $(p_{2})_{\flat}$ is the equivariant decent functor of the principal bundle $p_2$ and $d_1,d_2$ are the dimension of the fibers of $p_1,p_2$ respectively. The restriction functor is defined by
\begin{align*}
	\mathbf{Res}^{\mathbf{V}}_{\mathbf{V}',\mathbf{V}''}: &\mathcal{D}^b_{\bG_{\bV}}(\bfEVO) \rightarrow \mathcal{D}^{b}_{G_{\mathbf{V}'} \times G_{\mathbf{V}''}}(\mathbf{E}_{\mathbf{V}',\Omega}\times \mathbf{E}_{\mathbf{V}'',\Omega}), \\
    &C \mapsto (\kappa)_{!} (\iota)^{\ast}(C)[-\langle\nu',\nu''\rangle_{Q}],
\end{align*}
where $\langle \nu',\nu''\rangle_{Q}=\sum_{i\in I}\nu'_i\nu''_i-\sum_{h\in \Omega}\nu'_{h'}\nu''_{h''}$ is the Euler form of the quiver $Q$.

\begin{proposition}[{\cite[Section 9.2.7, 9.2.11]{MR1227098}}] \label{indres formula}
For any $\ftnu'\in \mathcal{V}_{\nu'},\ftnu''\in \mathcal{V}_{\nu''}$, we have 
	\begin{equation*}
		\mathbf{Ind}^{\mathbf{V}}_{\mathbf{V}',\mathbf{V}''}(L_{\ftnu'} \boxtimes L_{\ftnu''})\cong L_{\ftnu'\ftnu''},
	\end{equation*}
where $\ftnu' \ftnu''\in \mathcal{V}_{\nu}$ is the sequence formed by elements in the sequence $\ftnu'$ followed by elements in the sequence $\ftnu''$. For any $\ftnu\in \mathcal{V}_{\nu}$, we have 
	\begin{equation*}
		\mathbf{Res}^{\mathbf{V}}_{\mathbf{V}',\mathbf{V}''}( L_{\ftnu}) =\bigoplus L_{\ftnu'} \boxtimes L_{\ftnu''}[M(\ftnu',\ftnu'')],
	\end{equation*}
where the direct sum is taken over $(\ftnu',\ftnu'')\in \mathcal{V}_{\nu'}\times \mathcal{V}_{\nu''}$ such that $\ftnu'+\ftnu''=\ftnu$ and $M(\ftnu',\ftnu'')\in \mathbb{Z}$.
\end{proposition}
As a result, the induction and restriction functor can be restricted to be the functors
\begin{align*}
\mathbf{Ind}^{\mathbf{V}}_{\mathbf{V}',\mathbf{V}''}:\mQ_{\nu'}\boxtimes \mQ_{\nu''}\rightarrow \mQ_{\nu},\ \mathbf{Res}^{\mathbf{V}}_{\mathbf{V}',\mathbf{V}''}:\mQ_{\nu}\rightarrow \mQ_{\nu'}\boxtimes \mQ_{\nu''}.
\end{align*}

\subsection{Fourier-Sato transforms}\label{Fourier-Sato transform}
In this section, we recall the definition of Fourier-Sato transform for quivers in \cite{HLSS}. One can also see details in \cite{AHJR} and \cite{MR1074006} for general manifolds.
Assume $i$ is a source  in $\Omega$ , that is, there is no $h\in \Omega$ such that $h''=i$, then there is a natural projection
$$\pi_{i}:\bfEVO \rightarrow  \dot{\mathbf{E}}_{\mathbf{V},i}= \bigoplus_{h\in \Omega,h'\neq i} \Hom (\bV_{h'},\bV_{h''}).$$
The formula $\lambda:\mathbb{G}_{m} \rightarrow \bG_{\bV}, t \mapsto t^{-1} \textrm{Id}_{\bV_{i}}$ defines an action of $\mathbb{G}_{m}$ on the fiber of  $\pi_{i}$, which contracts the fiber. Similarly, let $\Omega'=\mu_{i}(\Omega)$ be the orientation obtained from $\Omega$ by reversing all arrows starting from $i$, then $i$ is a sink in $\Omega'$, that is, there is no $h\in \Omega'$ such that $h'=i$. We can consider the action of $\lambda':\mathbb{G}_{m} \rightarrow \bG_{\bV}, t \mapsto t \textrm{Id}_{\bV_{i}}$ on $\mathbf{E}_{\bV,\Omega'}$, then $\mathbf{E}_{\bV,\Omega'}$ is identified with the dual bundle of $\bfEVO$ over $\dot{\mathbf{E}}_{\mathbf{V},i}$ via the trace form.
\[
\xymatrix{
\bfEVO \ar[dr]_{\pi_{i}} & & \mathbf{E}_{\bV,\Omega'} \ar[dl]^{\pi^{'}_{i}} \\
 & \dot{\mathbf{E}}_{\mathbf{V},i} &
}
\]
Notice that any object in $\mathcal{D}^b_{\bG_{\bV}}(\bfEVO)$ is monodromic with respect to $\lambda$ in the sense of \cite{AHJR}, by \cite[D.4]{HLSS} the Fourier-Sato transform defines a perverse equivalence 
$$ {\rm{Four}}_{i}:  \mD^{b}_{\bG_{\bV}}(\bfEVO) \rightarrow \mD^{b}_{\bG_{\bV}}(\mathbf{E}_{\bV,\Omega'}).$$

Since Fourier-Sato transforms are compatible with  pull-back or push-forward of base change and  morphisms between vector bundles, one can apply the proof of \cite[10.4.5]{MR4337423} for Fourier-Laumon transforms, and prove that ${\rm{Four}}_{i}$ commutes with the induction functor 
$$\mathbf{Ind}^{\mathbf{V}}_{\mathbf{V}',\mathbf{V}''}({\rm{Four}}_{i}(L) \boxtimes {\rm{Four}}_{i}(L')) ={\rm{Four}}_{i}(\mathbf{Ind}^{\mathbf{V}}_{\mathbf{V}',\mathbf{V}''} (L \boxtimes L') ).$$

In general, for any other orientation $\Omega,\Omega'\subset H$ without oriented cycles, we assume that  there is a sequence $(i_{1},i_{2},\cdots ,i_{m})$ of $i \in I$, such that $i_{s}$ is a sink in $\mu_{s-1}\cdots \mu_{1} (\Omega)$ for any $s$,  and $\mu_{m}\cdots \mu_{1} (\Omega)=\Omega'$. Consider the vector bundle $\bfEVO \rightarrow \mathbf{E}_{\bV,\Omega \cap \Omega'}$ and define the corresponding Fourier-Sato  transform ${\rm{Four}}_{\Omega,\Omega'}$. In this case, we say  $\Omega,\Omega'$ are equivalent under mutations.

\begin{lemma}\label{FST}
With the above assumption,	any object in $\mD^{b}_{\bG_{\bV}}(\bfEVO)$ is monodromic with respect to the above vector bundle and the Fourier-Sato  transform is well-defined on $\mD^{b}_{\bG_{\bV}}(\bfEVO)$. 
\end{lemma}
\begin{proof}
	We prove that there exists an one parameter subgroup $\lambda: \mathbb{G}_{m} \rightarrow \bG_{\bV}$ such that $\lambda(\mathbb{G}_{m})$ acts on $\bfEVO$ by contracting the fiber of $\bfEVO \rightarrow \mathbf{E}_{\bV,\Omega \cap \Omega'}$ by induction on the length $m$ of the sequence. When $m=1$, we use the action of $\mathbb{G}_{m}$ for ${\rm{Four}}_{i_{1}}$.  If $\lambda_{k}:\mathbb{G}_{m} \rightarrow \bG_{\bV} $ has been constructed  for $(i_{1},i_{2},\cdots ,i_{k})$, then we can take $\lambda_{k+1}(t)=\lambda_{k}(t) \circ t \textrm{Id}_{\bV_{k+1}}$. 
	
	Let $\Omega_{k}=\mu_{k}\cdots \mu_{1} (\Omega) $, then $\lambda_{k}(t)$ acts on  $x_{h}$ of arrows $h \in \Omega \backslash (\Omega \cap \Omega_{k})$ with weight $1$. Notice that if $h \in \Omega \backslash (\Omega \cap \Omega_{k})$ and $h$ is adjacent to $k+1$, then $\bar{h} \in \Omega_{k}$ must has a target $k+1$, since $k+1$ is a sink in $\Omega_{k}$. In particular, $h$ starts at $k+1$ and $t\textrm{Id}_{\bV_{k+1}}$ acts on such $x_{h}$ by $t^{-1}$. Hence $\lambda_{k+1}(t)$ acts trivially on such $x_{h}$. When $h \in \Omega \backslash (\Omega \cap \Omega_{k})$ is not adjacent to $k+1$, $\lambda_{k+1}(t)$ acts on such $x_{h}$ as $\lambda_{k}(t)$ with weight $1$. If $h \notin \Omega \backslash (\Omega \cap \Omega_{k})$, then $\lambda_{k+1}(t)$ acts on such $x_{h}$ as $t\textrm{Id}_{\bV_{k+1}}$. With the argument above, we can see that $\lambda_{k+1}(t)$ acts by contracting $x_{h}$ with $h \in \Omega \backslash (\Omega \cap \Omega_{k+1}). $
	
	With the action of $\lambda: \mathbb{G}_{m} \rightarrow \bG_{\bV}$, any object in $\mD^{b}_{\bG_{\bV}}(\bfEVO)$ is $\mathbb{G}_{m}$-equivariant, hence is monodromic.
\end{proof}

\begin{lemma}\label{FST1}
	For any acyclic orientation $\Omega$ and $i \in I$, there exists an orientation $\Omega_{i}$ equivalent to $\Omega$ under mutations such that $i$ is a source in $\Omega_{i}$. Similarly, there exists an orientation $\Omega^{i}$ equivalent to $\Omega$ under mutations such that $i$ is a sink in $\Omega^{i}$.
\end{lemma}

\begin{proof}
	We only prove the first statement. Consider the set $P(\Omega,i)$ of paths in $Q$ with target $i$. Since there is no oriented cycle in $\Omega$, there exists a minimal integer $m$ such that all paths in $P(\Omega,i)$ has length $\leqslant m$. If $m=0$, then $i $ is a source in $\Omega$.
	If $m>0$, then we consider the paths in $P(\Omega,i)$ with length equal to $m$ and denote their sources by $i_{1},i_{2},\cdots, i_{t} \in I$.   Then $i_{1},i_{2},\cdots, i_{t}$ must be sources in $\Omega$, otherwise, there exists a path with length $> m$ in $\Omega$. After applying $\mu_{i_{1}}, \cdots ,\mu_{i_{t}}$, we obtain an orientation $\tilde{\Omega}$ such that paths in $P(\tilde{\Omega},i)$ has length $\leqslant m-1$. By induction on $m$, we  get the sequence desired.
\end{proof}

We remark that Lusztig used the Fourier-Deligne transform in \cite{MR1088333} and \cite{MR1227098}, because he considered $\overline{\mathbb{F}}_q$-varieties and $\mathbb{C}$-sheaves. In this paper, we consider $\mathbb{C}$-varieties and $\mathbb{C}$-sheaves, and so the Fourier-Deligne transform should be replaced by the Fourier-Sato transform. The statements in \cite{MR1088333} and \cite{MR1227098} involving the Fourier-Deligne transform still hold after replacing by the Fourier-Sato transform.

\subsection{Lusztig's key lemmas}
By Lemma \ref{FST} and \ref{FST1}, we can always apply the Fourier-Sato transform and assume $i$ is a source or sink.
\subsubsection{Analysis at sink}
Let $i$ be a sink for the orientation $\Omega$. For any $\nu\in \mathbb{N}I, 0\leqslant t\leqslant \nu_i$, we define ${^{t}\bfEVO^{i}}\subset \bfEVO$ to be the locally closed subset consisting of $x$ such that
$$\mathrm{codim}_{\bV_i}(\sum_{h\in \Omega, h''=i}\mathrm{Im} (x_h:\bV_{h'}\rightarrow \bV_i))=t,$$
then all these ${^{t}\bfEVO^{i}}, 0\leqslant t\leqslant \nu_i$ form a $\bG_{\bV}$-invariant stratification of $\bfEVO$. For any simple perverse sheaf $L\in \mathcal{D}^b_{\mathbf{G}_{\bV}}(\bfEVO)$, there exists a unique integer $t_i(L)$ such that ${^{t_i(L)}\bfEVO^{i}}\cap \mathrm{supp}(L)$ is open dense in $\mathrm{supp}(L)$. 

\begin{lemma}[{\cite[Lemma 6.4]{MR1088333}}]\label{lkey}
For any $0 \leqslant t \leqslant \nu_{i}$, let $\nu'=ti$, then $\mathbf{E}_{\mathbf{V}',\Omega}=\{0\}$ is a point. For any $\nu''\in \mathbb{N}I$ such that $\nu=\nu'+\nu''$, we identify $\mathbf{E}_{\mathbf{V}',\Omega}\times \mathbf{E}_{\mathbf{V}'',\Omega}$ with $\mathbf{E}_{\mathbf{V}'',\Omega}$. \\
{{\rm{(a)}}} For any $L\in \mP_{\nu}$ with $t_{i}(L)=t$, we have 
$$\mathbf{Res}^{\mathbf{V}}_{\mathbf{V'},\mathbf{V}''}(L)= K \oplus \bigoplus_{t_{i}(K')>0} K'[d']^{\oplus n(K',d')},$$ 
where $K, K'\in \mP_{\nu''},d'\in \mathbb{Z}, n(K',d')\in \mathbb{N}$ and $K$ is the unique direct summand satisfying $t_{i}(K)=0$.\\ 
{{\rm{(b)}}} For any $K \in \mP_{\nu''}$ with $t_{i}(K)=0$, we have
$$\mathbf{Ind}^{\mathbf{V}}_{\mathbf{V}',\mathbf{V}''}(\mathbb{C} \boxtimes K) = L \oplus \bigoplus_{t_{i}(L')>t} L'[d']^{\oplus n(L',d')}, $$ 
where $L, L'\in \mP_{\nu},d'\in \mathbb{Z}, n(L',d')\in \mathbb{N}$ and $L$ is the unique direct summand satisfying $t_{i}(L)=t$.\\
{{\rm{(c)}}} The maps $L\mapsto K$ and $K\mapsto L$ define a bijection  $\pi_{i,t}:\{ L \in \mP_{\nu}| t_{i}(L)=t \} \rightarrow \{ K \in \mP_{\nu''}| t_{i}(K)=0 \} $.
\end{lemma}

\subsubsection{Analysis at source}\label{Analysis at source}
Let $i$ be a source for the orientation $\Omega$. For any $\nu\in \mathbb{N}I, 0\leqslant t\leqslant \nu_i$, we define $\bfEVO^{i,t}\subset \bfEVO$ to be the locally closed subset consisting of $x$ such that
$${\mathrm{dim}}\,(\bigcap_{h\in \Omega,h'=i} \mathrm{Ker}(x_h:\bV_i\rightarrow \bV_{h''}))=t,$$
then all these $\bfEVO^{i,t},0\leqslant t\leqslant \nu_i$ form a $\bG_{\bV}$-invariant stratification of $\bfEVO$. For any simple perverse sheaf $L\in \mathcal{D}^b_{\mathbf{G}_{\bV}}(\bfEVO)$, there exists a unique integer $t_i^{\ast}(L)$ such that $\bfEVO^{i,t_i^{\ast}(L)}\cap \mathrm{supp}(L)$ is open dense in $\mathrm{supp}(L)$. The following lemma is dual to Lemma \ref{lkey}.

\begin{lemma} \label{rkey}
For any $0 \leqslant t \leqslant \nu_{i}$, let $\nu''=ti$, then $\mathbf{E}_{\mathbf{V}'',\Omega}=\{0\}$ is a point. For any $\nu'\in \mathbb{N}I$ such that $\nu=\nu'+\nu''$, we identify $\mathbf{E}_{\mathbf{V}',\Omega}\times \mathbf{E}_{\mathbf{V}'',\Omega}$ with $\mathbf{E}_{\mathbf{V}',\Omega}$. \\ 
{{\rm{(a)}}} For any $L\in \mP_{\nu}$ with $t_{i}^{\ast}(L)=t$, we have 
$$\mathbf{Res}^{\mathbf{V}}_{\mathbf{V'},\mathbf{V}''}(L)= K \oplus \bigoplus_{t_{i}^{\ast}(K')>0} K'[d']^{\oplus n(K',d')},$$ 
where $K, K'\in \mP_{\nu'},d'\in \mathbb{Z}, n(K',d')\in \mathbb{N}$ and $K$ is the unique direct summand satisfying $t_{i}^{\ast}(K)=0$.\\ 
{{\rm{(b)}}} For any $K \in \mP_{\nu'}$ with $t_{i}^{\ast}(K)=0$, we have
$$\mathbf{Ind}^{\mathbf{V}}_{\mathbf{V}',\mathbf{V}''}(\mathbb{C} \boxtimes K) = L \oplus \bigoplus_{t_{i}^{\ast}(L')>t} L'[d']^{\oplus n(L',d')}, $$ 
where $L, L'\in \mP_{\nu},d'\in \mathbb{Z}, n(L',d')\in \mathbb{N}$ and $L$ is the unique direct summand satisfying $t_{i}^{\ast}(L)=t$.\\
{{\rm{(b)}}} The maps $L\mapsto K$ and $K\mapsto L$ define a bijection   $\pi^{\ast}_{i,t}:\{ L \in \mP_{\nu}| t^{\ast}_{i}(L)=t \} \rightarrow \{ K \in \mP_{\nu'}| t^{\ast}_{i}(K)=0 \} $.
\end{lemma}

\subsection{Categorification of $\Uvm$ and canonical bases}\label{Categorification theorem}
Let $\mK(\mQ_{\nu})$ be the Grothendieck group of $\mQ_{\nu}$, that is, a $\mathbb{Z}$-module spanned by the isomorphism classes $[L]$ of objects $L\in \mathcal{Q}_\nu$ subject to the relations
$$[L \oplus L']=[L]+[L'].$$
We equip it with a $\mathcal{A}$-module structure by  
$$v[L]=[L[1]].$$
We denote by the direct sum 
$$\mK=\bigoplus_{\nu \in \mathbb{N}I} \mK(\mQ_{\nu}).$$
Then all induction functors and restriction functors induce two $\mathcal{A}$-linear maps
$\mK\otimes \mK\rightarrow \mK,\ \mK\rightarrow \mK\otimes \mK$.

\begin{theorem}\label{Lumain}
All induction functors induce a multiplication, and restriction functors induce a comultiplication on $\mK$ such that $\mK$ is a bialgebra over $\mathcal{A}$. Moreover, there is a bialgebra isomorphism
$$\chi:\mK\rightarrow \Uvm$$
sending $[\mathbb{C}|_{\bfEVO}]$ to $F_i^{(n)}$, where $|\mathbf{V}|=ni$ for $i\in I, n\in \mathbb{N}$.
\end{theorem}
The $\mathcal{A}$-modules $\mK$ has a $\mathcal{A}$-basis $\bigsqcup_{\nu\in \mathbb{N}I}\{[L]\in \mK(\mQ_{\nu})\mid L\in \mathcal{P}_\nu\}$ given by simple Lusztig sheaves. Its image 
$$\mathcal{B}=\bigsqcup_{\nu\in \mathbb{N}I}\mathcal{B}_\nu=\bigsqcup_{\nu\in \mathbb{N}I}\{\chi([L])\in \Uvm\mid L\in \mathcal{P}_\nu\}$$
under the isomorphism $\chi$, is a $\mathcal{A}$-basis of $\Uvm$ which is defined to be the canonical basis of $\Uvm$, see \cite{MR1088333}. Moreover, for any dominant weight $\lambda$, by Lemma \ref{rkey}, Lusztig proved that 
$$\mathcal{B}_\lambda=\{b\eta\in {_\mathcal{A}}L_v(\lambda)\mid b\in \mathcal{B}\}\setminus \{0\}$$
is a $\mathcal{A}$-basis of ${_\mathcal{A}}L_v(\lambda)$, where $\eta$ is the highest weight vector of ${_\mathcal{A}}L_v(\lambda)$, and defined it to be the canonical basis of ${_\mathcal{A}}L_v(\lambda)$, see \cite{MR1088333}.

\subsection{Refine string order}\label{Refine string order}
For any $\nu\in \mathbb{N}I, L\in \mathcal{P}_\nu$ and $i\in I$, we define $0\leqslant \epsilon_{i}(L)\leqslant \nu_i$ to be the maximum integer $t$ such that $L$ is a direct summand of $L_{\ftnu}$ up to shift, where $\ftnu=(ti,...)\in \mathcal{V}_\nu$ is a sequence starting by $ti$. Dually, we define $0\leqslant \epsilon^\ast_{i}(L)\leqslant \nu_i$ to be the maximum integer $t$ such that $L$ is a direct summand of $L_{\ftnu}$ up to shift, where $\ftnu=(...,ti)\in \mathcal{V}_\nu$ is a sequence ending by $ti$. If $i$ is a sink for the orientation $\Omega$, we have $t_{i}(L)=\epsilon_{i}(L)$ by \cite[Proposition 6.6]{MR1088333}. Dually, if $i$ is a source for the orientation $\Omega$, we have $t^\ast_{i}(L)=\epsilon^\ast_{i}(L)$. If $\nu\not= 0$, there exists $i\in I$ such that $\epsilon_{i}(L)>0$ by \cite[Lemma 7.1]{MR1088333}. Dually, if $\nu\not= 0$, there exists $i\in I$ such that $\epsilon^\ast_{i}(L)>0$. Then we define two maps $\epsilon_{i},\epsilon^\ast_{i}:\mathcal{B}\rightarrow \mathbb{N}$ by 
$$\epsilon_{i}(\chi([L]))=\epsilon_{i}(L),\,\epsilon^\ast_{i}(\chi([L]))=\epsilon^\ast_{i}(L).$$

We fix an order $\prec$ on $I=\{i_1,...,i_n\}$ and assume that $i_{1} \prec i_{2} \prec \cdots \prec i_{n}$. Then for any $\chi([L]) \in \mathcal{B}$, where $L\in \mathcal{P}_\nu,\nu\in \mathbb{N}I$, we inductively define a sequence $s^{\prec}(\chi([L]))$ of $\mathbb{N}\times I$ with respect to $\prec$ as follows,\\
$\bullet$ if $\nu=0$, we define $s^{\prec}(\chi([L]))$ to be the empty sequence;\\
$\bullet$ if $\nu\not=0$, let $i\in I$ be the minimal vertex such that $\epsilon_{i}(L)>0$, then $s^{\prec}(\chi\mathrm{Four}_{\Omega_i,\Omega}\pi_{i,\epsilon_{i}(L)}\mathrm{Four}_{\Omega,\Omega_i}(L))$ has been defined by inductive hypothesis, where $\Omega_i$ is an orientation such that $i$ is a sink for it, $\mathrm{Four}_{\Omega_i,\Omega},\mathrm{Four}_{\Omega,\Omega_i}$ are the Fourier-Sato (or Fourier-Deligne) transforms and $\pi_{i,\epsilon_{i}(L)}$ is the bijection in Lemma \ref{lkey}, and we define 
$$s^{\prec}(\chi([L]))=((i,\epsilon_{i}(L)),s^{\prec}(\chi\mathrm{Four}_{\Omega_i,\Omega}\pi_{i,\epsilon_{i}(L)}\mathrm{Four}_{\Omega,\Omega_i}(L))).$$

\begin{definition}
For any $\nu\in \mathbb{N}I$, we define a partial order on $\mathcal{B}_\nu$ as follows, for any $b,b'\in \mathcal{B}_\nu$ such that  
$$s^{\prec}(b)=((i_{n_{1}},m_{1} ), (i_{n_{2}},m_{2}),\cdots,(i_{n_{k}},m_{k}) ),\ s^{\prec}(b')=((i_{n'_{1}},m'_{1} ), (i_{n'_{2}},m'_{2}),\cdots,(i_{n'_{l}},m'_{l})),$$ 
we define $b'\prec b$, if there exists $r$ such that $(i_{n_{t}},m_{t})=(i_{n'_{t}},m'_{t})$ for $1 \leqslant t< r$ and $(i_{n'_{r}},m'_{r})\prec (i_{n_{r}},m_{r})$ with respect to the lexicographical order. We define a partial order on $\mathcal{B}$ by taking disjoint union.
\end{definition}

We remark that above partial order on $\mathcal{B}$ is defined in \cite{fang2022correspondence} and is a refinement of the string order defined in \cite[Section 2.5]{baumann2011canonical}.

Notice that for any dominant weight $\lambda$, the partial order on $\mathcal{B}$ induces a partial order on $\mathcal{B}_\lambda$ by 
$$b\eta\prec b'\eta \Leftrightarrow b\prec b'$$
for any $b\eta,b'\eta\in \mathcal{B}_\lambda$, see \cite{fang2023lusztig}.

\section{Geometric realization of canonical basis of integrable highest module}

In this section, we recall the results in \cite{fang2023lusztig}, \cite{fang2023tensor} and generalize the constructions over $\mathbb{C}$. 
\subsection{Lusztig sheaves on the framed quivers and the localizations}\label{Lusztig sheaves on the framed quivers and the localizations}
Recall that in subsection \ref{Lusztig sheaf}, we associate the symmetric generalized matrix $C$ with a finite graph $(I,H)$. We define the framed graph $(I^{(1)}, H^{(1)})$ of $(I,H)$ by
$$I^{(1)}= I \cup I^{1},\ H^{(1)} =H \cup \{ i \rightarrow i^{1},i^{1} \rightarrow i\mid i\in I \},$$
where $I^{1}=\{i^{1}\mid i \in I \}$ is a copy of $I$, with an involution of $H^{(1)}$ obtained from the involution of $H$ and exchanging $i \rightarrow i^{1},i^{1} \rightarrow i$ for $i\in I$. Moreover, the two-framed graph $(I^{(2)}, H^{(2)})$ of $(I,H)$ by
$$I^{(2)}= I \cup I^{1} \cup I^{2},\ H^{(2)} =H^{(1)} \cup \{ i \rightarrow i^{2},i^{2} \rightarrow i\mid i\in I \},$$
where $I^{2}=\{i^{2}\mid i \in I \}$ is another copy of $I$, with an involution of $H^{(2)}$ obtained from the involution of $H^{(1)}$ and exchanging $i \rightarrow i^{2},i^{2} \rightarrow i$ for $i\in I$. Then for an orientation $\Omega\subset H$, we define 
$$\Omega^{(1)}=\Omega\cup \{i \rightarrow i^{1}\mid i \in I \},\ \Omega^{(2)}=\Omega^{(1)}\cup \{i \rightarrow i^{2}\mid i \in I \},$$
and define the framed quiver, the two-framed quiver of the quiver $Q=(I,H,\Omega)$ to be 
$$Q^{(1)} =(I^{(1)},H^{(1)},\Omega^{(1)}),\ Q^{(2)} =(I^{(2)},H^{(2)},\Omega^{(2)})$$
respectively.

Given two dominant weights $\lambda_{1},\lambda_{2}$, we fix an $I^{1}$-graded $\mathbb{C}$-vector space $\mathbf{W}^{1}$ and an $I^{2}$-graded $\mathbb{C}$-vector space $\mathbf{W}^{2}$ such that 
$$\dim \mathbf{W}^{1}_{i^{1}}= \langle i ,\lambda_{1} \rangle,\ \dim \mathbf{W}^{2}_{i^{2}}= \langle i ,\lambda_{2} \rangle$$ 
for any $i\in I$, where $\langle-,-\rangle$ is the paring on $\mathbb{Z}I$ induced by the symmetric generalized Cartan matrix $C$, see section \ref{introduction}. We denote by $|\mathbf{W}^{1}|=\omega^1\in \mathbb{N}I^1,|\mathbf{W}^{2}|=\omega^2\in \mathbb{N}I^2$ the dimension vectors of $\mathbf{W}^{1},\mathbf{W}^{2}$ respectively.

For any $\nu\in \mathbb{N}I$, we fix a $I$-graded $\mathbb{C}$-vector space $\mathbf{V}$ such that $|\mathbf{V}|=\nu$, then $\mathbf{V}\oplus \mathbf{W}^1$ is a $I^{(1)}$-graded $\mathbb{C}$-vector space such that $|\mathbf{V}\oplus \mathbf{W}^1|=\nu+\omega^1$, and $\mathbf{V}\oplus \mathbf{W}^1\oplus \mathbf{W}^2$ is a $I^{(2)}$-graded $\mathbb{C}$-vector space such that $|\mathbf{V}\oplus \mathbf{W}^1\oplus \mathbf{W}^2|=\nu+\omega^1+\omega^2$. Regarding the framed quiver $Q^{(1)}$ and the two-framed quiver $Q^{(2)}$ as kinds of quiver used in subsection \ref{Lusztig sheaf}, there are affine spaces
\begin{align*}
\bfEVOf &= \bfEVO \oplus \bigoplus_{i \in I} \Hom (\mathbf{V}_{i},\mathbf{W}^{1}_{i^1}),\\
\bfEVOff &= \bfEVOf \oplus \bigoplus_{i \in I} \Hom (\mathbf{V}_{i},\mathbf{W}^{2}_{i^2})
\end{align*}
which have $\bG_{\bV\oplus \mathbf{W}^1}$-action and $\bG_{\bV\oplus \mathbf{W}^1\oplus\mathbf{W}^2}$-action respectively. 

We also define the torus 
$$\mathbb{T}=\mathbb{G}_{m}\times...\times \mathbb{G}_{m}$$
of rank $\sharp I$ such that it acts on $\bfEVOf$ and $\bfEVOff$ respectively by
$$t=(t_{i})_{i\in I} \mapsto (t_{i} \textrm{Id}_{\mathbf{W}^1} )_{i \in I} \in \bG_{\bV\oplus \mathbf{W}^1},$$
$$t=(t_{i})_{i\in I} \mapsto (t_{i} \textrm{Id}_{\mathbf{W}^1 \oplus \mathbf{W}^{2}} )_{i \in I} \in \bG_{\bV\oplus \mathbf{W}^1\oplus\mathbf{W}^2}.$$
Throughout this paper, we only consider the $\bG_{\bV} \times \mathbb{T}$-actions on $\bfEVOf$ and $\bfEVOff$, and the $\bG_{\bV} \times \mathbb{T}$ equivariant derived categories of constructible sheaves on $\bfEVOf$ and $\bfEVOff$. Here we need the $\mathbb{T}$-action, because we need to use the Fourier-Sato transforms in subsection \ref{Fourier-Sato transform} to reverse arrows of the form $i \rightarrow i^{1}, i \rightarrow i^{2}$ or $i \leftarrow i^{1}, i \leftarrow i^{2}$. We denote these mutations reversing framed arrows respectively by $\mu_{i^{1}}$ and  $\mu_{i^{1}}\mu_{i^{2}}$  for framed and 2-framed quivers.  If two orientations $\Omega^{(1)}$ and $\Omega^{(1),'}$ can be obtained by a sequence of $\mu_{i}$ or  $\mu_{i^{1}}$, we say they are equivalent under mutation. If two orientations $\Omega^{(2)}$ and $\Omega^{(2),'}$ can be obtained by a sequence of $\mu_{i}$ and  $\mu_{i^{1}}\mu_{i^{2}}$, we say they are equivalent under mutation. By  Lemma \ref{FST}, we can define Fourier-Sato transform for any equivalent orientations.

There are corresponding categories $\mathcal{Q}_{\nu+\omega^1},\mathcal{Q}_{\nu+\omega^1+\omega^2}$ of Lusztig sheaves on them respectively, see subsection \ref{Lusztig sheaf}. Suppose $I=\{ i_{1} ,i_{2}, \cdots , i_{n}\}$, we fix two sequences
$$ \boldsymbol{\omega}^{1}=(\omega^{1}_{i_{1}} i_{1}, \omega^{1}_{i_{2}} i_{2},\cdots \omega^{1}_{i_{n}} i_{n})\in \mathcal{V}_{\omega^1},\ \boldsymbol{\omega}^{2}=(\omega^{2}_{i_{1}} i_{1}, \omega^{2}_{i_{2}} i_{2},\cdots \omega^{2}_{i_{n}} i_{n})\in \mathcal{V}_{\omega^2}.$$
Then for any $\ftnu\in \mathcal{V}_\nu$, we have $\ftnu \boldsymbol{\omega}^{1}\in \mathcal{V}_{\nu+\omega^1}$ and the semisimple complex $L_{\ftnu \boldsymbol{\omega}^{1}}\in \mathcal{Q}_{\nu+\omega^1}$. Similarly, for any $\ftnu'\in \mathcal{V}_{\nu'},\ftnu''\in \mathcal{V}_{\nu''}$ such that $\nu=\nu'+\nu''$, we have $\ftnu' \boldsymbol{\omega}^{1}\ftnu'' \boldsymbol{\omega}^{2}\in \mathcal{V}_{\nu+\omega^1+\omega^2}$ and the semisimple complex $L_{\ftnu' \boldsymbol{\omega}^{1}\ftnu'' \boldsymbol{\omega}^{2}}\in \mathcal{Q}_{\nu+\omega^1+\omega^2}$.

\begin{definition}\label{small Lusztig shevaes categories}
{\rm (a)} We define $\mP_{\nu,\omega^{1}}$ to be the full subcategory of $\mP_{\nu+\omega^{1}}$ consisting of simple perverse sheaves $L$ such that $L[n]$ is a direct summand of $L_{\ftnu \boldsymbol{\omega}^{1}}$ for some $n\in\mathbb{Z},\ftnu \in \mathcal{V}_{\nu}$, and define $\mQ_{\nu,\omega^{1}}$ to be the full subcategory of $\mathcal{Q}_{\nu+\omega^1}$ consisting of direct sums of complexes of the form $L[n]$ for any $L\in \mP_{\nu,\omega^{1}},n\in \mathbb{Z}$.\\
{\rm (b)} We define $\mP_{\nu,\omega^1,\omega^2}$ to be the full subcategory of $\mP_{\nu+\omega^{1}+\omega^2}$ consisting of simple perverse sheaves $L$ such that $L[n]$ is a direct summand of $L_{\ftnu' \boldsymbol{\omega}^{1}\ftnu'' \boldsymbol{\omega}^{2}}$ for some $n\in\mathbb{Z},\ftnu'\in \mathcal{V}_{\nu'},\ftnu''\in \mathcal{V}_{\nu''}$ satisfying $\nu=\nu'+\nu''$, and define $\mQ_{\nu,\omega^1,\omega^2}$ to be the full subcategory of $\mathcal{Q}_{\nu+\omega^1+\omega^2}$ consisting of direct sums of complexes of the form $L[n]$ for any $L\in \mP_{\nu,\omega^{1},\omega^2},n\in \mathbb{Z}$.
\end{definition}

Recall that in subsection \ref{Categorification theorem}, we define the $\mathcal{A}$-module $\mK(\mQ_\nu)$ associated to $\mQ_\nu$ for the quiver $Q$. Similarly, for the framed quiver $Q^{(1)}$ and the two-framed quiver $Q^{(2)}$, there are corresponding $\mathcal{A}$-modules $\mK(\mQ_{\nu+\omega^1})$ associated to $\mQ_{\nu+\omega^1}$ and $\mK(\mQ_{\nu+\omega^1+\omega^2})$ associated to $\mQ_{\nu+\omega^1+\omega^2}$. Let $\mK(\mQ_{\nu,\omega^{1}})\subset \mK(\mQ_{\nu+\omega^1})$ be the $\mathcal{A}$-submodule spanned by $[L]$ for $L\in \mathcal{Q}_{\nu,\omega^1}$, and let $\mK(\mQ_{\nu,\omega^{1},\omega^2})\subset \mK(\mQ_{\nu+\omega^1+\omega^2})$ be the $\mathcal{A}$-submodule spanned by $[L]$ for $L\in \mathcal{Q}_{\nu,\omega^1,\omega^2}$. We denote the direct sums by
$$\mK(\omega^1)=\bigoplus_{\nu \in \mathbb{N}I} \mK(\mQ_{\nu,\omega^{1}}),\ \mK(\omega^1,\omega^2)=\bigoplus\limits_{\nu \in \mathbb{N}I} \mK(\mQ_{\nu,\omega^1,\omega^2}).$$

\begin{proposition}[{\cite[Lemma 4.3, Corollary 4.7]{fang2023tensor}}]\label{span}
As $\mathcal{A}$-modules, $\mK(\mQ_{\nu,\omega^{1}})$ and $\mK(\mQ_{\nu,\omega^{1},\omega^2})$ are spanned by the images of those $L_{\ftnu \boldsymbol{\omega}^{1}}$ and $L_{\ftnu' \boldsymbol{\omega}^{1}\ftnu''\boldsymbol{\omega}^{2}}$ appeared in Definition \ref{small Lusztig shevaes categories} respectively, that is, 
\begin{align*}
 \mK(\mQ_{\nu,\omega^{1}}) &={\rm{span}}_{\mathcal{A}} \{[L_{\ftnu \boldsymbol{\omega}^{1}}]| \ftnu \in \mathcal{V}_{\nu}  \},\\
\mK(\mQ_{\nu,\omega^1,\omega^2}) &={\rm{span}}_{\mathcal{A}} \{[L_{\ftnu' \boldsymbol{\omega}^{1}\ftnu''\boldsymbol{\omega}^{2}}]| \ftnu' \in \mathcal{V}_{\nu'},\ftnu''\in \mathcal{V}_{\nu''},\nu'+\nu''=\nu  \}. 
\end{align*}
\end{proposition} 
\begin{proof}
 We only need to replace the Fourier-Deligne transform by the Fourier-Sato transform constructed by the above lemma in the proof of \cite[Lemma 4.3, Corollary 4.7]{fang2023tensor}.
\end{proof}

Let $i\in I$ be a source for the orientation $\Omega$ of the quiver $Q$, then it is a source for the orientation $\Omega^{(1)}$ of the framed quiver $Q^{(1)}$, and a source for the orientation $\Omega^{(2)}$ of the two-framed quiver $Q^{(2)}$. Recall that in subsection \ref{Analysis at source}, we define the locally closed subsets $\bfEVO^{i,t}\subset \bfEVO$ for $0\leqslant t\leqslant \nu_i$. Similarly, we can define the corresponding locally closed subsets $\bfEVOf^{i,t}\subset \bfEVOf$ and $\bfEVOff^{i,t}\subset \bfEVOff$ for $0\leqslant t\leqslant \nu_i$. In particular, we have 
\begin{align}\label{EVWi0}
\bfEVOf&=\bfEVOf^{i,0} \cup \bfEVOf^{i,\geqslant 1},\\
\bfEVOff&=\bfEVOff^{i,0} \cup \bfEVOff^{i,\geqslant 1},
\end{align}
where $\bfEVOf^{i,0}\subset \bfEVOf$ and $\bfEVOff^{i,0}\subset \bfEVOff$ are open subsets with the complements $\bfEVOf^{i,\geqslant 1}$ and $\bfEVOff^{i,\geqslant 1}$ respectively.

\begin{definition}\label{def localization}
{\rm (a)} Let $i\in I$ be a source for the orientation $\Omega$. For the framed quiver $Q^{(1)}$, we define $\mathcal{N}_{\nu,i}$ to be the thick subcategory of $\mathcal{D}^b_{\bG_{\bV}\times \mathbb{T}}(\bfEVOf)$ consisting of $L$ such that $\mathrm{supp}(L)\subset \bfEVOf^{i,\geqslant 1}$. Similarly, for the two-framed quiver $Q^{(2)}$, we define $\mathcal{N}_{\nu,i}$ (the same notation) to be the thick subcategory of $\mathcal{D}^b_{\bG_{\bV}\times \mathbb{T}}(\bfEVOff)$ consisting of $L$ such that $\mathrm{supp}(L)\subset \bfEVOff^{i,\geqslant 1}$.\\
{\rm (b)} For any $i\in I$, we fix an orientation $\Omega_i$ such that $i$ is a source, then there are thick subcategories $\mathcal{N}_{\nu,i}$ of $\mathcal{D}^b_{\bG_{\bV}\times \mathbb{T}}(\mathbf{E}_{\mathbf{V}\oplus\mathbf{W}^1,\Omega_i^{(1)}})$ and of $\mathcal{D}^b_{\bG_{\bV}\times \mathbb{T}}(\mathbf{E}_{\mathbf{V}\oplus\mathbf{W}^1\oplus\mathbf{W}^2,\Omega_i^{(2)}})$ by {\rm (a)}. We define $\mathcal{N}_{\nu}$ to be the thick subcategory of $\mathcal{D}^b_{\bG_{\bV}\times \mathbb{T}}(\bfEVOf)$ generated by ${\rm{Four}}_{\Omega_{i}^{(1)},\Omega^{(1)}}(\mathcal{N}_{\nu,i})$ for any $i\in I$, where ${\rm{Four}}_{\Omega_{i}^{(1)},\Omega^{(1)}}$ is Fourier-Sato transform for the framed quiver $Q^{(1)}$. Similarly, we define $\mathcal{N}_{\nu}$ to be the thick subcategory of $\mathcal{D}^b_{\bG_{\bV}\times \mathbb{T}}(\bfEVOff)$ generated by ${\rm{Four}}_{\Omega_{i}^{(2)},\Omega^{(2)}}(\mathcal{N}_{\nu,i})$ for any $i\in I$, where ${\rm{Four}}_{\Omega_{i}^{(2)},\Omega^{(2)}}$ is Fourier-Sato transform for the two-framed quiver $Q^{(2)}$.\\
{\rm (c)} Let $\mathcal{D}^b_{\bG_{\bV}\times \mathbb{T}}(\bfEVOf)/\mathcal{N}_{\nu}$ and $\mathcal{D}^b_{\bG_{\bV}\times \mathbb{T}}(\bfEVOff)/\mathcal{N}_{\nu}$ be localizations of the triangulated categories with respect to their thick subcategories. Note that they have the same objects as $\mathcal{D}^b_{\bG_{\bV}\times \mathbb{T}}(\bfEVOf)$ and $\mathcal{D}^b_{\bG_{\bV}\times \mathbb{T}}(\bfEVOff)$ respectively. We define $\mQ_{\nu,\omega^{1}}/\mathcal{N}_{\nu}$ to be full subcategory of $\mathcal{D}^b_{\bG_{\bV}\times \mathbb{T}}(\bfEVOf)/\mathcal{N}_{\nu}$ consisting of 
$L$ such that $L$ is isomorphic to some $L'\in \mQ_{\nu,\omega^{1}}$ in $\mathcal{D}^b_{\bG_{\bV}\times \mathbb{T}}(\bfEVOf)/\mathcal{N}_{\nu}$. Similarly, we define $\mQ_{\nu,\omega^{1},\omega^2}/\mathcal{N}_{\nu}$ to be full subcategory of $\mathcal{D}^b_{\bG_{\bV}\times \mathbb{T}}(\bfEVOff)/\mathcal{N}_{\nu}$ consisting of $L$ such that $L$ is isomorphic to some $L'\in \mQ_{\nu,\omega^{1},\omega^2}$ in $\mathcal{D}^b_{\bG_{\bV}\times \mathbb{T}}(\bfEVOff)/\mathcal{N}_{\nu}$. \\
{\rm (d)} We denote the Grothendieck groups of $\mQ_{\nu,\omega^{1}}/\mathcal{N}_{\nu}$ and $\mQ_{\nu,\omega^1,\omega^2}/\mathcal{N}_{\nu}$ by $\mathcal{L}(\nu,\omega^{1})$ and $\mathcal{L}(\nu,\omega^1,\omega^2)$ respectively. Note that they are quotients of $\mK(\mQ_{\nu,\omega^{1}})$ and $\mK(\mQ_{\nu,\omega^{1},\omega^2})$ respectively. We denote the direct sums by 
$$\mathcal{L}(\omega^{1})=\bigoplus_{\nu \in \mathbb{N}I}\mathcal{L}(\nu,\omega^{1}),\ \mathcal{L}(\omega^1,\omega)=\bigoplus_{\nu \in \mathbb{N}I}\mathcal{L}(\nu,\omega^1,\omega^2).$$
\end{definition}

\subsection{Functor between localization categories}
\subsubsection{The functor $\mathcal{F}^{(r)}_{i}$}
For any $i\in I,r\in \mathbb{N}$ and $\nu=ri+\nu''\in\mathbb{N}I$, we denote by $\nu'=ri$ and identify $$\mathbf{E}_{\mathbf{V}',\Omega}=\mathbf{E}_{\mathbf{V}'\oplus 0,\Omega^{(1)}}=\mathbf{E}_{\mathbf{V}'\oplus 0\oplus0,\Omega^{(2)}}=\{0\}$$
Then we use Lusztig's induction functors, see subsection \ref{Induction and restriction functors}, for the framed quiver $Q^{(1)}$ and the two-framed quiver $Q^{(2)}$ to define the following functors
\begin{align*}
\mathcal{F}^{(r)}_{i}&=\mathbf{Ind}^{\mathbf{V}\oplus \mathbf{W}^{1}}_{\mathbf{V}',\mathbf{V}''\oplus \mathbf{W}^{1}}( \mathbb{C}|_{\mathbf{E}_{\mathbf{V}',\Omega}}\boxtimes-):\mathcal{D}^b_{\bG_{\bV''}\times \mathbb{T}}(\mathbf{E}_{\mathbf{V}''\oplus\mathbf{W}^1,\Omega^{(1)}})\rightarrow\mathcal{D}^b_{\bG_{\bV}\times \mathbb{T}}(\bfEVOf),\\
\mathcal{F}^{(r)}_{i}&=\mathbf{Ind}^{\mathbf{V}\oplus \mathbf{W}^{1}\oplus\mathbf{W}^2}_{\mathbf{V}',\mathbf{V}''\oplus \mathbf{W}^{1}\oplus\mathbf{W}^2}( \mathbb{C}|_{\mathbf{E}_{\mathbf{V}',\Omega}}\boxtimes-):\mathcal{D}^b_{\bG_{\bV''}\times \mathbb{T}}(\mathbf{E}_{\mathbf{V}''\oplus\mathbf{W}^1\oplus\mathbf{W}^2,\Omega^{(2)}})\rightarrow\mathcal{D}^b_{\bG_{\bV}\times \mathbb{T}}(\bfEVOff).
\end{align*}
For simplicity, we also denote $\mathcal{F}^{(1)}_{i}$ by $\mathcal{F}_{i}$.

\subsubsection{The functor $\mathcal{E}^{(r)}_{i}$}  
Let $i\in I$ be a source for the orientation $\Omega$. For any $r\in \mathbb{N}$ and $\nu=ri+\nu''\in\mathbb{N}I$, consider the following diagram.
$$\xymatrix{
   		\bfEVOf
   		&
   		& \bfEVOfpp \\
   		\bfEVOf^{i,0} \ar[d]_{\phi_{\mathbf{V},i}} \ar[u]^{j_{\mathbf{V},i}}
   		&
   		& \bfEVOfpp^{i,0} \ar[d]^{\phi_{\mathbf{V}'',i}} \ar[u]_{j_{\mathbf{V}'',i}} \\
   		\bfEVOfid \times \mathbf{Grass}(\nu_i, \tilde{\nu}_{i})
   		& \bfEVOfid  \times \mathbf{Flag}(\nu''_{i},\nu_{i},\tilde{\nu}_{i}) \ar[r]^{q_{2}} \ar[l]_{q_{1}}
   		&  \bfEVOfidpp  \times \mathbf{Grass}(\nu''_{i}, \tilde{\nu}_{i}),
   	}$$
where $\bfEVOf^{i,0},\bfEVOfpp^{i,0}$ are defined in (\ref{EVWi0}), and $j_{\mathbf{V},i},j_{\mathbf{V}'',i}$ are the open embeddings;
\begin{align*}
\bfEVOfid &=\bigoplus_{h \in \Omega, h'\neq i} \mathbf{Hom}(\mathbf{V}_{h'},\mathbf{V}_{h''}) \oplus \bigoplus_{j\in I,j\not=i} \mathbf{Hom}(\mathbf{V}_{j},\mathbf{W}^1_{j^1}),\\
\bfEVOfidpp &=\bigoplus_{h \in \Omega, h'\neq i} \mathbf{Hom}(\mathbf{V}''_{h'},\mathbf{V}''_{h''}) \oplus \bigoplus_{j\in I,j\not=i} \mathbf{Hom}(\mathbf{V}''_{j},\mathbf{W}^1_{j^1}),
\end{align*}
$\mathbf{Grass}(\nu_i, \tilde{\nu}_{i})$ and $\mathbf{Grass}(\nu''_{i}, \tilde{\nu}_{i})$ are the Grassmannian varieties consisting of $\nu_{i}$-dimensional and $\nu''_i$-dimensional subspaces respectively, of 
$$\bigoplus_{h\in \Omega, h'=i}\mathbf{V}_{h''}\oplus \mathbf{W}^{1}_{i^1}$$
which is of dimension $\tilde{\nu}_{i}=\sum_{h\in \Omega,h'=i}\nu_{h''}+\omega^{1}_{i}$, and 
\begin{align*}
\phi_{\mathbf{V},i}(x)=((x_{h})_{h\in \Omega^{(1)},h' \neq i}, {\rm{Im}}  (\bigoplus _{h \in \Omega^{(1)}, h'=i} x_{h}))\ \textrm{for}\ x\in \bfEVOf^{i,0},\\
\phi_{\mathbf{V}'',i}(x)=((x_{h})_{h\in \Omega^{(1)},h' \neq i}, {\rm{Im}}  (\bigoplus _{h \in \Omega^{(1)}, h'=i} x_{h}))\ \textrm{for}\ x\in \bfEVOfpp^{i,0}
\end{align*}
which are principal $\mathrm{GL}(\bV_{i})$-bundle, $\mathrm{GL}(\bV''_{i})$-bundle respectively; $\mathbf{Flag}(\nu''_{i},\nu_{i},\tilde{\nu}_{i})$ is the flag variety, and $q_{1}, q_{2}$ are natural projections which are smooth and proper.

Recall that the torus $\mathbb{T}$ of rank $\sharp I$ acts on $\bfEVOf,\bfEVOfpp$ by scaling, and there are natural $\mathbb{T}$-actions on $\bfEVOf^{i,0},\bfEVOfpp^{i,0},\bfEVOfid,\bfEVOfidpp,\mathbf{Grass}(\nu_i, \tilde{\nu}_{i}),\mathbf{Flag}(\nu''_{i},\nu_{i},\tilde{\nu}_{i})$ such that all morphisms appearing in above diagram are $\mathbb{T}$-equivariant. We define the functor
$$\mathcal{E}^{(r)}_{i}=(j_{\mathbf{V}'',i})_{!} (\phi_{\mathbf{V},i})^{\ast} (q_{2})_{!}(q_{1})^{\ast} (\phi_{\mathbf{V},i})_{\flat}(j_{\mathbf{V},i})^{\ast}[-r\nu_{i}]: \mathcal{D}^b_{\bG_{\bV}\times \mathbb{T}}(\mathbf{E}_{\mathbf{V}\oplus\mathbf{W}^1,\Omega^{(1)}})\rightarrow\mathcal{D}^b_{\bG_{\bV''}\times \mathbb{T}}(\bfEVOfpp).$$
Similarly, by replacing $\mathbf{W}^1$ by $\mathbf{W}^1\oplus \mathbf{W}^2$, we can define the functor for the two-framed quiver $Q^{(2)}$
$$\mathcal{E}^{(r)}_{i}=(j_{\mathbf{V}'',i})_{!} (\phi_{\mathbf{V},i})^{\ast} (q_{2})_{!}(q_{1})^{\ast} (\phi_{\mathbf{V},i})_{\flat}(j_{\mathbf{V},i})^{\ast}[-r\nu_{i}]: \mathcal{D}^b_{\bG_{\bV}\times \mathbb{T}}(\mathbf{E}_{\mathbf{V}\oplus\mathbf{W}^1\oplus\mathbf{W}^2,\Omega^{(2)}})\rightarrow\mathcal{D}^b_{\bG_{\bV''}\times \mathbb{T}}(\mathbf{E}_{\mathbf{V}''\oplus\mathbf{W}^1\oplus\mathbf{W}^2,\Omega^{(2)}}).$$

Above definitions of $\mathcal{E}^{(r)}_{i}$ use the assumption that $i$ is a source for the orientation $\Omega$. In general, for any orientation $\Omega$ and $i\in I$, we fix a orientation $\Omega_i$ such that $i$ is a source for $\Omega_i$ and $\Omega^{(N)},\Omega^{(N)}_{i}$  then we use $\mathcal{E}^{(r)}_{i}$ associated to $\Omega_i$ and Fourier-Sato transforms to define the following functors 
\begin{align*}
\mathcal{E}^{(r)}_{i}&=\mathrm{Four}_{\Omega_i^{(1)},\Omega^{(1)}} \mathcal{E}^{(r)}_{i} \mathrm{Four}_{\Omega^{(1)},\Omega_{i}^{(1)}}: \mathcal{D}^b_{\bG_{\bV}\times \mathbb{T}}(\mathbf{E}_{\mathbf{V}\oplus\mathbf{W}^1,\Omega^{(1)}})\rightarrow\mathcal{D}^b_{\bG_{\bV''}\times \mathbb{T}}(\bfEVOfpp),\\
\mathcal{E}^{(r)}_{i}&=\mathrm{Four}_{\Omega_i^{(2)},\Omega^{(2)}} \mathcal{E}^{(r)}_{i} \mathrm{Four}_{\Omega^{(2)},\Omega_{i}^{(2)}}: \mathcal{D}^b_{\bG_{\bV}\times \mathbb{T}}(\mathbf{E}_{\mathbf{V}\oplus\mathbf{W}^1\oplus\mathbf{W}^2,\Omega^{(2)}})\rightarrow\mathcal{D}^b_{\bG_{\bV''}\times \mathbb{T}}(\mathbf{E}_{\mathbf{V}''\oplus\mathbf{W}^1\oplus\mathbf{W}^2,\Omega^{(2)}}).
\end{align*}
For simplicity, we also denote $\mathcal{E}^{(1)}_{i}$ by $\mathcal{E}_{i}$.

\begin{proposition}[{\cite[Proposition 3.18, Corollary 3.19]{fang2023lusztig}}]
For any $i \in I, r \in \mathbb{N}$, the functors $\mathcal{E}^{(r)}_{i}, \mathcal{F}^{(r)}_{i}$ induce well-defined functors 
\begin{align*}
&\mathcal{E}^{(r)}_{i}:\mathcal{Q}_{\nu,\omega^1}/\mathcal{N}_\nu\rightarrow\mathcal{Q}_{\nu'',\omega^1}/\mathcal{N}_{\nu''},\\
&\mathcal{F}^{(r)}_{i}:\mathcal{Q}_{\nu'',\omega^1}/\mathcal{N}_{\nu''}\rightarrow\mathcal{Q}_{\nu,\omega^1}/\mathcal{N}_\nu,\\
&\mathcal{E}^{(r)}_{i}:\mathcal{Q}_{\nu,\omega^1,\omega^2}/\mathcal{N}_\nu\rightarrow\mathcal{Q}_{\nu'',\omega^1,\omega^2}/\mathcal{N}_{\nu''},\\
&\mathcal{F}^{(r)}_{i}:\mathcal{Q}_{\nu'',\omega^1,\omega^2}/\mathcal{N}_{\nu''}\rightarrow\mathcal{Q}_{\nu,\omega^1,\omega^2}/\mathcal{N}_\nu.
\end{align*}
\end{proposition}
\begin{proof}
	The proof is essentially the same as the proof of \cite[Proposition 3.18, Corollary 3.19]{fang2023lusztig}.
\end{proof}

\subsubsection{The functors $\mathcal{K}_i,\mathcal{K}_{-i}$}
We define the functors $\mathcal{K}_i,\mathcal{K}_{-i}$ to be the shift functors
\begin{align*}
&\mathcal{K}_{i}={\mathrm{Id}}[-2\nu_{i}+\sum_{h \in \Omega_{i}, h'=i} \nu_{h''} +\omega^{1}_{i}]:\mathcal{Q}_{\nu,\omega^1}/\mathcal{N}_\nu\rightarrow\mathcal{Q}_{\nu,\omega^1}/\mathcal{N}_{\nu},\\
&\mathcal{K}_{-i}={\mathrm{Id}}[2\nu_{i}-\sum_{h \in \Omega_{i}, h'=i} \nu_{h''} -\omega^{1}_{i}]:\mathcal{Q}_{\nu,\omega^1}/\mathcal{N}_\nu\rightarrow\mathcal{Q}_{\nu,\omega^1}/\mathcal{N}_{\nu},\\
&\mathcal{K}_{i}={\mathrm{Id}}[-2\nu_{i}+\sum_{h \in \Omega_{i}, h'=i} \nu_{h''} +\omega^{1}_{i}+\omega^{2}_{i}]:\mathcal{Q}_{\nu,\omega^1,\omega^2}/\mathcal{N}_\nu\rightarrow\mathcal{Q}_{\nu,\omega^1,\omega^2}/\mathcal{N}_{\nu},\\
&\mathcal{K}_{-i}={\mathrm{Id}}[2\nu_{i}-\sum_{h \in \Omega_{i}, h'=i} \nu_{h''}-\omega^{1}_{i}-\omega^{2}_{i}]:\mathcal{Q}_{\nu,\omega^1,\omega^2}/\mathcal{N}_\nu\rightarrow\mathcal{Q}_{\nu,\omega^1,\omega^2}/\mathcal{N}_{\nu}.
\end{align*}

\subsection{Categorification theorem}

\begin{theorem}[{\cite[Theorem 3.28]{fang2023lusztig}}]\label{high}
The functors $\mathcal{E}^{(n)}_i,\mathcal{F}^{(n)}_i,\mathcal{K}_i$ for $i\in I,n\in \mathbb{N}$ induces $\mathcal{A}$-linear endomorphisms on the direct sums of Grothendieck groups $\mL(\omega^{1}),\mL(\omega^1,\omega^2)$, see Definition \ref{def localization}, such that $\mL(\omega^{1}),\mL(\omega^1,\omega^2)$ become $\Uv$-modules. Moreover, there is a $\Uv$-module isomorphism from $\mL(\omega^{1})$ to the integral form ${_{\mathcal{A}}L}_{v}(\lambda_{1})$ of the irreducible integrable highest weight module $L_v(\lambda_1)$,
$$\chi^{\omega^1}:\mL(\omega^{1}) \rightarrow  {_{\mathcal{A}}L}_{v}(\lambda_{1})$$
send $[\mathbb{C}|_{\mathbf{E}_{0\oplus \mathbf{W}^1,\Omega^{(1)}}}]$ to the highest weight vector of ${_{\mathcal{A}}L}_{v}(\lambda_{1})$. Moreover, $\mL(\omega^{1})$ has a $\mathcal{A}$-basis $$\bigsqcup_{\nu\in \mathbb{N}I}\{[L]\in \mL(\nu,\omega^{1})|L\in\mP_{\nu,\omega^1},[L] \neq 0 \textrm{ in } \mL(\nu,\omega^1)\}$$
and its image 
$$\bigsqcup_{\nu\in \mathbb{N}I}\{\chi^{\omega^1}([L])\in {_{\mathcal{A}}L}_{v}(\lambda_{1})|L\in\mP_{\nu,\omega^1},[L] \neq 0 \textrm{ in } \mL(\nu,\omega^1)\}$$
under the isomorphism $\chi^{\omega^1}$, is the canonical basis $\mathcal{B}_\lambda$ of ${_{\mathcal{A}}L}_{v}(\lambda_{1})$.
\end{theorem}

For any $\nu=\nu^1+\nu^2\in \mathbb{N}I$, we identify 
$$\mathbf{E}_{\mathbf{V}^2\oplus0\oplus \mathbf{W}^2,\Omega^{(2)}}\times \mathbf{E}_{\mathbf{V}^1\oplus \mathbf{W}^1\oplus 0,\Omega^{(2)}}=\mathbf{E}_{\mathbf{V}^2\oplus \mathbf{W}^2,\Omega^{(1)}}\times \mathbf{E}_{\mathbf{V}^1\oplus \mathbf{W}^1,\Omega^{(1)}}$$
and consider the composition of functors 
\begin{align*}
&\mathcal{D}^b_{\mathbf{G}_\bV\times\mathbb{T}}(\bfEVOff)\xrightarrow{\mathbf{D}}\mathcal{D}^b_{\mathbf{G}_\bV\times\mathbb{T}}(\bfEVOff)\xrightarrow{\mathbf{Res}^{\bV\oplus \mathbf{W}^{1} \oplus \mathbf{W}^{2}}_{\mathbf{V}^1\oplus \mathbf{W}^1,\mathbf{V}^2\oplus \mathbf{W}^2}}\\
&\mathcal{D}^b_{\mathbf{G}_{\bV^1}\times\mathbb{T}\times \mathbf{G}_{\bV^2}\times\mathbb{T}}(
\mathbf{E}_{\mathbf{V}^1\oplus \mathbf{W}^1,\Omega^{(1)}}\times \mathbf{E}_{\mathbf{V}^2\oplus \mathbf{W}^2,\Omega^{(1)}})
\xrightarrow{\mathbf{D} \boxtimes \mathbf{D}}\mathcal{D}^b_{\mathbf{G}_{\bV^1}\times\mathbb{T}\times \mathbf{G}_{\bV^2}\times\mathbb{T}}(
\mathbf{E}_{\mathbf{V}^1\oplus \mathbf{W}^1,\Omega^{(1)}}\times \mathbf{E}_{\mathbf{V}^2\oplus \mathbf{W}^2,\Omega^{(1)}})
\xrightarrow{\mathrm{sw}}\\
&\mathcal{D}^b_{\mathbf{G}_{\bV^2}\times\mathbb{T}\times \mathbf{G}_{\bV^1}\times\mathbb{T}}(\mathbf{E}_{\mathbf{V}^2\oplus \mathbf{W}^2,\Omega^{(1)}}\times \mathbf{E}_{\mathbf{V}^1\oplus \mathbf{W}^1,\Omega^{(1)}}),
\end{align*}
where $\mathbf{D}$ is the Verdier dual functor, $\mathbf{Res}^{\bV\oplus \mathbf{W}^{1} \oplus \mathbf{W}^{2}}_{\mathbf{V}^1\oplus \mathbf{W}^1,\mathbf{V}^2\oplus \mathbf{W}^2}$ is the Lusztig's restriction functor for the two-framed quiver $Q^{(2)}$, and $\mathrm{sw}$ is the natural isomorphism induced by swapping two coordinates.

\begin{theorem}[{\cite[Theorem 4.15]{fang2023tensor}}] \label{tensor}
Above composition of functors induces an $\Uv$-module isomorphism $\Delta: \mL(\omega^1,\omega^2) \rightarrow \mL(\omega^{1}) \otimes \mL(\omega^{2})$ such that $\mL(\omega^1,\omega^2)$ is isomorphic to $_{\mathcal{A}}L_{v}(\lambda_{2}) \otimes {_{\mathcal{A}}L}_{v}(\lambda_{1})$ via 
$$\tilde{\chi}^{\omega^1,\omega^2}=(\chi^{\omega^1}\otimes \chi^{\omega^2})\Delta:\mL(\omega^1,\omega^2)\rightarrow _{\mathcal{A}}L_{v}(\lambda_{2}) \otimes {_{\mathcal{A}}L}_{v}(\lambda_{1}).$$ 
Moreover, $\mL(\omega^1,\omega^2)$ has a $\mathcal{A}$-basis
$$\bigsqcup_{\nu\in \mathbb{N}I}\{[L]\in \mathcal{L}(\nu,\omega^1,\omega^2)\mid L\in \mathcal{P}_{\nu,\omega^1,\omega^2},[L]\not=0\ \textrm{in}\ \mathcal{L}(\nu,\omega^1,\omega^2)\}$$
and its image 
$$\bigsqcup_{\nu\in \mathbb{N}I}\{\tilde{\chi}^{\omega^1,\omega^2}([L])\in _{\mathcal{A}}L_{v}(\lambda_{2}) \otimes {_{\mathcal{A}}L}_{v}(\lambda_{1})\mid L\in \mathcal{P}_{\nu,\omega^1,\omega^2},[L]\not=0\ \textrm{in}\ \mathcal{L}(\nu,\omega^1,\omega^2)\}$$
under the isomorphism $\tilde{\chi}^{\omega^1,\omega^2}$ is the canonical basis $\mathcal{B}_{\lambda_1,\lambda_2}$ of $_{\mathcal{A}}L_{v}(\lambda_{2}) \otimes {_{\mathcal{A}}L}_{v}(\lambda_{1})$ defined in \cite{bao2016canonical}.
\end{theorem}

\begin{remark}
	In Lemma \ref{FST}, we only consider the Fourier-Sato transforms for derived categories associated to equivalent orientations. More generally, we can consider  the vector bundle $$\bfEVOf \rightarrow \mathbf{E}_{\bV\oplus \mathbf{W}^{1},\Omega^{(1)} \cap \Omega^{(1),'}},$$
	$$\bfEVOff \rightarrow \mathbf{E}_{\bV\oplus \mathbf{W}^{1}\oplus \mathbf{W}^{2},\Omega^{(2)} \cap \Omega^{(2),'}}$$   for any orientations $\Omega^{(N)},\Omega^{(N),'}$. 
	All Lusztig's semisimple sheaves are monodromic with respect to this vector bundle, and the Lusztig's sheaves in thick subcategories $\mN_{\nu}$ are preserved under the Fourier-Sato transform. It implies that our construction is independent on the choice of the orientation, but not only independent on the choice of mutation equivalent class of the orientation.
\end{remark}

\section{Nakajima's quiver variety and tensor product variety for integrable modules}

\subsection{Quiver varieties and tensor product varieties}
In this section, we recall the results about Lusztig's nilpotent varieties in \cite{MR1088333} and \cite{MR1758244} and Nakajima's quiver and tensor product varieties in  \cite{MR1302318}, \cite{MR1604167} and \cite{MR1865400}.

\subsubsection{Lusztig's nilpotent variety}\label{Lusztig's nilpotent variety}
We refer \cite{MR1088333} for details of this subsection.

Recall that in subsection \ref{Lusztig sheaf}, we associate the symmetric generalized matrix $C$ with a finite graph $(I,H)$ with an involution of $H$. For any $\nu\in \mathbb{N}I$, we fix a $I$-graded $\mathbb{C}$-vector space $\bV$ such that $|\bV|=\nu$ and define an affine space
$$\bfEVH= \bigoplus_{h\in H} \Hom (\bV_{h'},\bV_{h''})$$
such that the connected algebraic group 
$\bG_{\bV}$ acts on $\bfEVH$ by conjugation $g.x=(g_{h''}x_hg_{h'}^{-1})_{h\in H}$. We fix a function $\varepsilon:H\rightarrow\mathbb{C}^*$ such that $\varepsilon(h)+\varepsilon(\overline{h})=0$ for any $h\in H$. There is a non-degenerate $\bG_{\bV}$-invariant symplectic form on $\bfEVH$ given by 
$$\langle x,x'\rangle=\sum_{h\in H}\varepsilon(h)\mathrm{tr}(x_hx'_{\overline{h}}:V_{h''}\rightarrow V_{h''})$$
such that $\bfEVH$ can be identified with the cotangent bundle of $\bfEVO$, see \cite[Section 12.8]{MR1088333}. The moment map attached to the $\bG_{\bV}$-action on the symplectic vector space $\bfEVH$ is 
\begin{align*}
\mu_{\bV}:\bfEVH&\rightarrow \bigoplus_{i\in I}\mathrm{End}(\bV_i)\\
x&\mapsto (\sum_{h\in H, h''=i}\varepsilon(h)x_hx_{\overline{h}})_{i\in I}.
\end{align*}
Lusztig's nilpotent variety is defined to be
$$\Lambda_{\bV}=\{x\in \mu_{\bV}^{-1}(0)\mid x\ \textrm{is nilpotent}\},$$
where $x\in \bfEVH$ is said to be nilpotent, if there exists an $N\geqslant 2$ such that for any sequence $h_1,...,h_N\in H$ satisfying $h_1'=h_2'',...,h_{N-1}'=h_N''$, the composition $x_{h_1}...x_{h_N}:\bV_{h'_N}\rightarrow \bV_{h_1''}$ is zero. By \cite[Theorem 12.9]{MR1088333}, the nilpotent variety $\Lambda_{\bV}$ is a Lagrangian subvariety of $\bfEVH$.

\subsubsection{Nakajima's quiver variety}\label{Nakajima's quiver variety}
We refer \cite{MR1604167} for details of this subsection.

Recall that in subsection \ref{Lusztig sheaves on the framed quivers and the localizations}, we define the framed graph $(I^{(1)},H^{(1)})$ of the graph $(I,H)$ and the framed quiver $Q^{(1)}=(I^{(1)},H^{(1)},\Omega^{(1)})$ of the quiver $Q=(I,H,\Omega)$. For any $\nu\in \mathbb{N}I$ and $\omega\in \mathbb{N}I^1$, we fix an $I$-graded $\mathbb{C}$-vector space $\bV$ and an $I^1$-graded $\mathbb{C}$-vector space $\mathbf{W}$ such that $|\bV|=\nu$ and $|\mathbf{W}|=\omega$, then $\bV\oplus \mathbf{W}$ is an $I^{(1)}$-graded $\mathbb{C}$-vector space such that $|\bV\oplus\mathbf{W}|=\nu+\omega$. Regarding the framed graph $(I^{(1)},H^{(1)})$ as a kind of graph used in subsection \ref{Lusztig's nilpotent variety}, there is an affine space 
$$\bfEVHf=\bfEVH\oplus \bigoplus_{i \in I} \Hom(\bV_{i},\mathbf{W}_{i}) \oplus \bigoplus_{i \in I} \Hom(\mathbf{W}_{i},\bV_{i})$$
which has a $\bG_{\bV\oplus \mathbf{W}}$-action. Throughout this paper, we only consider the $\bG_\bV$-action on it.

We extend the fix function $\varepsilon:H\rightarrow\mathbb{C}^*$ to $H^{(1)}\rightarrow \mathbb{C}^*$ by $\varepsilon(i\rightarrow i^1)=-1,\varepsilon(i^1\rightarrow i)=1$ for $i\in I$. Then there is a corresponding symplectic form on $\bfEVHf$ such that $\bfEVHf$ can be identified with the cotangent bundle of $\bfEVOf$. The corresponding moment map attached to be $\bG_\bV$-action on the symplectic vector space $\bfEVHf$ is 
\begin{align*}
\mu_{\bV\oplus \mathbf{W}}:\bfEVHf&\rightarrow \bigoplus_{i\in I}\mathrm{End}(\bV_i)\\
(x,y,z)&\mapsto(\sum_{h\in H, h''=i}\varepsilon(h)x_hx_{\overline{h}}+z_iy_i)_{i\in I},
\end{align*}
where $x\in \bfEVH, y=(y_i)_{i\in I}\in \bigoplus_{i \in I}\Hom(\bV_{i},\mathbf{W}_{i}),z=(z_i)_{i\in I}\in \bigoplus_{i \in I}\Hom(\mathbf{W}_{i},\bV_{i})$. Then we have the corresponding Lusztig's nilpotent $\Lambda_{\bV\oplus \mathbf{W}}$ for the framed graph $(I^{(1)},H^{(1)})$ which is the subset of $\mu_{\bV\oplus \mathbf{W}}^{-1}(0)$ consisting of nilpotent elements. We define the subset
\begin{align*}
\Lambda_{\bV,\mathbf{W}}=\{(x,y,z)\in \Lambda_{\bV\oplus \mathbf{W}}\mid z=0\}=\Lambda_\bV\times \bigoplus_{i \in I} \Hom(\bV_{i},\mathbf{W}_{i}).
\end{align*}

An element $(x,y,z)\in \bfEVHf$ is said to be stable, if the zero space is the unique $I$-graded subspace $\bV'$ of $\bV$ such that $x_h(\bV'_{h'})\subset \bV'_{h''}$ and $y_i(\bV'_i)=0$ for any $h\in H,i\in I$, see \cite[Lemma 3.8]{MR1604167}. We define $\mu_{\bV\oplus \mathbf{W}}^{-1}(0)^s$ to be the subset of $\mu_{\bV\oplus \mathbf{W}}^{-1}(0)$ consisting of stable elements, and define $\Lambda_{\bV,\mathbf{W}}^s$ to be the subset of $\Lambda_{\bV,\mathbf{W}}$ consisting of stable elements. 

Notice that $\mu_{\bV\oplus \mathbf{W}}^{-1}(0)^s$ and $\Lambda_{\bV,\mathbf{W}}^s$ are $\bG_\bV$-invariant. By \cite[Lemma 3.10]{MR1604167}, the group $\bG_{\bV}$ acts freely on $\mu_{\bV\oplus \mathbf{W}}^{-1}(0)^s$ and $\Lambda_{\bV,\mathbf{W}}^s$. Nakajima's quiver varieties 
$$\mathfrak{M}(\nu,\omega),\ \mathfrak{L}(\nu,\omega)$$ 
are defined to be the geometric quotients of $\mu_{\bV\oplus \mathbf{W}}^{-1}(0)^s$ and $\Lambda_{\bV,\mathbf{W}}^s$ by $\bG_{\bV}$ respectively. We denote by $[x,y,z]$ the $\bG_\bV$-orbit of $(x,y,z)$, considered as an element in the geometric quotients.

\subsubsection{Tensor product variety}\label{Tensor product variety}
We refer \cite{MR1865400} for details of this subsection.

We use the same notations as subsection \ref{Nakajima's quiver variety}. Let $\mathbf{W}=\mathbf{W}^{1} \oplus \mathbf{W}^{2}$ be a fixed $I^1$-graded $\mathbb{C}$-vector space decomposition such that $|\mathbf{W}^1|=\omega^1$ and $|\mathbf{W}^2|=\omega^2$, then $\omega=\omega^1+\omega^2$. Recall that the group $\mathrm{GL}_{\mathbf{W}}$ acts on $\bfEVHf$ by conjugation, it induces a $\mathrm{GL}_{\mathbf{W}}$-action on $\mathfrak{M}(\nu,\omega)$. Then the one-parameter subgroup 
\begin{align*}
\lambda:\mathbb{G}_m&\rightarrow \mathrm{GL}_{\mathbf{W}^1}\times \mathrm{GL}_{\mathbf{W}^2}\subset \mathrm{GL}_{\mathbf{W}}\\
t&\mapsto ({\mathrm{Id}}_{\mathbf{W}^{1}},t\, {\mathrm{Id}}_{\mathbf{W}^{2}})
\end{align*}
acts on $\mathfrak{M}(\nu,\omega)$. By \cite[Lemma 3.2]{MR1865400}, the $\lambda(\mathbb{G}_m)$-fixed point set $\mathfrak{M}(\nu,\omega)^{\lambda(\mathbb{G}_m)}$ is isomorphic to 
$$\bigsqcup_{\nu'+\nu''=\nu} \mathfrak{M}(\nu',\omega^{1}) \times \mathfrak{M}(\nu'',\omega^{2}).$$
Nakajima's tensor product variety is defined to be
$$\tilde{\mathfrak{Z}}(\nu,\omega)=\{ [x,y,z] \in \mathfrak{M}(\nu,\omega)\mid \lim_{t \rightarrow  0} \lambda(t). [x,y,z]  \in \bigsqcup_{\nu'+\nu''=\nu} \mathfrak{L} (\nu',\omega^{1}) \times \mathfrak{L}(\nu'',\omega^{2}) \}.$$

Following \cite{Malkin2003Tensor}, the tensor product variety $\tilde{\mathfrak{Z}}$ is a geometric quotient as follows. 

Let $\Pi_{\bV,\mathbf{W}^1\oplus \mathbf{W}^2}$ be the subvariety of $\bfEVHf$ consisting of those nilpotent elements $(x,y,z)$ such that $z(\mathbf{W}^{2})=0$ and $\mathbf{S}_{1}(x,y,z) \subset \mathbf{S}_{2}(x,y,z)$, where $\mathbf{S}_{1}(x,y,z)$ is the smallest $x$-stable subspace of $\mathbf{V}$ containing $z(\mathbf{W})$ and $\mathbf{S}_{2}(x,y,z)$ is the largest $x$-stable subspace of $\mathbf{V}$ contained in $y^{-1}(\mathbf{W}^2)$. Let $\Pi_{\bV,\mathbf{W}^1\oplus\mathbf{W}^2,\nu''}$ be the locally closed subset of $\Pi_{\bV,\mathbf{W}^1\oplus\mathbf{W}^2}$ consisting of $(x,y,z)$ such that $|\mathbf{S}_{2}(x,y,z)|=\nu''$, and let $\Pi_{\bV,\mathbf{W}^1\oplus\mathbf{W}^2,\mathbf{V}^{2}}$ be the closed subset of $\Pi_{\bV,\mathbf{W}^1\oplus\mathbf{W}^2,\nu''}$ consisting of $(x,y,z)$ such that $\mathbf{S}_{2}(x,y,z)$ equals to a fixed subspace $\mathbf{V}^{2} \subset\mathbf{V}$ satisfying $|\mathbf{V}^2|=\nu''$. Then $\mathbf{V}^{2} \oplus \mathbf{W}^{2}$ is $(x,y,z)$-stable for any $(x,y,z)$ in $\Pi^{s}_{\bV,\mathbf{W}^1\oplus\mathbf{W}^2,\mathbf{V}^{2}} $. We fix an isomorphism $\mathbf{V}/\mathbf{V}^{2} \cong \mathbf{V}^{1}$, and denote the restrictions of  $(xy,z)$ on  $\mathbf{V}^{1} \oplus \mathbf{W}^{1}$ and $\bV^{2} \oplus \mathbf{W}^{2}$ by  $(x^{1},y^{1},z^{1})$  and $ (x^{2},y^{2},z^{2})$ respectively. The conditions that $z(\mathbf{W}^{2})=0$ and $\mathbf{S}_{1}(x,y,z) \subset \mathbf{S}_{2}(x,y,z)$ imply that $z^{1}=z^{2}=0$. Moreover, if $(x,y,z)$ is a stable point, we can deduce that $(x^{1},y^{1},z^{1})$  and $ (x^{2},y^{2},z^{2})$ are stable. There is an isomorphism
\begin{align*}
\Pi^{s}_{\bV,\mathbf{W}^1\oplus\mathbf{W}^2,\mathbf{V}^{2}} \rightarrow \Lambda^{s}_{\bV^{1},\mathbf{W}^{1}} \times  \Lambda^{s}_{\bV^{2},\mathbf{W}^{2}}\\(x,y,z) \mapsto ((x^{1},y^{1},z^{1}), (x^{2},y^{2},z^{2}))
\end{align*}
which is a vector bundle. In particular, the geometric quotient of  $\Pi^{s}_{\bV,\mathbf{W}^1\oplus\mathbf{W}^2,\nu''}$ is exactly those $[x,y,z]$ with limit in $\mathfrak{L} (\nu',\omega^{1}) \times \mathfrak{L}(\nu'',\omega^{2})$. Hence $\tilde{\mathfrak{Z}}$ is isomorphic to the geometric quotient of  $\Pi^{s}_{\bV,\mathbf{W}^1\oplus\mathbf{W}^2}$.

\subsection{Hecke correspondence and realizations of $L_1(\lambda)$ and $L_1(\lambda_1)\otimes L_1(\lambda_2)$}

We refer \cite{MR1604167} and \cite{MR1865400} for details of this subsection.

Given dominant weight $\lambda$, we set $\omega\in \mathbb{N}I^1$ in subsections \ref{Nakajima's quiver variety} and \ref{Tensor product variety} to be 
$\omega=\sum_{i\in I}\langle i,\lambda\rangle i^1$. For any $i\in I$ and $\nu=\nu'+i\in \mathbb{N}I$, we assume that the fixed $I$-graded $\mathbb{C}$-vector space $\bV,\bV'$ such that $|\bV|=\nu,|\bV'|=\nu'$ and $\bV'$ is a subspace of $\bV$. The Hecke correspondence $\mathfrak{P}_{i}(\nu,\omega)$ is defined to be the subvariety of $\mathfrak{M}(\nu',\omega) \times \mathfrak{M}(\nu,\omega)$ consisting of $([x',y',z'],[x,y,z])$ such that there exists $\xi=(\xi_{j})_{j \in I} \in \bigoplus_{j \in I}\Hom(\bV'_{j},\bV_{j})$ such that 
$$\xi x'=x\xi,y\xi=y',\xi z'=z.$$

For the orientation $\Omega\subset H$, let $\mathbf{A}_{\Omega}$ be the adjacency matrix of the quiver $Q=(I,H,\Omega)$, that is, it is a $I\times I$-matrix whose $(i,j)$ entry is the number of arrows $h\in \Omega$ such that $h'=i,h''=j$. Let $\mathbf{C}_{\Omega}=\mathrm{Id}_{I\times I}-\mathbf{A}_{\Omega}$. Nakajima defined the following operators
\begin{align*}
&e_i=(-1)^{\langle i , \mathbf{C}_{\bar{\Omega}}\nu \rangle} [\mathfrak{P}_{i}(\nu,\omega)]:\mathbf{H}^{\mathrm{BM}}_{\mathrm{top}}(\mathfrak{L}(\nu,\omega) ,\mathbb{Q}) \rightarrow \mathbf{H}^{\mathrm{BM}}_{\mathrm{top}}(\mathfrak{L}(\nu',\omega) ,\mathbb{Q}),\\
&e_i=(-1)^{\langle i , \mathbf{C}_{\bar{\Omega}}\nu \rangle} [\mathfrak{P}_{i}(\nu,\omega)]:\mathbf{H}^{\mathrm{BM}}_{\mathrm{top}}(\tilde{\mathfrak{Z}}(\nu,\omega) ,\mathbb{Q}) \rightarrow \mathbf{H}^{\mathrm{BM}}_{\mathrm{top}}(\tilde{\mathfrak{Z}}(\nu',\omega) ,\mathbb{Q}),\\
&f_{i}=(-1)^{\langle i , \omega- \mathbf{C}_{\Omega}\nu \rangle} [\mathrm{sw}(\mathfrak{P}_{i}(\nu,\omega))]:\mathbf{H}^{\mathrm{BM}}_{\mathrm{top}}(\mathfrak{L}(\nu',\omega) ,\mathbb{Q}) \rightarrow \mathbf{H}^{\mathrm{BM}}_{\mathrm{top}}(\mathfrak{L}(\nu,\omega) ,\mathbb{Q}),\\
&f_{i}=(-1)^{\langle i , \omega- \mathbf{C}_{\Omega}\nu \rangle} [\mathrm{sw}(\mathfrak{P}_{i}(\nu,\omega))]:\mathbf{H}^{\mathrm{BM}}_{\mathrm{top}}(\tilde{\mathfrak{Z}}(\nu',\omega) ,\mathbb{Q}) \rightarrow \mathbf{H}^{\mathrm{BM}}_{\mathrm{top}}(\tilde{\mathfrak{Z}}(\nu,\omega) ,\mathbb{Q}),
\end{align*}
where $\mathrm{sw}: \mathfrak{M}(\nu',\omega) \times \mathfrak{M}(\nu,\omega) \rightarrow \mathfrak{M}(\nu,\omega) \times \mathfrak{M}(\nu',\omega) $ the natural map exchanging the coordinates. We remark that the sign twists in  \cite{MR1604167} is incorrect and one should use the sign twists in formulas of \cite[Section 9.3.2]{nakajima2001quiver}. The fundamental class $[\Delta(\nu,\omega)]$ of the diagonal defines  operators
\begin{align*}
&[\Delta(\nu,\omega)]:\mathbf{H}^{\mathrm{BM}}_{\mathrm{top}}(\mathfrak{L}(\nu,\omega) ,\mathbb{Q}) \rightarrow \mathbf{H}^{\mathrm{BM}}_{\mathrm{top}}(\mathfrak{L}(\nu,\omega) ,\mathbb{Q}),\\
&[\Delta(\nu,\omega)]:\mathbf{H}^{\mathrm{BM}}_{\mathrm{top}}(\tilde{\mathfrak{Z}}(\nu,\omega) ,\mathbb{Q}) \rightarrow \mathbf{H}^{\mathrm{BM}}_{\mathrm{top}}(\tilde{\mathfrak{Z}}(\nu,\omega) ,\mathbb{Q}).
\end{align*}

\begin{theorem}[{\cite[Thoerem 10.2]{MR1604167}}] \label{canN1}
With the operators $e_{i},f_{i}$ and $[\Delta(\nu,\omega)]$, the direct sum of Borel-Moore homology group $\mathbf{H}(\mathfrak{L}(\omega))= \bigoplus\limits_{\nu \in \mathbb{N}I} \mathbf{H}^{\mathrm{BM}}_{\mathrm{top}}(\mathfrak{L}(\nu,\omega) ,\mathbb{Q}) $ becomes an integrable $\mathbf{U}_{1}(\mathfrak{g})$-module, and there is an $\mathbf{U}_{1}(\mathfrak{g})$-module isomorphism 
$$\varphi^\omega:\mathbf{H}(\mathfrak{L}(\omega)) \rightarrow L_{1}(\lambda) $$
sends $[\mathfrak{L}(0,\omega)]$ to the highest weight vector of $L_{1}(\lambda)$.
\end{theorem}

Given dominant weights $\lambda_1,\lambda_2$, we set $\omega^1=\sum_{i\in I}\langle i,\lambda_1\rangle i^1,\omega^2=\sum_{i\in I}\langle i,\lambda_2\rangle i^1$ and $\omega=\omega^1+\omega^2$ in subsections \ref{Tensor product variety}.

\begin{theorem}[{\cite[Theorem 5.2]{MR1865400}}]
	With the operators $e_{i},f_{i}$ and $[\Delta(\nu,\omega)]$, the direct sum of Borel-Moore homology group $\mathbf{H}(\tilde{\mathfrak{Z}}(\omega))= \bigoplus\limits_{\nu \in \mathbb{N}I} \mathbf{H}^{\mathrm{BM}}_{\mathrm{top}}(\tilde{\mathfrak{Z}}(\nu,\omega) ,\mathbb{Q}) $ becomes an integrable $\mathbf{U}_{1}(\mathfrak{g})$-module, and is isomorphic to the tensor products $L_{1}(\lambda_{1}) \otimes L_{1}(\lambda_{2})$.
\end{theorem}

Nakajima proved above theorem above by calculating the character of $\mathbf{H}(\tilde{\mathfrak{Z}}(\omega))$, so the isomorphism in the theorem is not canonical. In order to construct a canonical isomorphism, he considered the subvariety $\tilde{\mathfrak{Z}}_{1}$ of $\tilde{\mathfrak{Z}}$ consisting of those $[x,y,z]$ such that  $$\lim_{t \rightarrow  0} \lambda(t).[x,y,z]  \in \coprod\limits_{\nu} \mathfrak{L} (0,\omega^{1}) \times \mathfrak{L}(\nu,\omega^{2}).$$
By \cite[Proposition 5.5]{MR1865400}, since $\tilde{\mathfrak{Z}}_{1}$ is a vector bundle over $\coprod\limits_{\nu} \mathfrak{L} (0,\omega^{1}) \times \mathfrak{L}(\nu,\omega^{2})$, there is an $\mathbf{U}_{1}^{+}(\mathfrak{g})$-linear isomorphism given by Thom isomorphism $\mathbf{H}^{\mathrm{BM}}_{\mathrm{top}}(\tilde{\mathfrak{Z}}_{1},\mathbb{Q}) \cong \mathbf{H}(\mathfrak{L}(\omega^{2})). $ Moreover, the subspace $\mathbf{H}^{\mathrm{BM}}_{\mathrm{top}}(\tilde{\mathfrak{Z}}_{1},\mathbb{Q})$ generates the  $\mathbf{H}(\tilde{\mathfrak{Z}}(\omega))$ as a $\mathfrak{g}$-module. Nakajima raised following conjecture and proved that this conjecture holds for ADE case, see \cite[Theorem 5.9]{MR1865400}.

\begin{conjecture}[{\cite[Conjecture 5.6]{MR1865400}}] \label{canN2}
	There exists a unique isomorphism of $\mathfrak{g}$-modules 
    $$\tilde{\varphi}: \mathbf{H}(\mathfrak{L}(\omega^{1})) \otimes \mathbf{H}(\mathfrak{L}(\omega^{2})) \rightarrow \mathbf{H}(\tilde{\mathfrak{Z}}(\omega)), $$
	such that the restriction of $\tilde{\varphi}$ on the subspace $[\mathfrak{L}(0,\omega^{1})] \otimes \mathbf{H}(\mathfrak{L}(\omega^{2}))  $ equals to the composition of Thom isomorphism $\mathbf{H}^{\mathrm{BM}}_{\mathrm{top}}(\tilde{\mathfrak{Z}}_{1},\mathbb{Q}) \cong \mathbf{H}(\mathfrak{L}(\omega^{2})) $ and the inclusion $\mathbf{H}^{\mathrm{BM}}_{\mathrm{top}}(\tilde{\mathfrak{Z}}_{1},\mathbb{Q}) \rightarrow  \mathbf{H}(\tilde{\mathfrak{Z}}(\omega)).$
\end{conjecture}

\subsection{The string order and monomial basis}

In this section, we recall the construction of a monomial basis of $\mathbf{H}(\mathfrak{L}(\omega))$ in \cite{fang2023lusztig}.

For any $i\in I, \nu\in \mathbb{N}I$ and $0\leqslant t\leqslant \nu_i$, we define $\mathfrak{L}_{i,t}(\nu,\omega)\subset \mathfrak{L}(\nu,\omega)$ to be the locally closed subset consisting of $[x,y,0]$ such that 
$$\mathrm{codim}_{\mathbf{V}_i}(\sum_{h\in H,h''=i}\mathrm{Im}(x_h:\mathbf{V}_{h'}\rightarrow \mathbf{V}_i))=t,$$
then all these $\mathfrak{L}_{i,t}(\nu,\omega),0\leqslant t\leqslant \nu_i$ form a stratification of $\mathfrak{L}(\nu,\omega)$. For any irreducible component $Z\subset \mathfrak{L}(\nu,\omega)$, there exists a unique integer $\epsilon_{i}(Z)$ such that $Z \cap \mathfrak{L}_{i,\epsilon_{i}(Z)}(\nu,\omega)$ is open dense in $Z$.

\begin{lemma}[{\cite[Proposition 4.3.2]{saito2002crystal}}]
For any $0\leqslant t\leqslant \nu_i$, let $\nu'=\nu-ti$, then there is a bijection 
$$\varrho_{i,t}:\{Z \in {\mathrm{Irr}}\mathfrak{L}(\nu,\omega) \mid\epsilon_{i}(Z)=t \} \rightarrow \{Z' \in {\mathrm{Irr}}\mathfrak{L}(\nu',\omega) \mid\epsilon_{i}(Z')=0 \}.$$
\end{lemma}

\begin{lemma}[{\cite[Lemma 10.1]{MR1604167}}] \label{Nkey}
For any irreducible component $Z \subseteq \mathfrak{L}(\nu,\omega)$ such that $\epsilon_{i}(Z)=t>0$, suppose that $Z'=\varrho_{i,t}(Z) $, we have
\begin{equation*}
f_{i}^{(t)} ([Z'])=\pm [Z] + \sum_{\epsilon_{i}(Z'')>t} c_{Z''}[Z''],
\end{equation*}
where $c_{Z''}\in \mathbb{Q}$ are constants.
\end{lemma} 

\begin{corollary}\label{Ns}
	The subspace $f_{i}^{t}\mathbf{H}(\mathfrak{L}(\omega))$ equals to the subspace $${\rm{span}}_{\mathbb{Q}}\{[Z]\mid Z \in {\rm{Irr}} \mathfrak{L}(\nu,\omega),\epsilon_{i}(Z) \geqslant t   \}. $$
\end{corollary}
\begin{proof}
By the definition of $f_{i}$, we have $f_{i}^{t}\mathbf{H}(\mathfrak{L}(\omega))\subset {\rm{span}}_{\mathbb{Q}}\{[Z]\mid Z \in {\rm{Irr}} \mathfrak{L}(\nu,\omega),\epsilon_{i}(Z) \geqslant t \}$. Conversely, we make an decreasing induction on $\epsilon_{i}(Z)$ to prove $[Z]\in f_{i}^{t}\mathbf{H}(\mathfrak{L}(\omega))$ when $t\epsilon_{i}(Z) \geqslant t$. If $\epsilon_{i}(Z)=\nu_i$, we have 
$$[Z]=\pm f_{i}^{(\nu_i)}[\varrho_{i,\nu_i}(Z)]\in f_{i}^{t}\mathbf{H}(\mathfrak{L}(\omega)),$$ 
as desired. If $t\leqslant \epsilon_{i}(Z)<\nu_i$, by inductive hypothesis, we have 
$$[Z]=\pm(f_{i}^{(\epsilon_{i}(Z))}[\varrho_{i,\epsilon_{i}(Z)}(Z)]-\sum_{\epsilon_{i}(Z'')>\epsilon_{i}(Z)} c_{Z''}[Z''])\in f_{i}^{t}\mathbf{H}(\mathfrak{L}(\omega)),$$
as desired.
\end{proof}

Analogue to the refine order on $\mathcal{B}_\lambda$, see subsection \ref{Refine string order}, we define a partial order on $\bigsqcup_{\nu\in \mathbb{N}I}\mathrm{Irr}\mathfrak{L}(\nu,\omega)$. We fix an order on $I=\{i_1,...,i_n\}$ such that $i_{1} \prec i_{2} \prec \cdots \prec i_{n}$, and inductively define a sequence $s^{\prec}(Z)$ in $\mathbb{N}\times I$ for any $Z\in \mathrm{Irr}\mathfrak{L}(\nu,\omega)$ as follows,\\
$\bullet$ if $\nu=0$, we define $s^{\prec}(Z)$ to be the empty sequence;\\
$\bullet$ if $\nu\not=0$, let $i$ be the minimal vertex such that $\epsilon_{i}(Z)>0$, then $s^{\prec}(\varrho_{i,\epsilon_{i}(Z)}(Z))$ has been defined by inductive hypothesis, and we define
$$s^{\prec}(Z)=((i,\epsilon_{i}(Z)),s^{\prec}(\varrho_{i,\epsilon_{i}(Z)}(Z))).$$
Then we define a partial order on $\mathrm{Irr}\mathfrak{L}(\nu,\omega)$ as follows, for any $Z,Z'\in \mathrm{Irr}\mathfrak{L}(\nu,\omega)$ such that 
$$s^{\prec}(Z)=((i_{n_{1}},m_{1} ), (i_{n_{2}},m_{2}),\cdots,(i_{n_{k}},m_{k}) ),\ s^{\prec}(Z')=((i_{n'_{1}},m'_{1} ), (i_{n'_{2}},m'_{2}),\cdots,(i_{n'_{l}},m'_{l})),$$ 
we define $Z'\prec Z$, if there exists $r$ such that $(i_{n_{t}},m_{t})=(i_{n'_{t}},m'_{t})$ for $1 \leqslant t< r$ and $(i_{n'_{r}},m'_{r})\prec (i_{n_{r}},m_{r})$ with respect to the lexicographical order. We define a partial order on $\bigsqcup_{\nu\in \mathbb{N}I}\mathrm{Irr}\mathfrak{L}(\nu,\omega)$ by taking disjoint union.

\begin{proposition}[{\cite[Section 4.3]{fang2023lusztig}}]\label{mono}
For any irreducible component $Z\in \mathrm{Irr}\mathfrak{L}(\nu,\omega)$ such that $$s^{\prec}(Z)= ((i_{1},a_{1}),(i_{2},a_{2}),\cdots, (i_{l},a_{l}) ) ,$$ let
$$m_{Z}=f^{(a_{1})}_{i_{1}} f^{(a_{2})}_{i_{2}} \cdots f^{(a_{l})}_{i_{l}} [\mathfrak{L}(0,\omega)] \in \mathbf{H}(\mathfrak{L}(\omega)). $$
Then we have
$$m_{Z}=\pm[Z]+\sum \limits_{Z \prec Z'} c_{Z,Z'}[Z'].$$
 	In particular, $\{m_{Z}\mid Z \in {\rm{Irr}} \mathfrak{L}(\nu,\omega) \}$ forms a monomial basis of $\mathbf{H}^{\mathrm{BM}}_{\mathrm{top}}(\mathfrak{L}(\nu,\omega),\mathbb{Q})$.
\end{proposition}
\begin{proof}
We use Lemma \ref{Nkey} and make an induction on the length $l$ of $s^{\prec}(Z)$.  
If $l=1$ and $s^{\prec}(Z)=((i_{1},a_{1} ))$, then $\nu=a_{1}i_{1}$ and 
\begin{equation*}
[Z]= [\mathfrak{L}(\nu,\omega) ] =\pm f_{i_{1}}^{(a_{1})}[\mathfrak{L}(0,\omega)]=m_{Z}
\end{equation*}
holds trivially. If $l>1$, by Lemma \ref{Nkey}, there exists another irreducible component $Z'$ such that $\epsilon_{i_{1}}(Z')=0$, $s^{\prec}(Z)=((i_{1},a_{1}),s^{\prec}(Z')) $ and 
	\begin{equation*}
		f_{i_{1}}^{(a_{1})}[Z']=\pm [Z]+\sum\limits_{\epsilon_{i_{1}}(Z'')>a_{1} } c_{Z',Z''}[Z'']
	\end{equation*}
	By inductive hypothesis, we can assume that
	\begin{equation*}
		m_{Z'}=\pm [Z']+\sum \limits_{Z' \prec Z'' }c_{Z',Z''}[Z''].
	\end{equation*}
  Since $s^{\prec}(Z)=((i_{1},a_{1}),s^{\prec}(Z')) $, we have 
	\begin{equation*}
		\begin{split}	
			m_{Z}=& f_{i_{1}}^{(a_{1})} m_{Z'} \\
			=& f_{i_{1}}^{(a_{1})} (\pm[Z']+\sum \limits_{Z' \prec Z'' }c_{Z',Z''}[Z''] ) \\
			=&[L]+\sum\limits_{\epsilon_{i_{1}}(Z'')>a_{1} } c_{Z',Z''}[Z'']+\sum \limits_{Z' \prec Z'' }c_{Z',Z''}f_{i_{1}}^{(a_{1})}[Z''].
		\end{split}
	\end{equation*}
	Notice that if $\epsilon_{i_{1}}(Z'')=a>a_{1}$, then $s^{\prec}(Z'')$ either starts with $(i_{1},m)$ or starts  with $(i_{r},m')$ for some $i_{r}\prec i_{1}$, hence  we have the following fact
	\begin{center}
		$(\heartsuit)\ \ Z \prec Z''$ holds for those $Z''$ which satisfies $\epsilon_{i_{1}}(Z'')>a_{1}$.
	\end{center}
Now it suffices to show that 
	\begin{equation*}
		f_{i_{1}}^{(a_{1})}[Z'']= \sum \limits_{Z \prec Z''' } d_{Z'''}[Z''']
	\end{equation*}
	for those $Z''$ which satisfy $Z' \prec Z''$. If $Z''$ satisfies $\epsilon_{i_{1}}(Z'')>0$, by the Corollary \ref{Ns} above, $[Z'']$ belongs to the subspace
	$f_{i_{1}}\mathbf{H}(\mathfrak{L}(\omega))  $,
	so 
	$f_{i_{1}}^{(a_{1})}  [Z'']$ belongs to 	$f_{i_{1}}^{a_{1}+1}\mathbf{H}(\mathfrak{L}(\omega)).  $ 
	Again by the Corollary \ref{Ns}, $f_{i_{1}}^{(a_{1})}  [Z'']$ can be written as a linear combination of $[Z''']$ with $\epsilon_{i_{1}}(Z''')>a_{1}$.	By $(\heartsuit)$, those $Z'''$ satisfy $Z \prec Z'''$. For those $Z''$ such that $\epsilon_{i_{1}}(Z'')=0$,
	\begin{equation*}
		f_{i_{1}}^{(a_{1})}  [Z'']=\pm[\tilde{Z}''] + \sum \limits_{\epsilon_{i_{1}}(Z''')>a_{1} } e_{Z'''}[Z''']
	\end{equation*} 
	where $\tilde{Z}''$ is  an irreducible component with $\epsilon_{i_{1}}(\tilde{Z})=a_{1}$. By $(\heartsuit)$,  those $Z'''$  with $\epsilon_{i_{1}}(Z''')>a_{1}$ satisfy $Z \prec Z'''$. So we only need to check $Z \prec \tilde{Z}''$. In fact, if $\epsilon_{i_{r}}(\tilde{Z}'') >0$ holds for some $i_{r} \prec i_{1}$, then $s^{\prec}(\tilde{Z}'')$ starts with $(i_{r},m')$ and $Z \prec \tilde{Z}''$ holds by definition of the string order. Otherwise, we have $s^{\prec}(\tilde{Z}'')=((i_{1},a_{1}),s^{\prec}(Z''))$. Since $Z' \prec Z''$, we have $Z \prec Z'''$. In a conclusion, $m_{Z}=\pm[Z]+\sum \limits_{Z \prec Z'} c_{Z,Z'}[Z']$ always holds. 
\end{proof}

\section{Equivariant characteristic cycles for Lusztig sheaves}
In this section, we establish the relation between canonical basis and the top Borel-Moore homology  groups of Nakajima's quiver varieties and tensor product varieties. 

\subsection{Singular supports of Lusztig sheaves}

\subsubsection{Singular supports of $\mP_{\nu}$}
We refer \cite[Section 12 and 13]{MR1088333} for details of this subsection.

\begin{proposition}[{\cite[Corollary 13.6]{MR1088333}}]
For any $\nu\in \mathbb{N}I$ and $L\in \mP_{\nu}$, the singular support $SS(L)$ is a union of irreducible components of $\Lambda_{\bV}$.
\end{proposition}

For any $i\in I, \nu\in \mathbb{N}I$ and $0\leqslant t\leqslant \nu_i$, we define $\Lambda_{\bV,i,t}\subset \Lambda_{\bV}$ to be the locally closed subset consisting of $x$ such that 
$$\mathrm{comdim}_{\bV_i}(\sum_{h\in H,h''=i}\mathrm{Im}(x_h:\bV_{h'}\rightarrow \bV_i))=t,$$
then all these $\Lambda_{\bV,i,t},0\leqslant t\leqslant \nu_i$ form a stratification of $\Lambda_{\bV}$. Dually, we define $\Lambda^t_{\bV,i}\subset \Lambda_{\bV}$ to be the locally closed subset consisting of $x$ such that
$$\mathrm{dim}\bigcap_{h\in H,h'=i}\mathrm{Ker}(x_h:\bV_i\rightarrow \bV_{h''})=t,$$
then all these $\Lambda^t_{\bV,i},0\leqslant t\leqslant \nu_i$ form a stratification of $\Lambda_{\bV}$. For any irreducible component $Z\subset \Lambda_{\bV}$, there are unique integers $\epsilon_{i}(Z),\epsilon_{i}^\ast(Z)$ such that $\Lambda_{\bV,i,\epsilon_{i}(Z)} \cap Z, \Lambda^{\epsilon^\ast_{i}(Z)}_{\bV,i} \cap Z$ are open dense in $Z$.
\begin{lemma}[{\cite[Lemma 12.5]{MR1088333}}]
For any $0\leqslant t\leqslant \nu_i$, let $\nu',\nu''=\nu-ti$, then there are bijections
\begin{align*}
\rho_{i,t}: \{Z \in {\rm{Irr}} \Lambda_{\bV}\mid\epsilon_{i}(Z)=t \} \rightarrow   \{Z' \in {\rm{Irr}} \Lambda_{\bV''}\mid\epsilon_{i}(Z')=0 \},\\
\rho^{\ast}_{i,t}: \{Z \in {\rm{Irr}} \Lambda_{\bV}\mid\epsilon^{\ast}_{i}(Z)=t \} \rightarrow   \{Z' \in {\rm{Irr}} \Lambda_{\bV'}\mid\epsilon^{\ast}_{i}(Z')=0 \}.
\end{align*}
\end{lemma}

\begin{theorem}[{\cite[Theorem 5.3.2]{MR1458969}}]
There is an isomorphism of crystals
$$\bigsqcup_{\nu\in \mathbb{N}I}\mathrm{Irr}\Lambda_\bV\cong \mathcal{B}.$$
\end{theorem}
More precisely, there is a bijection $$\Phi=\Phi_Q:\bigsqcup_{\nu \in \mathbb{N}I}\mP_\nu \rightarrow \bigsqcup_{\nu \in \mathbb{N}I} {\rm{Irr}} \Lambda_{\bV}$$ such that $\epsilon_{i}(L)= \epsilon_{i}(\Phi(L)),\epsilon^{\ast}_{i}(L)= \epsilon^{\ast}_{i}(\Phi(L))$ for any $i\in I$, see subsection \ref{Refine string order} for the definitions of $\epsilon_{i}(L),\epsilon^\ast_{i}(L)$. Moreover, by \cite[Theorem 6.6.6]{MR1458969}, we have 
\begin{equation}\label{KS}
	\Phi(L) \subseteq SS(L) \subseteq \bigcup Z,
\end{equation}
where the union runs over all $Z\in \Lambda_\bV$ such that $\epsilon_{i}(Z) \geqslant \epsilon_{i}(L),\epsilon^{\ast}_{i}(Z) \geqslant \epsilon^{\ast}_{i}(L)$ for any $i \in I$. By \cite[Lemma 8.2.1]{MR1458969} we have 
$	\rho_{i,t}(\Phi(L)) =	\Phi(L') $ if and only if $\pi_{i,t}(L)=L' $,  and $\rho^{\ast}_{i,t}(\Phi(L)) =	\Phi(L') $ if and only if $\pi^{\ast}_{i,t}(L)=L'.$ 

\subsubsection{Singular supports of $\mP_{\nu,\omega}$}

\begin{proposition} \label{fss}
	Assume $L$ is a simple perverse sheaf in $\mP_{\nu,\omega}$, then $SS(L) \subset \Lambda_{\bV,\mathbf{W}}$ and $\Phi_{Q^{(1)}}$ restricts to a bijection between ${\rm{Irr}}\Lambda_{\bV,\mathbf{W}}$ and $\mP_{\nu,\omega} $.  
	In particular, 	 $L\in \mathcal{N}_{\nu}$ if and only if $SS(L) \cap  \Lambda^{s}_{\bV,\mathbf{W}}$ is empty, if and only if $\Phi_{Q^{(1)}}(L) \cap \Lambda^{s}_{\bV,\mathbf{W}} $ is empty.
\end{proposition} 
\begin{proof}
	 Consider  the natural projections $\pi_{\mathbf{W}}:\Lambda_{\bV,\mathbf{W}} \rightarrow \Lambda_{\bV}$ and $\pi_{\mathbf{W}}:\mathbf{E}_{\bV,\mathbf{W},\Omega^{(1)}} \rightarrow \bfEVO$, then $(\pi_{\mathbf{W}})^{\ast}$ sends $L' \in \mP_{\nu}$ to $L=(\pi_{\mathbf{W}})^{\ast}(L') \in \mP_{\nu,\omega} $ up to shifts, and $(\pi_{\mathbf{W}})^{-1} $ sends $Z' \in {\rm{Irr}} \Lambda_{\bV}$ to $Z=(\pi_{\mathbf{W}})^{-1}(Z') \in {\rm{Irr}} \Lambda_{\bV,\mathbf{W}}$. Notice that $(\pi_{\mathbf{W}})^{\ast}(L')= \prod\limits_{i \in I}( \pi^{\ast}_{i^{1},\omega_{i}})^{-1}(L') $ and $(\pi_{\mathbf{W}})^{-1}(Z')= \prod\limits_{i \in I}( \rho^{\ast}_{i^{1},\omega_{i}})^{-1}(Z') $. Since $\Phi_{Q^{(1)}}$ intertwines $ \pi^{\ast}_{i^{1},\omega_{i}}$ and  $\rho^{\ast}_{i^{1},\omega_{i}}$, the first statement follows directly. 
	 
     Assume $L$ belongs to $\mathcal{N}_{\nu}$, then there exists some $i$ such that $\epsilon^{\ast}_{i}(L) >0$, so $Z= \Phi_{Q^{(1)}}(L)$ also has $\epsilon_{i}^{\ast}(Z)>0.$ In particular, the subspace ${\rm{Ker}} (\sum\limits_{h \in H, h'=i} x_{h})$ fails the stable condition, and $\Phi_{Q^{(1)}}$ sends $\mathcal{N}_{\nu} $ into the subset $\{ Z \in {\rm{Irr}}\Lambda_{\bV,\mathbf{W}}| Z \cap \Lambda_{\bV,\mathbf{W}}^{s}=\emptyset \}$. Notice that the complements of $\mathcal{N}_{\nu}$ and $\{ Z \in {\rm{Irr}}\Lambda_{\bV,\mathbf{W}}| Z \cap \Lambda_{\bV,\mathbf{W}}^{s}=\emptyset \}$ have the same cardinality equal to the dimension of the weight space of the irreducible highest weight module $L_{v}(\lambda)$. We can see that $\Phi_{Q^{(1)}}$ induces a bijection between the set $\mathcal{N}_{\nu}$ and the set $\{ Z \in {\rm{Irr}}\Lambda_{\bV,\mathbf{W}}| Z \cap \Lambda_{\bV,\mathbf{W}}^{s}=\emptyset \}$, so  $L$ belongs to $\mathcal{N}_{\nu}$ if and only if $\Phi_{Q^{(1)}}(L) \cap \Lambda^{s}_{\bV,\mathbf{W}} $ is empty. The other statements follows from (\ref{KS}).
\end{proof}

\subsubsection{Singular supports of $\mP_{\nu,\omega^1,\omega^2}$}

Given $I^{1}$-graded vector spaces $\mathbf{W}^{1}$ and $I^{2}$-graded vector space  $\mathbf{W}^{2}$ such that $|\mathbf{W}^{1}|=\omega^{1}$ and $|\mathbf{W}^{2}|=\omega^{2}$, we consider the moduli space $\bfEVOff$ for the framed quiver.  We take another $I^{1}$-graded space $\mathbf{W}$ such that  $|\mathbf{W}|=\omega=\omega^{1}+\omega^{2}$. Notice that $$\Hom(\bV_{i},\mathbf{W}^{1}) \oplus  \Hom(\bV_{i},\mathbf{W}^{2}) \cong \Hom(\bV_{i},\mathbf{W}^{1} \oplus \mathbf{W}^{2}),$$ 
the cotangent bundle $T^{\ast}\bfEVOff$ is naturally isomorphic to $\bfEVHf$. We denote an element of $T^{\ast}\bfEVOff$ by $(x,y^1,z^1,y^2,z^2)$, where 
\begin{align*}
&x=(x_{h},x_{\bar{h}})_{h \in \Omega} \in \bfEVH, y^1=(y^1_{i})_{i\in I} \in \bigoplus_{i \in I} \Hom(\bV_{i},\mathbf{W}^{1}_{i}), z^1=(z^1_{i})_{i \in I} \in \bigoplus\limits_{i \in I}\Hom(\mathbf{W}^{1}_{i},\bV_{i}),\\
&y^2=(y^2_{i})_{i\in I} \in \bigoplus\limits_{i \in I} \Hom(\bV_{i},\mathbf{W}^{2}_{i}), z^2=(z^2_{i})_{i \in I} \in \bigoplus\limits_{i \in I}\Hom(\mathbf{W}^{2}_{i},\bV_{i})
\end{align*}
$x=(x_{h},x_{\bar{h}})_{h \in \Omega} \in \bfEVH$, $y^1=(y^1_{i})_{i\in I} \in \bigoplus_{i \in I} \Hom(\bV_{i},\mathbf{W}^{1}_{i})$, $z=(z_{i})_{i \in I} \in \bigoplus\limits_{i \in I}\Hom(\mathbf{W}^{1}_{i},\bV_{i})$, and $z=(z_{i})_{i\in I} \in \bigoplus\limits_{i \in I} \Hom(\bV_{i},\mathbf{W}^{2}_{i})$, $\bar{z}=(\bar{z}_{i})_{i \in I} \in \bigoplus\limits_{i \in I}\Hom(\mathbf{W}^{2}_{i},\bV_{i})$. We denote the image of $(x,y^1,z^1,y^2,z^2)$ in $\bfEVHf$ by $(x,y^1 + y^2, z^1+ z^2)$. 
\begin{lemma}
	Assume $L$ is a simple perverse sheaf in $\mP_{\nu,\omega^1,\omega^2}$, then $SS(L)\subset\Pi_{\bV,\mathbf{W}^1\oplus\mathbf{W}^2}$.
\end{lemma}
\begin{proof}
In order to show $SS(L)$ is contained in $\Pi_{\bV,\mathbf{W}^1\oplus\mathbf{W}^2}$, it suffices to check for $L=L_{\ftnu'\boldsymbol{\omega}^{1}\ftnu''\boldsymbol{\omega}^{2}}$. By applying the proof of \cite[Theorem 13.3 and Corollary 13.6]{MR1088333} to $Q^{(2)}$, we can see that  any point in $SS(L_{\ftnu'\boldsymbol{\omega}^{1}\ftnu''\boldsymbol{\omega}^{2}})$ admits a flag $f$ of flag type $\ftnu'\boldsymbol{\omega}^{1}\ftnu''\boldsymbol{\omega}^{2}$. In the flag $f$, there are two special subspaces. The first one is $\bV'' \oplus \mathbf{W}^{2}$ corresponding to $\ftnu''\boldsymbol{\omega}^{2} $, the second one is $\bV''\oplus \mathbf{W}^{1} \oplus \mathbf{W}^{2}$ corresponding to $\boldsymbol{\omega}^{1}\ftnu''\boldsymbol{\omega}^{2} $. In particular, $\bV''\oplus \mathbf{W}^{1} \oplus \mathbf{W}^{2}$ is $(x,y^1,z^1,y^2,z^2)$-stable, and so $(z^1+z^2)(\mathbf{W}) \subset \mathbf{V}'' $. Similarly,  $\bV'' \oplus \mathbf{W}^{2}$ is $(x,y^1,z^1,y^2,z^2)$-stable, and so $(y^1 + y^2)(\bV'') \subset \mathbf{W}^{2}$. Notice that  $z^2=0$ implies that $(z^1+x^2)(\mathbf{W}^{2})=0$. We get the proof. 
\end{proof}
The above lemma tells us $\Phi_{Q^{(2)}}$ for the two-framed quiver $Q^{(2)}$ sends $\mP_{\nu,\omega^1,\omega^2}$ to  ${\rm{Irr}}\Pi_{\bV,\mathbf{W}^1\oplus\mathbf{W}^2}$. We want to show that it is a bijection. In order to prove it, we firstly study a easier case.  Consider the decomposition $\mathbf{W}^{1}=\mathbf{W}^{1}\oplus 0$, we can define $\mP_{\nu,\omega^{1},0}$ to be the set of simple direct summands of semisimple complexes of the form $L_{\ftnu'\boldsymbol{\omega}^{1}\ftnu''}$ up to shifts, and define the tensor product variety $\Pi_{\bV,\mathbf{W}^{1},0}$.
\begin{lemma}
	The isomorphism $\Phi_{Q^{(2)}}$ of the two-framed quiver induces a bijection between $\mP_{\nu,\omega^{1},0}$ and  ${\rm{Irr}}\Pi_{\bV,\mathbf{W}^{1},0}$. 
\end{lemma}

\begin{proof}
We only need to show that the restriction of $\Phi_{Q^{(2)}}$ on $\mP_{\nu,\omega^{1,\bullet}}$ is onto	 ${\rm{Irr}}\Pi_{\bV,\mathbf{W}^{1,\bullet}}$. 

(a) Given $Z \in {\rm{Irr}}\Pi_{\bV,\mathbf{W}^{1,\bullet}}$, if there exists $i$ such that $\epsilon^{\ast}_{i}(Z)>0$, then we can take $\rho^{\ast}_{i,t}(Z)=Z'$. By induction on dimension vector of $\bV$, we can assume $Z'= \Phi_{Q^{(2)}}(L')$ for some $L' \in \mP_{\nu',\omega^{1,\bullet}}$. Using the induction functor, we can find $L \in \mP_{\nu,\omega^{1,\bullet}}$ such that $\pi_{i,t}^{\ast}(L)=L'$. Then we must have $Z= \Phi_{Q^{(2)}}(L)$. 

(b) Otherwise, $\epsilon^{\ast}_{i}(Z)=0$ for any $i \in I$, we claim that in this case  $Z \subset \Lambda_{\bV,\mathbf{W}^{1}}=   \Pi_{\bV,\mathbf{W}^{1}\oplus 0,0}$. In fact, if $ (x,y^1,z^1,y^2,z^2)$ is in $\Pi_{\bV,\mathbf{W}^{1}\oplus 0,\nu''}$ for some nonzero $\nu'$, then the restriction of $(x,y^1+ y^2,z^1+z^2)$ on $\bV''=\mathbf{S}_{2}(x,y^1+ y^2,z^1+z^2)$ defines a point in $\Lambda_{\bV''}$, it implies that $${\rm{codim}} (\sum_{h \in H, h''=i} {\rm{Im}}x_{h}+{\rm{Im}}(z^1+z^2) )>0.$$
Since $Z \subseteq \Lambda_{\bV,\mathbf{W}^{1}}$, there exists $L \in \mP_{\nu,\omega^{1}}$ such that $\Phi_{Q^{(2)}}(L)=\Phi_{Q^{(1)}}(L)=Z $. 

Combining (a) and (b), we get the proof.
\end{proof}

\begin{proposition} \label{ffss}
	The isomorphism $\Phi_{Q^{(2)}}$ of the two-framed quiver induces a bijection between $\mP_{\nu,\omega^1,\omega^2}$ and  ${\rm{Irr}}\Pi_{\bV,\mathbf{W}^1\oplus\mathbf{W}^2}$. 
\end{proposition}
\begin{proof}
	Consider the natural projections $\pi_{\mathbf{W}^{2}}:\bfEVOff \rightarrow \bfEVOf $ and  $\pi_{\mathbf{W}^{2}}:\Pi_{\bV,\mathbf{W}^1\oplus\mathbf{W}^2} \rightarrow \Pi_{\bV,\mathbf{W}^{1},0} $ forgetting $z$, then by a similar argument as the proof of Proposition 6.1, the statement follows from the lemma above.
\end{proof}

\begin{proposition}
	Let $L$ be a simple perverse sheaf in $ \mP_{\nu,\omega^1,\omega^2}$, then $L$ belongs to $\mathcal{N}_{\nu}$ if and only if $SS(L) \cap \Pi^{s}_{\bV,\mathbf{W}^1\oplus\mathbf{W}^2}=\emptyset,$ if and only if $\Phi_{Q^{(2)}}(L) \cap \Pi^{s}_{\bV,\mathbf{W}^1\oplus\mathbf{W}^2} =\emptyset$.
\end{proposition}
\begin{proof}
	By a similar argument as the proof of Proposition 6.1, we can see that if $L$ belongs to $\mathcal{N}_{\nu}$, then  $\Phi_{Q^{(2)}}(L) \cap \Pi^{s}_{\bV,\mathbf{W}^1\oplus\mathbf{W}^2} =\varnothing$. Notice that the cardinality of $\mP_{\nu,\omega^1,\omega^2} \backslash \mathcal{N}_{\nu}$ equals that of ${\rm{Irr}} \Pi^{s}_{\bV,\mathbf{W}^1\oplus\mathbf{W}^2}$, since both of them are the dimension of the same weight space of $L_{v}(\lambda_{1})\otimes L_{v}(\lambda_{2})$. We can see that $L$ belongs to $\mathcal{N}_{\nu}$ if and only if $\Phi_{Q^{(2)}}(L) \cap \Pi^{s}_{\bV,\mathbf{W}^1\oplus\mathbf{W}^2} =\varnothing$. The other statement follows from equation (\ref{KS}).
\end{proof}

If we denote  the subset $\{L| L \notin \mathcal{N}_{\nu} \}$ of $\mP_{\nu,\omega^1,\omega^2}$  by $\mP^{s}_{\nu,\omega^1,\omega^2}$, then $\Phi_{Q^{(2)}}$ induces to a bijection between $\mP^{s}_{\nu,\omega^1,\omega^2}$ and the set of irreducible components of $\tilde{\mathfrak{Z}}(\nu,\omega)$.  In order to describe the irreducible components of $\tilde{\mathfrak{Z}}_{1}(\nu,\omega)$, we define $\mP^{1}_{\nu,\omega^1,\omega^2}$ to be the set of simple direct summands of those semisimple complexes $L_{\boldsymbol{\omega}^{1}\ftnu \boldsymbol{\omega}^{2}}$ and denote $\mP^{1}_{\nu,\omega^1,\omega^2} \cap \mP^{s}_{\nu,\omega^1,\omega^2}$ by $\mP^{1,s}_{\nu,\omega^1,\omega^2}$.

\begin{proposition} \label{ssZ1}
	The isomorphism $\Phi_{Q^{(2)}}$ of the two-framed quiver induces a bijection between $\mP^{1}_{\nu,\omega^1,\omega^2}$ and  ${\rm{Irr}}\Pi_{\bV,\mathbf{W}^1\oplus\mathbf{W}^2,\nu}$, and it induces a bijection between $\mP^{1,s}_{\nu,\omega^1,\omega^2}$ and ${\rm{Irr}}\tilde{\mathfrak{Z}}_{1}(\nu,\omega)$.
\end{proposition}

\begin{proof}
	When $\mathbf{W}^{1}=0$, then $\mP^{1}_{\nu,\omega^1,\omega^2}=\mP_{\nu,\omega^{2}}$ and $\Pi_{\bV,\mathbf{W}^1\oplus\mathbf{W}^2,\nu}= \Lambda_{\bV,\mathbf{W}^{2}}$ and the proposition follows by Proposition 6.1. For general $\mathbf{W}^{1}$, we consider the following morphism $$i_{\mathbf{W}^{1}}:\mathbf{E}_{\bV,\mathbf{W}^{2},\Omega^{(1)}} \rightarrow \bfEVOff,(x,z^2) \mapsto (x,0,z^2),$$
	$$\pi'_{\mathbf{W}^{1}}: \Pi_{\bV,\mathbf{W}^1\oplus\mathbf{W}^2,\nu} \rightarrow \Lambda_{\bV,\mathbf{W}^{2}},(x,0,0,y^2,x^2) \mapsto (x,0,0,0,z^2). $$
	Note that if $\mathbf{S}_{2}(x,y,z)$ equals $\bV$, the image of $y+ z$ is contained in $\mathbf{W}^{2}$ and $y=0$, so $\pi'_{\mathbf{W}^{1}}$ can be expressed as the formula above.  Since the left multiplication of $L_{\boldsymbol{\omega}^{1}}$ is isomorphic to $i_{\mathbf{W}^{1}\ast}$,  $i_{\mathbf{W}^{1}\ast}(L)= \prod\limits_{i \in I}(\pi_{i,\omega^{1}_{i}})^{-1} (L) $ holds for any $L \in \mP_{\nu,\omega^{2}}$. Similarly, $(\pi'_{\mathbf{W}^{1}})^{-1}(Z)= \prod\limits_{i \in I}(\rho_{i,\omega^{1}_{i}})^{-1} (Z) $ holds for any irreducible component $Z$ of $\Lambda_{\bV,\mathbf{W}^{2}}$.  Since $\Phi_{Q^{(2)}}$ intertwines $\pi_{i,t}$ and $\rho_{i,t}$, the proposition is proved. 
\end{proof}

\subsection{The $\psi$-twist of modules}
For a quiver $Q=(I,H,\Omega)$, its Euler form is defined by
$$\langle\nu',\nu''\rangle_Q=\sum_{i\in I}\nu'_i\nu''_i-\sum_{h\in \Omega}\nu'_{h'}\nu''_{h''}$$
for any $\nu',\nu''\in \mathbb{Z}I$. We denote the opposite quiver by $\bar{Q}=(I,H,\bar{\Omega})$. 

We define $\psi^{-}=\psi^{-}_{Q^{(1)}}=(-1)^{\langle \nu',\nu''\rangle_{Q^{(1)}}  }$ and $\psi^{+}=\psi^{+}_{Q^{(1)}}=(-1)^{\langle \nu',\nu''\rangle_{\overline{Q^{(1)}}}}$ for the framed quiver.  We also define $\psi^{-}=\psi^{-}_{Q^{(2)}}=(-1)^{\langle \nu',\nu''\rangle_{Q^{(2)}}  }$ and $\psi^{+}=\psi^{+}_{Q^{(2)}}=(-1)^{\langle \nu',\nu''\rangle_{\overline{Q^{(2)}}}}$ for the 2 framed quiver.  

We use the notations in section 4. With the action of $[\mathcal{F}^{(r)}_{i}]$, the Grothendieck groups $\mK(\omega),\mK(\omega^1,\omega^2)$ are $\mathbb{N}I$-graded $\mathbf{U}^{-}_{v}(\mathfrak{g})$-modules. 
Let $[\mathcal{F}^{(r)}_{i}]^{\psi}$ be the operator defined by 
$$[\mathcal{F}^{(r)}_{i}]^{\psi}([L])= \psi^{-}(ri,deg([L]))[\mathcal{F}^{(r)}_{i}]([L]).  $$
Then under the action of $[\mathcal{F}^{(r)}_{i}]^{\psi}$, the $\mathcal{A}$-module $\mK(\omega),\mK(\omega^1,\omega^2)$ become $_{\mathcal{A}}\mathbf{U}^{-}_{-v}(\mathfrak{g})$-modules, and we denote them by $\mK(\omega)^\psi,\mK(\omega^1,\omega^2)^\psi$.  In particular, if we take $v=-1$,  the Grothendieck group $\mK(\omega)^\psi_{-1},\mK(\omega^1,\omega^2)^\psi_{-1}$ become $_{\mathbb{Z}}\mathbf{U}^{-}(\mathfrak{g})$-modules.

Similarly, we define an operators $[\mathcal{F}^{(r)}_{i}]^{\psi}$ and $[\mathcal{E}^{(r)}_{i}]^{\psi}$ on $\mL(\omega),\mL(\omega^1,\omega^2)$ by
$$[\mathcal{F}^{(r)}_{i}]^{\psi}([L])= \psi^{-}(ri,deg([L]))[\mathcal{F}^{(r)}_{i}]([L]),  $$
$$[\mathcal{E}^{(r)}_{i}]^{\psi}([L])= \psi^{+}(ri,deg([L]))[\mathcal{E}^{(r)}_{i}]([L]).  $$

\begin{proposition}\label{modulet}
	When $v=-1$, the $\mathbb{Z}$-modules $\mL(\omega)^\psi_{-1},\mL(\omega^1,\omega^2)^\psi_{-1}$ together with the actions of $[\mathcal{F}^{(r)}_{i}]^{\psi}$ and $[\mathcal{E}^{(r)}_{i}]^{\psi}$ become  $_{\mathbb{Z}}\mathbf{U}(\mathfrak{g})$-modules. 
\end{proposition}
\begin{proof}
	We know that $[\mathcal{F}^{(r)}_{i}]^{\psi},i\in I,r\in \mathbb{N}$ satisfy the relations in $\mathbf{U}^{-}(\mathfrak{g})$  and $[\mathcal{E}^{(r)}_{i}]^{\psi},i\in I,r\in \mathbb{N}$ satisfy the relations in $\mathbf{U}^{+}(\mathfrak{g})$. It suffices to show that $[\mathcal{F}_{i}]^{\psi}$ commutes with $[\mathcal{E}_{j}]^{\psi}$ for $i \neq j$, and $[\mathcal{F}_{i}]^{\psi},$ $[\mathcal{E}_{i}]^{\psi}$ satisfy  the $\mathfrak{sl}_{2}$-relation. By definition, for $i\neq j$, $$[\mathcal{F}_{i}]^{\psi}[\mathcal{E}_{j}]^{\psi}([L])= \psi^{-}(i,deg([L]))\psi^{+}(j,deg([L])) \psi^{-}(i,j)[\mathcal{F}_{i}][\mathcal{E}_{j}]([L]), $$
	$$[\mathcal{E}_{j}]^{\psi}[\mathcal{F}_{i}]^{\psi}([L])= \psi^{-}(i,deg([L]))\psi^{+}(j,deg([L])) \psi^{+}(j,i)[\mathcal{E}_{j}][\mathcal{F}_{i}]([L]). $$
	Notice that $\psi^{+}(j,i) =\psi^{-}(i,j) $, we have $$[\mathcal{F}_{i}]^{\psi}[\mathcal{E}_{j}]^{\psi}([L])- [\mathcal{E}_{j}]^{\psi}[\mathcal{F}_{i}]^{\psi}([L])= \psi^{-}(i,deg([L]))\psi^{+}(j,deg([L])) \psi^{-}(i,j) ([\mathcal{F}_{i}][\mathcal{E}_{j}]([L])-)[\mathcal{E}_{j}][\mathcal{F}_{i}]([L]))=0 . $$
	
	For the $\mathfrak{sl}_{2}$-relation, we assume $L$ is a complex on $\bfEVOf$ or $\bfEVOff$ and $\nu$ be the dimension vector of $\bV$. We also set $\tilde{\nu}_{i}= \sum\limits_{h \in H, h'=i} \nu_{i}+\omega^{1}_{i}$ for the framed quivers, and set $\tilde{\nu}_{i}= \sum\limits_{h \in H, h'=i} \nu_{i}+\omega^{1}_{i}+\omega^{2}_{i}$ for the two-framed quivers. By definition, 
	$$[\mathcal{F}_{i}]^{\psi}[\mathcal{E}_{i}]^{\psi}([L])= (-1)^{(2\nu_{i}-\tilde{\nu}_{i}-1 )}[\mathcal{F}_{i}][\mathcal{E}_{i}]([L]), $$
	$$[\mathcal{E}_{i}]^{\psi}[\mathcal{F}_{i}]^{\psi}([L])= (-1)^{(2\nu_{i}-\tilde{\nu}_{i}-1 )}[\mathcal{E}_{i}][\mathcal{F}_{i}]([L]), $$
	hence 
	$$([\mathcal{E}_{i}]^{\psi}[\mathcal{F}_{i}]^{\psi} -[\mathcal{F}_{i}]^{\psi}[\mathcal{E}_{i}]^{\psi})([L])=   (-1)^{(2\nu_{i}-\tilde{\nu}_{i}-1 )}([\mathcal{E}_{i}][\mathcal{F}_{i}]-[\mathcal{F}_{i}][\mathcal{E}_{i}] )([L]).$$
Notice that $2\nu_{i}-\tilde{\nu}_{i}$ is exactly the weight of $[L]$ at $i$, we can see that $[\mathcal{F}_{i}]^{\psi}$ and $[\mathcal{E}_{i}]^{\psi}$ satisfy the $\mathbf{U}(\mathfrak{sl}_{2})$-relation if and only if $[\mathcal{F}_{i}]$ and $[\mathcal{E}_{i}]$ satisfy the $\mathbf{U}_{-1}(\mathfrak{sl}_{2})$-relation.
\end{proof}

Following the proof of Theorem \ref{high} and \ref{tensor}, the module structure of $\mL(\omega^{\clubsuit})_{-1}^{\psi}$ is determined by the following propositions.
\begin{proposition}\label{cant1}
	The Grothendieck group $\mK(\omega^{1})_{-1}^{\psi}$ is  isomorphic to the irreducible highest weight module $_{\mathbb{Z}}L_{1}(\lambda_{1})$ via $$\chi^{\omega^{1},\psi}_{-1}: \mK(\omega^{1})_{-1}^{\psi} \longrightarrow  {_{\mathbb{Z}}L }_{1}(\lambda_{1}),$$
	where $\chi^{\omega^{1},\psi}_{-1}$ is the canonical isomorphism uniquely determined by  $\chi^{\omega^{1},\psi}_{-1}([L_{0}])=v_{\lambda_{1}}$. Here $v_{\lambda_{1}}$ is the fixed highest weight vector of $_{\mathbb{Z}}L_{1}(\lambda_{1})$.
\end{proposition}

\begin{proposition}\label{cant2}
	The Grothendieck group $\mK(\omega^1,\omega^2)_{-1}^{\psi}$ is  isomorphic to the tensor product of irreducible highest weight modules $_{\mathbb{Z}}L_{1}(\lambda_{1}) \otimes{_{\mathbb{Z}}L}_{1}(\lambda_{1})$ via $\tilde{\chi}^{\omega^1,\omega^2,\psi}_{-1}=(\chi^{\omega^{2},\psi}_{-1} \otimes \chi^{\omega^{1},\psi}_{-1}  )\Delta$.
\end{proposition}

Note that after $\psi$-twist, the image of simple perverse sheaves under the twisted canonical isomorphisms do not form the canonical basis, since $[\mathcal{F}^{(r)}_{i}]^{\psi}$ does not have positive property with respect to this basis.

\subsection{Characteristic cycles for framed quivers}
In this section, we build a map from $\mL(\omega^{1})_{-1}^{\psi}$ to $\mathbf{H}(\mathfrak{L}(\omega),\mathbb{Z})=\bigoplus\limits_{\nu \in \mathbb{N}I}\mathbf{H}^{\mathrm{BM}}_{\mathrm{top}}(\mathfrak{L}(\nu,\omega) ,\mathbb{Z})$ by characteristic cycles.

By Proposition \ref{fss}, $\mQ_{\nu,\omega^{1}}$ is a full subcategory of $\mD^{b}_{\bG_{\bV}\times \mathbb{T}}(\bfEVOf,\Lambda_{\bV,\mathbf{W}^{1}}) $, hence $\mK(\omega^{1})_{-1}^{\psi}$ can be regarded as a sub $\mathbb{Z}$-module of $\mK_{0}(\mD^{b}_{\bG_{\bV}\times \mathbb{T}}(\bfEVOf,\Lambda_{\bV,\mathbf{W}^{1}})) $. Compose the characteristic cycle map with $[\textrm{For}_{\bG_{\bV}}^{\bG_{\bV} \times \mathbb{T}}]$, we obtain a morphism
$$\CC_{\bG_{\bV}} : \mK(\nu,\omega^{1})_{-1}^{\psi} \longrightarrow  \mathbf{H}^{\mathrm{BM}}_{\mathrm{top}}(\Lambda_{\bV,\mathbf{W}^{1}} ,\mathbb{Z}), $$
here $[\textrm{For}_{\bG_{\bV}}^{\bG_{\bV} \times \mathbb{T}}]$ is the linear map $ \mK_{0}(\mD^{b}_{\bG_{\bV}\times \mathbb{T}}(\bfEVOf,\Lambda_{\bV,\mathbf{W}^{1}}))  \rightarrow \mK_{0}(\mD^{b}_{\bG_{\bV}}(\bfEVOf,\Lambda_{\bV,\mathbf{W}^{1}})) $ induced by the functor $\textrm{For}_{\bG_{\bV}}^{\bG_{\bV} \times \mathbb{T}}:L \mapsto L[{\rm{dim}}\mathbb{T} ]$.

Now we use the formalism in Section 2.1. Given dimension vector $\nu=\nu''+ri$ and fix graded spaces $\bV'' \subseteq  \bV$ with dimension vectors $\nu''$ and $\nu$ respectively. Let $Y=\mathbf{E}_{\bV'',\mathbf{W}^{1},\Omega^{(1)}}$, $X'=\bfEVOf$ and $V=F_{\mathbf{W}^{1}}$,  $X=\bG_{\bV}\times^{P}Y$ and $W=\bG_{\bV} \times^{P}F_{\mathbf{W}^{1}}$.
Then the induction correspondence of $\mathbf{Ind}$ in section 2.1.1 coincides with the diagram of  the functor $\mathcal{F}^{(r)}_{i}$.  In particular, we have an isomorphism of functors
\begin{equation}\label{forget}
	\mathbf{Ind}^{(r)}_{i} \textrm{For}_{\bG_{\bV}}^{\bG_{\bV} \times \mathbb{T}} \cong  \textrm{For}_{\bG_{\bV}}^{\bG_{\bV} \times \mathbb{T}}\mathcal{F}^{(r)}_{i},
\end{equation} 
and $[\textrm{For}_{\bG_{\bV}}^{\bG_{\bV} \times \mathbb{T}}]$ intertwines the operator $[\mathcal{F}^{(r)}_{i}]$ and the operator $[\mathbf{Ind}^{(r)}_{i}]$. 
Here we denote $\mathbf{Ind}$ by $\mathbf{Ind}^{(r)}_{i}$, in order to emphasize the change of dimension vectors.

Let $\bar{F}_{\mathbf{W}^{1}}$ be the subset of $T^{\ast}\bfEVOf$ consisting of $(x,y,z)$ such that $(x,y,z)(\bV'') \subseteq \bV''$ and denote  $ \bar{F}_{\mathbf{W}^{1}} \cap \Lambda_{\bV,\mathbf{W}^{1}}$ by $\bar{F}^{nil}_{\mathbf{W}^{1}} $. By \cite[Lemma 9.4]{hennecart2024geometric}, we have $\bG_{\bV} \times ^{P} \bar{F}_{\mathbf{W}^{1}} \cong T^{\ast}_{W}(X \times X')$, and we also  have the following commutative diagram,
	\[
\xymatrix{
	\Lambda_{\bV'',\mathbf{W}^{1}}    \ar[d] & \bar{F}^{nil}_{\mathbf{W}^{1}} \ar[d] \ar[l] \ar[r]   &  	\Lambda_{\bV,\mathbf{W}^{1}} \ar[d]\\
	T^{\ast}Y & \bar{F}_{\mathbf{W}^{1}} \ar[l] \ar[r] & T^{\ast}X' ,
}
\]
where the right square is Cartesian. Notice that the second raw of the diagram above is the cotangent correspondence in Section 2.1.2, and the commutative diagram above implies that $$\phi_{2}\phi_{1}^{-1}(\Lambda_{\bV'',\mathbf{W}^{1}}\times^{P} \bG_{\bV}) \subseteq \Lambda_{\bV,\mathbf{W}^{1}},$$ hence the induction operator $\mathbf{Ind}^{(r)}_{i}:\mathbf{H}^{\mathrm{BM}}_{\mathrm{top}}(\Lambda_{\bV'',\mathbf{W}^{1}} ,\mathbb{Z}) \rightarrow \mathbf{H}^{\mathrm{BM}}_{\mathrm{top}}(\Lambda_{\bV,\mathbf{W}^{1}} ,\mathbb{Z}) $ is well-defined. By \cite[Theorem 9.11]{hennecart2024geometric},  the top degree homology groups of CoHA considered in 
\cite{schiffmann2020cohomological} and \cite{davison2024bpsliealgebrasperverse} realize the nilpotent enveloping algebra. Hence together with $[\mathbf{Ind}^{(r)}_{i}]^{\psi}$, the homology group $\bigoplus\limits_{\nu \in \mathbb{N}I}\mathbf{H}^{\mathrm{BM}}_{\mathrm{top}}(\Lambda_{\bV,\mathbf{W}^{1}},\mathbb{Z})$  becomes a $_{\mathbb{Z}}\mathbf{U}^{-}(\mathfrak{g})$-module. We have the following lemma.
\begin{lemma}\label{lemma 6.9}
	After base change to $\mathbb{Q}$, the morphism 
    $$\CC^{\omega^{1}}= \bigoplus\limits_{\nu \in \mathbb{N}I} CC_{\bG_{\bV}}: \mathbb{Q}\otimes_{\mathbb{Z}}\mK(\omega^{1})^{\psi}_{-1} \rightarrow \bigoplus\limits_{\nu \in \mathbb{N}I}\mathbf{H}^{\mathrm{BM}}_{\mathrm{top}}(\Lambda_{\bV,\mathbf{W}^{1}},\mathbb{Q})$$ 
    is a $\mathbf{U}^{-}(\mathfrak{g})$-linear map. Moreover, this map is independent on the choice of $\Omega$ with respect to Fourier-Sato transforms.
\end{lemma}

\begin{proof}
	By equation (\ref{forget}) and Theorem \ref{com}, the map $\CC^{\omega^{1}}$ intertwines the operator $[\mathcal{F}^{(r)}_{i}]$ and the operator $[\mathbf{Ind}^{(r)}_{i}]$. Since we use the same $\psi$-twist, the proposition is proved.  By \cite[Exercise IX.7]{MR1074006}, the characteristic cycle map commutes with Fourier-Sato transforms, hence $\CC^{\omega^{1}}$ is independent on the choice of $\Omega$.
\end{proof}

Now we consider the open embedding $j^{s}:\Lambda^{s}_{\bV,\mathbf{W}^{1}} \rightarrow \Lambda_{\bV,\mathbf{W}^{1}}$, the pull-back  of $j^{s}$ defines 
$$ j^{s\ast}:  \mathbf{H}^{\mathrm{BM}}_{\mathrm{top}}(\Lambda_{\bV,\mathbf{W}^{1}} ,\mathbb{Z}) \rightarrow \mathbf{H}^{\mathrm{BM}}_{\mathrm{top}}(\Lambda^{s}_{\bV,\mathbf{W}^{1}} ,\mathbb{Z}) \cong \mathbf{H}^{\mathrm{BM}}_{\mathrm{top}}(\mathfrak{L}(\nu,\omega^{1}) ,\mathbb{Z}).$$
Compose $j^{s\ast}$ with $\CC_{\bG_{\bV}}$, we obtain 
$$\CC^{s}_{\bG_{\bV}}=j^{s\ast}  \CC_{\bG_{\bV}} :\mK(\nu,\omega^{1})^{\psi}_{-1} \longrightarrow \mathbf{H}^{\mathrm{BM}}_{\mathrm{top}}(\mathfrak{L}(\nu,\omega),\mathbb{Z}), $$
and 
$$\CC^{s,\omega^{1}}=\bigoplus\limits_{\nu \in \mathbb{N}I}\CC^{s}_{\bG_{\bV}}:\mK(\omega^{1})^{\psi}_{-1} \longrightarrow \mathbf{H}(\mathfrak{L}(\omega),\mathbb{Z}). $$

\begin{proposition}\label{premain1}
	After base change to $\mathbb{Q}$, the morphism $CC^{s,\omega^{1}}$ induces an isomorphism of $\mathbf{U}(\mathfrak{g})$-modules,
	$$  \CC^{s,\omega^{1}}:\mathbb{Q}\otimes_{\mathbb{Z}}\mL(\omega^{1})^{\psi}_{-1} \longrightarrow \mathbf{H}(\mathfrak{L}(\omega),\mathbb{Q}).$$
	Moreover, we have the following commutative diagram
		\[
	\xymatrix{
		\mathbb{Q}\otimes_{\mathbb{Z}}\mL(\omega^{1})^{\psi}_{-1} \ar[rr] ^{\CC^{s,\omega^{1}}} \ar[dr]_{\chi^{\omega^{1},\psi}_{-1}} &   &  	\mathbf{H}(\mathfrak{L}(\omega^{1}),\mathbb{Q}) \ar[dl]^{\varphi^{\omega^{1}}}\\
		 &L_{1}(\lambda_{1}), &  
	}
	\]
	where $\chi^{\omega^{1},\psi}_{-1}$ is the canonical isomorphism in Proposition \ref{cant1} and $\varphi^{\omega^{1}}$ is the canonical isomorphism in Theorem \ref{canN1}.
\end{proposition}

\begin{proof}
	Firstly, we show that $\CC^{s,\omega^{1}}$ is well-defined. Let $\mathcal{I}$ be the $\mathcal{A}$-submodule of $\mK(\nu,\omega^{1})$ spanned by those simple perverse sheaves $L$ in $\mathcal{N}_{\nu}$, then $\mL(\nu,\omega^{1})$ is isomorphic to $\mK(\nu,\omega^{1})/\mathcal{I}$ as $\mathcal{A}$-module. In order to show $\CC^{s,\omega^{1}}$ is well-defined, we only need to show $\CC_{\bG_{\bV}}([L])$ is in the kernel of $j^{s\ast}$ for any $L$ in $\mathcal{N}_{\nu}$. However, by Proposition \ref{fss}, $\CC_{\bG_{\bV}}([L])$ is a linear combination of some $[Z]$ with $Z \cap \Lambda^{s}_{\bV,\mathbf{W}^{1}}=\emptyset$, hence $j^{s\ast} \CC_{\bG_{\bV}}([L])=0$ for any $L$ in $\mathcal{N}_{\nu}$ and $\CC^{s,\omega^{1}}:\mL(\omega^{1})^{\psi}_{-1} \longrightarrow \mathbf{H}(\mathfrak{L}(\omega),\mathbb{Z})$ is well-defined.
	
	Secondly, we show that $\CC^{s,\omega^{1}}$ is $\mathbf{U}^{-}(\mathfrak{g})$-linear after base change to $\mathbb{Q}$. More precisely, we prove that $\CC^{s,\omega^{1}}$ intertwines the operator $[\mathcal{F}_{i}]^{\psi}$ and $f_{i}$. Now we consider the commutative diagram 
		\[
	\xymatrix{
		\bG_{\bV}\times^{P}\Lambda^{s}_{\bV'',\mathbf{W}^{1}}    \ar[d] & \bG_{\bV}\times^{P}\bar{F}^{nil,s}_{\mathbf{W}^{1}} \ar[d] \ar[l]_{\phi^{s}_{1}}  \ar[r]^{\phi^{s}_{2}}   &  	\Lambda^{s}_{\bV,\mathbf{W}^{1}} \ar[d]\\
	\bG_{\bV}\times^{P}\Lambda_{\bV'',\mathbf{W}^{1}}& \bG_{\bV}\times^{P}\bar{F}^{nil}_{\mathbf{W}^{1}} \ar[l]_{\phi_{1}} \ar[r]^{\phi_{2}} &\Lambda_{\bV,\mathbf{W}^{1}} ,
	}
	\]
	where the right square is Cartesian. Since $\phi_{2}$ is proper, we have $j^{s\ast}\phi_{2\ast}\phi^{!}_{1} \cong \phi^{s}_{2\ast}\phi^{s!}_{1} j^{s\ast}  $. 
	
	Recall that $\mathfrak{P}_{i}(\nu,\omega^{1})$ is isomorphic to the geometric quotient of the variety $\mathfrak{V}$ consisting of $(x,y,z,\mathbf{S})$, where $(x,y,z) \in \mu^{-1}(0)^{s}$ and $ \mathbf{S}$ is a $x$-stable subspace such that it contains the image of $z$ and its  dimension vector is $\nu''$. Then $\mathfrak{P}_{i}(\nu,\omega^{1}) \cap (\mathfrak{L}(\nu'',\omega) \times \mathfrak{L}(\nu,\omega))$ is isomorphic to the geometric quotient of the variety $\mathfrak{U}$ consisting of $(X,y,z,\mathbf{S})$, where  $(x,y,z) \in \Lambda^{s}_{\bV,\mathbf{W}^{1}}$ and $ \mathbf{S}$ is a $x$-stable subspace such that it contains the image of $z$ and its  dimension vector is $\nu''$. Notice that  $\bG_{\bV}\times^{P}\bar{F}^{nil}_{\mathbf{W}^{1}}$ is isomorphic to $\mathfrak{U}$ via the isomorphism $$ (g,(x,y,z) ) \mapsto (g\cdot(x,y,z),g\cdot \bV'' ),$$
	we can see that $\phi^{s}_{2\ast}\phi^{s!}_{1}$  equals to the cap product of $[\mathrm{sw}(\mathfrak{P}_{i}(\nu,\omega^{1}))]$. Notice that the $\psi$ twist coincides with the sign twist $(-1)^{\langle i , \omega- \mathbf{C}_{\Omega}\nu \rangle}$, we can see that the morphism  $\CC^{s,\omega^{1}}$ is $\mathbf{U}^{-}(\mathfrak{g})$-linear after base change to $\mathbb{Q}$.
	
	Since $\CC^{s,\omega^{1}}([L_{0}] )=[\mathfrak{L}(0,\omega^{1})]$, we can see that $\chi^{\omega^{1},\psi}_{-1}([L_{0}])= \varphi^{\omega^{1}} \CC^{s,\omega^{1}}([L_{0}] ). $   Since both $\chi^{\omega^{1},\psi}_{-1}$ and $\varphi^{\omega^{1}} \CC^{s,\omega^{1}}$ are $\mathbf{U}^{-}(\mathfrak{g})$-linear morphisms and they coincide on the highest weight vector, we must have $\chi^{\omega^{1},\psi}_{-1}=\varphi^{\omega^{1}} \CC^{s,\omega^{1}}$ as $\mathbf{U}^{-}(\mathfrak{g})$-linear isomorphisms. Hence $\CC^{s,\omega^{1}}= (\varphi^{\omega^{1}})^{-1} \chi^{\omega^{1},\psi}_{-1}$ is also a $\mathbf{U}(\mathfrak{g})$-linear isomorphism. 
\end{proof}

\begin{corollary}\label{co1}
	Given a simple perverse sheaf $L$ in $\mP^{s}_{\nu,\omega^{1}}$, let $Z$ be the irreducible component of $\mathfrak{L}(\nu,\omega^{1})$ such that $Z=\Phi_{Q^{(1)}}(L)$, then we have the following equation in ${_{\mathbb{Z}}L}_{1}(\lambda_{1})$, $$\chi^{\omega^{1},\psi}_{-1}([L])=\varphi^{\omega^{1}}([Z] + \sum\limits_{Z \prec Z'}  c_{Z,Z'}[Z']), $$
	where $c_{Z,Z'}$ are constants in $\mathbb{N}$ and $\prec$ is the string order of $B(\lambda_{1})$ with respect to a fixed order $\prec$ of $I$. 
\end{corollary}

\begin{proof}
	By the same argument as Proposition \ref{mono}, we can build a monomial basis $\{m_{L}| L \in \mP^{s}_{\nu,\omega^{1}} \}$ of $\mL(\nu,\omega)$, such that $m_{L}=\pm[L]+\sum \limits_{L \prec L'} c_{L,L'}[L']$, with $c_{L,L'} \in \mathcal{A} $. Notice that $ \Phi_{Q^{(1)}}$ induces an isomorphism of crystals, we can see that $\chi^{\omega^{1},\psi}_{-1}(m_{L})= \varphi^{\omega^{1}}(m_{Z}), $ 
	for $Z=\Phi_{Q^{(1)}}(L)$. Then by the upper-triangular property of the monomial bases, we obtain $\chi^{\omega^{1},\psi}_{-1}([L])=\varphi^{\omega^{1}}(\pm [Z] + \sum\limits_{Z \prec Z'}  c_{Z,Z'}[Z']). $
	Then by the positive property of characteristic cycle map for perverse sheaves, we get the proof.
\end{proof}

\begin{theorem}\label{main1}
    The morphism $\CC^{s,\omega^{1}}$ induces an isomorphism of ${_{\mathbb{Z}}\mathbf{U}}(\mathfrak{g})$-modules,
	$$  \CC^{s,\omega^{1}}:\mL(\omega^{1})^{\psi}_{-1} \longrightarrow \mathbf{H}(\mathfrak{L}(\omega),\mathbb{Z}).$$
	Moreover, we have the following commutative diagram
		\[
	\xymatrix{
		\mL(\omega^{1})^{\psi}_{-1} \ar[rr] ^{\CC^{s,\omega^{1}}} \ar[dr]_{\chi^{\omega^{1},\psi}_{-1}} &   &  	\mathbf{H}(\mathfrak{L}(\omega^{1}),\mathbb{Z}) \ar[dl]^{\varphi^{\omega^{1}}}\\
		 &{_{\mathbb{Z}}L}_{1}(\lambda_{1}), &  
	}
	\]
	where $\chi^{\omega^{1},\psi}_{-1}$ is the canonical isomorphism in Proposition \ref{cant1} and $\varphi^{\omega^{1}}$ is the canonical isomorphism in Theorem \ref{canN1}.
\end{theorem}

\begin{proof}
    By Corollary \ref{co1}, the morphism $\CC^{s,\omega^{1}}$ sends a $\mathbb{Z}$-basis of $\mL(\omega^{1})^{\psi}_{-1}$ to a $\mathbb{Z}$-basis of $\mathbf{H}(\mathfrak{L}(\omega),\mathbb{Z})$, hence it is an isomorphism of $\mathbb{Z}$-modules. Since $\mL(\omega^{1})^{\psi}_{-1}$ is a module of the $\mathbb{Z}$-form ${_{\mathbb{Z}}\mathbf{U}}(\mathfrak{g})$, so is $\mathbf{H}(\mathfrak{L}(\omega),\mathbb{Z})$. By Proposition \ref{premain1}, the  morphism $\CC^{s,\omega^{1}}$ is indeed an isomorphism of ${_{\mathbb{Z}}\mathbf{U}}(\mathfrak{g})$-modules. The commutative diagram also follows from Proposition \ref{premain1}.
\end{proof}

We also have the following corollary, which enhances  the results in \cite{MR1604167} and \cite{fang2023lusztig}.
\begin{corollary}
	The morphism $\varphi^{\omega^{1}}$ in Theorem \ref{canN1} restricts to an isomorphism of $_{\mathbb{Z}}\mathbf{U}(\mathfrak{g}) $ modules $$\varphi^{\omega^{1}}:\mathbf{H}(\mathfrak{L}(\omega^{1}),\mathbb{Z}) \longrightarrow {_{\mathbb{Z}}L}_{1}(\lambda_{1}).$$ Moreover, the constants $c_{Z,Z'}$ in Proposition \ref{mono} are all integers.
\end{corollary}

\begin{remark}
One should also be able to prove that the morphism $\CC^{s,\omega^{1}}$ satisfies the equations $$\CC^{s,\omega^{1}}[\mathcal{F}^{(r)}_{i}]^{\psi} = \pm[\mathrm{sw}(\mathfrak{P}^{(r)}_{i}(\nu,\omega^{1}) )]\CC^{s,\omega^{1}},$$ 
$$ \CC^{s,\omega^{1}}[\mathcal{E}^{(r)}_{i}]^{\psi} = \pm[\mathfrak{P}^{(r)}_{i}(\nu,\omega^{1}) ]\CC^{s,\omega^{1}}, $$
by using functorial properties of characteristic cycles, where $\mathfrak{P}^{(r)}_{i}$ is the generalization of the Hecke correspondence in \cite[Section 5.3]{nakajima2001quiver}.
Then $\CC^{s,\omega^{1}}$ is $_{\mathbb{Z}}\mathbf{U}(\mathfrak{g})$-linear directly, but we do not need such strong properties here.
\end{remark}

\begin{remark}
	 
	In \cite{CDL}, the authors consider the mixed Hodge modules on Grassmannians and construct an exact functor from the mixed Hodge modules to the  $\mathcal{D}_{X,h}$-modules. The authors in \cite{CDL} also point out that there should be a generalization to localized mixed Hodge modules on quivers, which could establish a relation between canonical bases and the quantized quiver varieties considered in \cite{BL} or the categorical realization via coherent sheaves in \cite{CKL}. Indeed, the results in \cite{CDL} solve this problem in $A_{1}$ quiver case.
\end{remark}

As we have mentioned before, $\chi^{\omega^{1},\psi}_{-1}([L])$ is not the  canonical basis of ${_{\mathbb{Z}}L}_{1}(\lambda_{1})$. In order to compare the canonical basis with fundamental classes, we need to consider the specialization at $v=-1$. Take $\psi$-twist on the Borel-Moore homology groups of $\mathfrak{L}(\nu,\omega^{1})$ as what we have done in Proposition \ref{modulet}, we get a $\mathbf{U}_{-1}(\mathfrak{g})$-module $\mathbf{H}(\mathfrak{L}(\omega))^{\psi}$, which is canonically isomorphic to $L_{-1}(\lambda_{1})$ via $\varphi^{\omega^{1},\psi} .$  Then we obtain the following corollary immediately.
\begin{corollary}
	There is a commutative diagram of $_{\mathbb{Z}}\mathbf{U}_{-1}(\mathfrak{g})$-isomorphisms,
	\[
	\xymatrix{
		\mL(\omega^{1})_{-1} \ar[rr] ^{\CC^{s,\omega^{1}}} \ar[dr]_{\chi^{\omega^{1}}_{-1}} &   &  	\mathbf{H}(\mathfrak{L}(\omega),\mathbb{Z})^{\psi} \ar[dl]^{\varphi^{\omega^{1},\psi}}\\
		& _{\mathbb{Z}}L_{-1}(\lambda_{1}). &  
	}
	\]
	Moreover, the transition matrix between the canonical basis at $v=-1$ and the fundamental classes is an upper triangular (with respect to the string order) matrix, whose diagonal elements are $1$ and the other elements are non-negative integers. More precisely, for $Z=\Phi_{Q^{(1)}}(L)$, we have the following equation 
	$$\chi^{\omega^{1}}_{-1}([L])=\varphi^{\omega^{1},\psi}( [Z] + \sum\limits_{Z \prec Z'}  c_{Z,Z'}[Z']), $$ 	with constants $c_{Z,Z'} \in \mathbb{N} $. Or equivalently,
	$$\CC^{s,\omega^{1}}([L])= [Z] + \sum\limits_{Z \prec Z'}  c_{Z,Z'}[Z']. $$
. 
\end{corollary}

\begin{proposition} \label{CC}
	Given a simple perverse sheaf $L$ in $\mP_{\nu}$, we have the following equation in $\mathbf{H}^{\mathrm{BM}}_{\mathrm{top}}(\Lambda_{\bV},\mathbb{Z}), $
	$$\CC_{\bG_{\bV}} ([L]) = [Z] + \sum\limits_{Z \prec Z'}c_{Z,Z'}[Z'], $$
	where $Z'$ runs over $\bG_{\bV}$-invariant irreducible components of $\Lambda_{\bV}$, $c_{Z,Z'}$ are constants in $\mathbb{N}$, and $\prec$ is  the refine string order. 
\end{proposition}
\begin{proof}
	For a fixed dimension vector $\nu$, we can choose $\omega^{1}$ such that $\omega^{1}_{i} > \nu_{i}$. In this case, $j^{s\ast}$ sends a fundamental class $[Z]$ in $\mathbf{H}^{\mathrm{BM}}_{\mathrm{top}}(\Lambda_{\bV,\mathbf{W}^{1}},\mathbb{Z})$ to $[Z \cap \Lambda^{s}_{\bV,\mathbf{W}} ]$ and it has a linear inverse $[Z'] \mapsto [\bar{Z}']$, where $\bar{Z}'$ is the closure of $Z' \subseteq \Lambda^{s}_{\bV,\mathbf{W}^{1}}$ in  $\Lambda_{\bV,\mathbf{W}^{1}}$. In particular, assume $L'$ is a simple perverse sheaf in $\mP_{\nu,\omega^{1}}$ such that $L' \cong \pi^{\ast}_{\mathbf{W}^{1}}L $ up to shifts, then
	$$\CC_{\bG_{\bV}} ([L']) =[\tilde{Z}] + \sum\limits_{\tilde{Z} \prec \tilde{Z}'}  c_{Z,Z'}[\tilde{Z}'],  $$
	where $\tilde{Z}= \Phi_{Q^{(1)}}(L') $ and $\tilde{Z}'$ are irreducible components of $\Lambda_{\bV,\mathbf{W}^{1}}$.  Notice that the characteristic cycle map $\CC$ commutes with $\pi^{\ast}_{\mathbf{W}^{1}}$, and $\pi^{\ast}_{\mathbf{W}^{1}}$ preserves the string order, we obtain $$\CC_{\bG_{\bV}} ([L]) = [Z] + \sum\limits_{Z \prec Z'}c_{Z,Z'}[Z'], $$
	where $Z$ is the unique irreducible component such that $\pi^{-1}_{\mathbf{W}^{1}}(Z)=\tilde{Z}$, and those $Z'$ are the irreducible components such that $\pi^{-1}_{\mathbf{W}^{1}}(Z')=\tilde{Z}'$ for some $\tilde{Z}'$.
\end{proof}

Apply the proposition above to $Q^{(1)}$, we obtain the following proposition.
\begin{corollary} \label{6.16}
	The morphism $\CC^{\omega^{1}}= \bigoplus\limits_{\nu \in \mathbb{N}I} \CC_{\bG_{\bV}}: \mK(\omega^{1})^{\psi}_{-1} \rightarrow \bigoplus\limits_{\nu \in \mathbb{N}I}\mathbf{H}(\Lambda_{\bV,\mathbf{W}^{1}},\mathbb{Z}) $ in Lemma \ref{lemma 6.9} is a $_{\mathbb{Z}}\mathbf{U}^{-}(\mathfrak{g})$-linear isomorphism.
\end{corollary}

\subsection{Characteristic cycles for two-framed quivers}

In this section, we construct a morphism from $\mL(\omega^1,\omega^2)^{\psi}$ to $\mathbf{H}(\tilde{\mathfrak{Z}},\mathbb{Z})= \bigoplus_{\nu \in \mathbb{N}I} \mathbf{H}^{\mathrm{BM}}_{\mathrm{top}}(\tilde{\mathfrak{Z}},\mathbb{Z})$.

By Proposition \ref{ffss}, $\mQ_{\nu,\omega^1,\omega^2}$ is a full subcategory of $\mD^{b}_{\bG_{\bV}\times \mathbb{T}}(\bfEVOff) $, hence $\mK(\omega^1,\omega^2)_{-1}^{\psi}$ can be regarded as a $\mathbb{Z}$-submodule of $\mK_{0}(\mD^{b}_{\bG_{\bV}\times \mathbb{T}}(\bfEVOff) $. Compose the characteristic cycle map with $[\textrm{For}_{\bG_{\bV}}^{\bG_{\bV} \times \mathbb{T}}]$, we obtain a morphism
$$\CC_{\bG_{\bV}} : \mK(\nu,\omega^1,\omega^2)_{-1}^{\psi} \longrightarrow  \mathbf{H}^{\mathrm{BM}}_{\mathrm{top}}(\Pi_{\bV,\mathbf{W}^1\oplus\mathbf{W}^2},\mathbb{Z}).$$

Given dimension vector $\nu=\nu''+ri$ and fix graded spaces $\bV'' \subseteq  \bV$ with dimension vectors $\nu''$ and $\nu$ respectively. Let $Y=\mathbf{E}_{\bV'',\mathbf{W}^1\oplus\mathbf{W}^2,\Omega^{(2)}}$, $X'=\bfEVOff$ and $V=F_{\mathbf{W}^1\oplus\mathbf{W}^2}$,  $X=\bG_{\bV}\times^{P}Y$ and $W=\bG_{\bV} \times^{P}F_{\mathbf{W}^1\oplus\mathbf{W}^2}$.
Then the induction correspondence of $\mathbf{Ind}$ in section 2.1.1 coincides with the diagram of  the functor $\mathcal{F}^{(r)}_{i}$. Let $\bar{F}_{\mathbf{W}^1\oplus\mathbf{W}^2}$ be the subset of $T^{\ast}\bfEVOff \cong \mathbf{E}_{\bV,\mathbf{W},\Omega^{(1)}}$ consisting of $(x,y,z)$ such that $(x,y,z)(\bV'') \subseteq \bV''$ and denote  $ \bar{F}_{\mathbf{W}^1\oplus\mathbf{W}^2} \cap \Pi_{\bV,\mathbf{W}^1\oplus\mathbf{W}^2}$ by $\bar{F}^{nil}_{\mathbf{W}^1\oplus\mathbf{W}^2} $. By \cite[Lemma 9.4]{hennecart2024geometric}, we have $\bG_{\bV} \times ^{P} \bar{F}_{\mathbf{W}^1\oplus\mathbf{W}^2} \cong T^{\ast}_{W}(X \times X')$, and there is a  commutative diagram,
\[
\xymatrix{
	\Pi_{\bV'',\mathbf{W}^1\oplus\mathbf{W}^2}   \ar[d] & \bar{F}^{nil}_{\mathbf{W}^1\oplus\mathbf{W}^2} \ar[d] \ar[l] \ar[r]   &  	\Pi_{\bV,\mathbf{W}^1\oplus\mathbf{W}^2} \ar[d]\\
	T^{\ast}Y & \bar{F}_{\mathbf{W}^1\oplus\mathbf{W}^2} \ar[l] \ar[r] & T^{\ast}X' ,
}
\]
where the right square is Cartesian. Notice that the second raw of the diagram above is the cotangent correspondence in Section 2.1.2, and the commutative diagram above implies that $$\phi_{2}\phi_{1}^{-1}(\Pi_{\bV'',\mathbf{W}^1\oplus\mathbf{W}^2}\times^{P} \bG_{\bV}) \subseteq \Pi_{\bV,\mathbf{W}^1\oplus\mathbf{W}^2},$$ hence the induction operator $\mathbf{Ind}^{(r)}_{i}:\mathbf{H}^{\mathrm{BM}}_{\mathrm{top}}(\Pi_{\bV'',\mathbf{W}^1\oplus\mathbf{W}^2} ,\mathbb{Z}) \rightarrow \mathbf{H}^{\mathrm{BM}}_{\mathrm{top}}(\Pi_{\bV,\mathbf{W}^1\oplus\mathbf{W}^2} ,\mathbb{Z}) $ is well-defined. 
In particular, together with $[\mathbf{Ind}^{(r)}_{i}]^{\psi}$, the homology group $\bigoplus\limits_{\nu \in \mathbb{N}I}\mathbf{H}^{\mathrm{BM}}_{\mathrm{top}}(\Pi_{\bV,\mathbf{W}^1\oplus\mathbf{W}^2},\mathbb{Z})$  becomes a $_{\mathbb{Z}}\mathbf{U}^{-}(\mathfrak{g})$-module. We have the following proposition.
\begin{proposition}
	After base change to $\mathbb{Q}$, the morphism $$\CC^{\omega^1,\omega^2}= \bigoplus\limits_{\nu \in \mathbb{N}I} \CC_{\bG_{\bV}}: \mathbb{Q}\otimes_{\mathbb{Z}}\mK(\omega^1,\omega^2)^{\psi}_{-1} \rightarrow \bigoplus\limits_{\nu \in \mathbb{N}I}\mathbf{H}^{\mathrm{BM}}_{\mathrm{top}}(\Pi_{\bV,\mathbf{W}^1\oplus\mathbf{W}^2},\mathbb{Q}) $$ is a $\mathbf{U}^{-}(\mathfrak{g})$-linear isomorphism.
\end{proposition}
\begin{proof}
	By a similar argument as the proof of Lemma \ref{lemma 6.9}, we can see the morphism is is  $\mathbf{U}^{-}(\mathfrak{g})$-linear. Apply Proposition \ref{CC} to $Q^{(2)}$, we can see that it is an isomorphism.
\end{proof}

Compose $\CC_{\bG_{\bV}} $ with the pull back of $j^{s}:\Pi^{s}_{\bV,\mathbf{W}^1\oplus\mathbf{W}^2} \rightarrow \Pi_{\bV,\mathbf{W}^1\oplus\mathbf{W}^2} $, we can obtain 
$$\CC^{s}_{\bG_{\bV}}=j^{s\ast} \CC_{\bG_{\bV}} : \mK(\nu,\omega^1,\omega^2)_{-1}^{\psi} \longrightarrow  \mathbf{H}^{\mathrm{BM}}_{\mathrm{top}}(\Pi^{s}_{\bV,\mathbf{W}^1\oplus\mathbf{W}^2},\mathbb{Z}) \cong \mathbf{H}^{\mathrm{BM}}_{\mathrm{top}}(\tilde{\mathfrak{Z}}(\nu,\omega),\mathbb{Z}).$$

\begin{theorem}\label{main2}
	The morphism $\CC^{s,\omega^1,\omega^2} =\bigoplus_{\nu \in \mathbb{N}I}\CC^{s}_{\bG_{\bV}}:  \mK(\omega^1,\omega^2)_{-1}^{\psi} \rightarrow \mathbf{H}(\tilde{\mathfrak{Z}},\mathbb{Z})$ induces an isomorphism of $_{\mathbb{Z}}\mathbf{U}(\mathfrak{g})$-modules,
	$$  \CC^{s,\omega^1,\omega^2}:\mL(\omega^1,\omega^2)^{\psi}_{-1} \longrightarrow \mathbf{H}(\tilde{\mathfrak{Z}},\mathbb{Z}).$$
\end{theorem}

\begin{proof}
	Firstly, we show that $\CC^{s,\omega^1,\omega^2}$ is a linear isomorphism. By Proposition \ref{CC}, for any simple perverse sheaf $L \in \mP^{s}_{\nu,\omega^1,\omega^2}$ and irreducible component $Z = \Phi_{Q^{(2)}}(L) \subseteq \mP^{s}_{\nu,\omega^1,\omega^2}$, we have $\CC_{\bG_{\bV}} ([L]) = [Z] + \sum\limits_{Z \prec Z'}c_{Z,Z'}[Z'],$
	where $Z'$ runs over irreducible components of $\Pi_{\bV,\mathbf{W}^1\oplus\mathbf{W}^2} $ and $\prec$ is the string order of the crystal $B_{Q^{(2)}}(\infty)$ of the quantum group $\mathbf{U}^{-}_{v}(\mathfrak{g}(Q^{(2)}) ) $ associated to the quiver $Q^{(2)}$. Hence $$\CC^{s,\omega^1,\omega^2}([L])=[(Z \cap \Pi^{s}_{\bV,\mathbf{W}^1\oplus\mathbf{W}^2})/\bG_{\bV} ] +  \sum\limits_{Z \prec Z'}c_{Z,Z'} [(Z' \cap \Pi^{s}_{\bV,\mathbf{W}^1\oplus\mathbf{W}^2})/\bG_{\bV} ].  $$ 
	Since $(Z \cap \Pi^{s}_{\bV,\mathbf{W}^1\oplus\mathbf{W}^2})/\bG_{\bV} $ is exactly the irreducible component of $\tilde{\mathfrak{Z}}(\nu,\omega)$ corresponding to $L$, and such irreducible components $[(Z \cap \Pi^{s}_{\bV,\mathbf{W}^1\oplus\mathbf{W}^2})/\bG_{\bV}]$ form a basis of $\mathbf{H}^{\mathrm{BM}}_{\mathrm{top}}(\tilde{\mathfrak{Z}}(\nu,\omega),\mathbb{Z})$, we can see that $\CC^{s,\omega^1,\omega^2}$ is a $\mathbb{Z}$-linear isomorphism.

	Secondly, we show that $\CC^{s,\omega^1,\omega^2}$ is $\mathbf{U}^{-}(\mathfrak{g})$-linear after base change to $\mathbb{Q}$. 
	Notice that there is a commutative diagram
	
	\[
	\xymatrix{
		\bG_{\bV}\times^{P}\Pi^{s}_{\bV'',\mathbf{W}^1\oplus\mathbf{W}^2}    \ar[d] & \bG_{\bV}\times^{P}\bar{F}^{nil,s}_{\mathbf{W}^1\oplus\mathbf{W}^2} \ar[d] \ar[l]_{\phi^{s}_{1}}  \ar[r]^{\phi^{s}_{2}}   &  	\Pi^{s}_{\bV,\mathbf{W}^1\oplus\mathbf{W}^2} \ar[d]\\
		\bG_{\bV}\times^{P}\Pi_{\bV'',\mathbf{W}^1\oplus\mathbf{W}^2}& \bG_{\bV}\times^{P}\bar{F}^{nil}_{\mathbf{W}^1\oplus\mathbf{W}^2} \ar[l]_{\phi_{1}} \ar[r]^{\phi_{2}} &\Pi_{\bV,\mathbf{W}^1\oplus\mathbf{W}^2} ,
	}
	\]
	where the right square is Cartesian. Since $\phi_{2}$ is proper, we have $j^{s\ast}\phi_{2\ast}\phi^{!}_{1} \cong \phi^{s}_{2\ast}\phi^{s!}_{1} j^{s\ast}  $. 
	
	Notice that $\mathfrak{P}_{i}(\nu,\omega^{1}) \cap (\tilde{\mathfrak{Z}}(\nu'',\omega) \times \tilde{\mathfrak{Z}}(\nu,\omega))$ is isomorphic to the geometric quotient of the variety $\tilde{\mathfrak{U}}$ consisting of $(X,y,z,\mathbf{S})$, where $(x,y,z) \in \Pi^{s}_{\bV,\mathbf{W}^1\oplus\mathbf{W}^2}$ and $ \mathbf{S}$ is a $X$-invariant subspace such that it contains the image of $z$ and its  dimension vector is $\nu''$. Similarly,  the variety $\bG_{\bV}\times^{P}\bar{F}^{nil}_{\mathbf{W}^1\oplus\mathbf{W}^2}$ is isomorphic to $\tilde{\mathfrak{U}}$, so $\phi^{s}_{2\ast}\phi^{s!}_{1}$  equals to the cap product of $[sw(\mathfrak{P}_{i}(\nu,\omega^{1}))]$. 
	We can see that the morphism  $\CC^{s,\omega^1,\omega^2}$ is a $\mathbf{U}^{-}(\mathfrak{g})$-linear isomorphism after base change to $\mathbb{Q}$.
	
	Since $\CC^{s,\omega^1,\omega^2}$ is a $\mathbf{U}^{-}(\mathfrak{g})$-linear isomorphism between integrable highest weight modules, it must be an isomorphism of $\mathbf{U}(\mathfrak{g})$-modules. Since the image of $\CC^{s,\omega^1,\omega^2}$ is contained in $\mathbf{H}(\tilde{\mathfrak{Z}},\mathbb{Z})$, we get the proof.
\end{proof}

As a corollary, we can use the isomorphism $ \Delta: \mL(\omega^1,\omega^2) \rightarrow \mL(\omega^{2}) \otimes \mL(\omega^{1})$ in Theorem \ref{tensor} to prove the Conjecture \ref{canN2}. Notice that similar coproducts for Yangians $Y(\mathfrak{g})$ have been  defined in \cite{MR3077693} and \cite{maulik2019quantum} in different settings, we do not know the concrete relation between these constructions.

\begin{theorem}\label{main3}
	Let $\delta:\mathbf{H}(\tilde{\mathfrak{Z}},\mathbb{Z}) \longrightarrow \mathbf{H}(\mathfrak{L}(\omega^{2}),\mathbb{Z})  \otimes \mathbf{H}(\mathfrak{L}(\omega^{1}),\mathbb{Z})$ be the unique isomorphism such that the following commutative diagram commutes
	\[
	\xymatrix{
		\mL(\omega^1,\omega^2)^{\psi}_{-1}   \ar[d]_{\CC^{s,\omega^1,\omega^2}}  \ar[r]^-{\Delta} &  \mL(\omega^{2})^{\psi}_{-1}  \otimes \mL(\omega^{1})^{\psi}_{-1}  \ar[d]^{\CC^{s,\omega^{2}} \otimes \CC^{s,\omega^{1}}}      \\
		\mathbf{H}(\tilde{\mathfrak{Z}},\mathbb{Z}) \ar[r]^-{\delta} & \mathbf{H}(\mathfrak{L}(\omega^{2}),\mathbb{Z})  \otimes \mathbf{H}(\mathfrak{L}(\omega^{1}),\mathbb{Z}) ,
	}
	\]
	where $\Delta$ is the morphism in Theorem \ref{tensor} and those $\CC^{s}$ are the morphisms in Theorem \ref{main1} and \ref{main2} respectively. Then after base change to $\mathbb{Q}$, $\delta$ is the unique isomorphism which satisfies the conditions in Conjecture \ref{canN2}.
\end{theorem}

\begin{proof}
	The uniqueness is trivial once such isomorphism exists, so we only need to check that $\delta$ satisfies the conditions in Conjecture \ref{canN2}.
	Let $\mK^{1}(\nu,\omega^1,\omega^2)$ be the $\mathcal{A}$-submodule of $\mK(\nu,\omega^1,\omega^2)$ spanned by those $[L]$ with $L \in \mP^{1}_{\nu,\omega^1,\omega^2}$ and let $\mL^{1}(\nu,\omega^1,\omega^2)$ be the $\mathcal{A}$-submodule of $\mL(\nu,\omega^1,\omega^2)$ spanned by those $[L]$ with $L \in \mP^{1,s}_{\nu,\omega^1,\omega^2}$. Then the left multiplication by $L_{\boldsymbol{\omega}^{1}}$ defines a linear isomorphism $[L_{\boldsymbol{\omega}^{1}}]: \mK(\nu,\omega^{2}) \rightarrow   \mK^{1}(\nu,\omega^1,\omega^2)$ and also induces a linear isomorphism $[L_{\boldsymbol{\omega}^{1}}]: \mL(\nu,\omega^{2}) \rightarrow   \mL^{1}(\nu,\omega^1,\omega^2)$. Regard vertices in $I^{1}$ as unframed vertex, we can see that the operator $[\mathcal{E}_{i}]$ commutes with $[L_{\boldsymbol{\omega}^{1}}]$. In particular, the morphism  $[L_{\boldsymbol{\omega}^{1}}]: \mL(\nu,\omega^{2}) \longrightarrow   \mL^{1}(\nu,\omega^1,\omega^2)$ is $\mathbf{U}_{v}^{+}(\mathfrak{g})$-linear.
	
	By Proposition \ref{ssZ1}, we can see that the restriction of $\CC^{s,\omega^1,\omega^2}$ on $\mL^{1}(\nu,\omega^1,\omega^2)$ defines a linear morphism 
	$\CC^{s,\omega^1,\omega^2}:\mL^{1}(\nu,\omega^1,\omega^2)^{\psi}_{-1} \longrightarrow \mathbf{H}^{\mathrm{BM}}_{\mathrm{top}}(\tilde{\mathfrak{Z}}_{1},\mathbb{Z}). $ By Proposition \ref{CC}, it is an isomorphism. Notice that the left multiplication by $L_{\boldsymbol{\omega}^{1}}$ is isomorphic to $i_{\mathbf{W}^{1} \ast}$, there is a commutative diagram 
		\[
	\xymatrix{
		\mathbb{Q}\otimes_{\mathbb{Z}}\mL^{1}(\nu,\omega^1,\omega^2)^{\psi}_{-1}  \ar[d]_{\CC^{s,\omega^1,\omega^2}}  \ar[rr]^-{ ([L_{\boldsymbol{\omega}^{1}}])^{-1}} & &  \mathbb{Q}\otimes_{\mathbb{Z}}\mL(\nu,\omega^{2})^{\psi}_{-1}   \ar[d]^{\CC^{s,\omega^{2}}}      \\
		\mathbf{H}(\tilde{\mathfrak{Z}}_{1}) \ar[rr]^-{Thom} & & \mathbf{H}(\mathfrak{L}(\omega^{2}))  ,
	}
	\]
	where $Thom$ is the Thom isomorphism in Section 5.2.

    By Proposition \ref{span}, we can see
    $$ {\rm{span}}_{\mathcal{A}} \{[L]|L \in \mP_{\nu,\omega^{2}} \} ={\rm{span}}_{\mathcal{A}} \{[L_{ \ftnu\boldsymbol{\omega}^{2}}]| \ftnu \in \mathcal{S}_{\nu}  \}.  $$
    Notice that the left multiplication by $[L_{\boldsymbol{\omega}^{1}}]$ is isomorphic to the linear map $[i_{\mathbf{W}^{1}\ast}]$ induced by $i_{\mathbf{W}^{1}\ast}$  and its inverse is given by the linear map $[i^{\ast}_{\mathbf{W}^{1}}]$ induced by $i^{\ast}_{\mathbf{W}^{1}}$, we can deduce that $$ {\rm{span}}_{\mathcal{A}} \{[L]|L \in \mP^{1}_{\nu,\omega^1,\omega^2} \} ={\rm{span}}_{\mathcal{A}} \{[L_{ \boldsymbol{\omega}^{1}\ftnu\boldsymbol{\omega}^{2}}]| \ftnu \in \mathcal{S}_{\nu}  \}.  $$
    In particular, $\{[L_{ \boldsymbol{\omega}^{1}\ftnu\boldsymbol{\omega}^{2}}]| \ftnu \in \mathcal{S}_{\nu}  \}]$ is also a  span set  of $\mL^{1}(\nu,\omega^1,\omega^2)^{\psi}_{-1}$.
    Given any $[L_{\boldsymbol{\omega}^{1}\ftnu \boldsymbol{\omega}^{2}}]$, by Proposition \ref{indres formula}, we have the following equation in $\mL(\omega^{2})^{\psi}_{-1} \otimes \mL(\omega^{1})^{\psi}_{-1}$,
    \begin{equation}
    	\begin{split}
    		\Delta ([L_{\boldsymbol{\omega}^{1}\ftnu \boldsymbol{\omega}^{2}}])=& [L_{\ftnu\boldsymbol{\omega}^{2}}]\otimes [L_{\boldsymbol{\omega}^{1}}]+ \sum\limits_{\ftnu',\ftnu''} n_{\ftnu',\ftnu''}[L_{\ftnu''\boldsymbol{\omega}^{2}}]\otimes [L_{\boldsymbol{\omega}^{1} \ftnu'}] \\
    		=&[L_{\ftnu\boldsymbol{\omega}^{2}}]\otimes [L_{\boldsymbol{\omega}^{1}}]\\
    		=&([L_{\boldsymbol{\omega}^{1}}])^{-1}([L_{\boldsymbol{\omega}^{1}\ftnu \boldsymbol{\omega}^{2}}])\otimes [L_{\boldsymbol{\omega}^{1}}], 
    	\end{split}
    \end{equation}	
    since $[L_{\boldsymbol{\omega}^{1} \ftnu'}]=0$ in $\mL(\omega^{1})$ for any flag type $\ftnu'$ non-empty. Since $\Delta(u)= ([L_{\boldsymbol{\omega}^{1}}])^{-1}(u)\otimes [L_{\boldsymbol{\omega}^{1}}]$ holds for a span set of $\mL^{1}(\nu,\omega^1,\omega^2)^{\psi}_{-1} $, $$\Delta(u)= ([L_{\boldsymbol{\omega}^{1}}])^{-1}(u)\otimes [L_{\boldsymbol{\omega}^{1}}] $$ for any $u \in \mL^{1}(\nu,\omega^1,\omega^2)^{\psi}_{-1}$. Notice that $\CC^{s,\omega^{1}} ( [L_{\boldsymbol{\omega}^{1}}] ) =[\mathfrak{L}(0,\omega^{1})]$, we have 
    $\delta(u)= Thom(u) \otimes [\mathfrak{L}(0,\omega^{1})] $ for any $u \in \mathbf{H}(\tilde{\mathfrak{Z}}_{1})$. The theorem is proved.
\end{proof}

We immediately have the following corollary, which enhances  the results in \cite{MR1865400}.
\begin{corollary}
The homology groups  $\mathbf{H}(\tilde{\mathfrak{Z}},\mathbb{Z})$ of the tensor product variety with  integer coefficients have a $_{\mathbb{Z}}\mathbf{U}(\mathfrak{g})$-module structure, which is canonically isomorphic to the tensor product ${_{\mathbb{Z}}L}_{1}(\lambda_{2}) \otimes {_{\mathbb{Z}}L}_{1}(\lambda_{1})$ via the composition of $(\varphi^{\omega^{2}}\otimes \varphi^{\omega^{1}})$ and $\delta$.
\end{corollary}

\begin{remark}
	 Indeed, after using the same analysis of Section 12 in \cite{nakajima2001quiver}, one can deduce that the homology groups $\mathbf{H}(\mathfrak{L}(\omega),\mathbb{Z})$ and $\mathbf{H}(\tilde{\mathfrak{Z}},\mathbb{Z})$ have $_{\mathbb{Z}}\mathbf{U}(\mathfrak{g})$-module structures.
\end{remark}

Denote the composition of $(\varphi^{\omega^{2}}\otimes \varphi^{\omega^{1}})$ and $\delta$ by $\tilde{\varphi}^{\omega^1,\omega^2}$, we get the following commutative diagram of $_{\mathbb{Z}}\mathbf{U}(\mathfrak{g})$-isomorphisms
	\[
\xymatrix{
	\mL(\omega^1,\omega^2)^{\psi}_{-1} \ar[rr] ^{\CC^{s,\omega^1,\omega^2}} \ar[dr]_{\tilde{\chi}^{\omega^1,\omega^2,\psi}_{-1}} &   &  	\mathbf{H}(\tilde{\mathfrak{Z}},\mathbb{Z}) \ar[dl]^{\tilde{\varphi}^{\omega^1,\omega^2}}\\
	& {_{\mathbb{Z}}L}_{1}(\lambda_{2}) \otimes {_{\mathbb{Z}}L}_{1}(\lambda_{1}), &  
}
\]
where  $Can^{\omega^1,\omega^2,\psi}_{-1}$ is the canonical isomorphism in Proposition \ref{cant2}.

Take $\psi$-twist and denote the composition of $(\varphi^{\omega^{2},\psi}\otimes \varphi^{\omega^{1},\psi})$ and $\delta$ by $\tilde{\varphi}^{\omega^1,\omega^2,\psi}$, we get the following commutative diagram of $_{\mathbb{Z}}\mathbf{U}_{-1}(\mathfrak{g})$-isomorphisms 
	\[
\xymatrix{
	\mL(\omega^1,\omega^2)_{-1} \ar[rr] ^{\CC^{s,\omega^1,\omega^2}} \ar[dr]_{\tilde{\chi}^{\omega^1,\omega^2}_{-1}} &   &  	\mathbf{H}(\tilde{\mathfrak{Z}},\mathbb{Z})^{\psi} \ar[dl]^{\tilde{\varphi}^{\omega^1,\omega^2,\psi}}\\
	& {_{\mathbb{Z}}L}_{-1}(\lambda_{2}) \otimes {_{\mathbb{Z}}L}_{-1}(\lambda_{1}), &  
}
\]
where  $\tilde{\chi}^{\omega^1,\omega^2}_{-1}$ is the canonical isomorphism in Theorem \ref{tensor} at $v=-1$. In particular, we can compare the canonical basis of tensor products at $v=-1$ with the fundamental classes.
\begin{corollary}
	The transition matrix between the canonical basis of the tensor product ${_{\mathbb{Z}}L}_{-1}(\lambda_{2}) \otimes {_{\mathbb{Z}}L}_{-1}(\lambda_{1})$ at $v=-1$ and the fundamental classes in $	\mathbf{H}(\tilde{\mathfrak{Z}},\mathbb{Z})^{\psi}$ is an upper triangular (with respect to the string order) matrix, whose diagonal elements are $1$ and the other elements are non-negative integers. More precisly, for $Z=\Phi_{Q^{(2)}}(L)$, we have the following equation 
	$$\tilde{\chi}^{\omega^1,\omega^2}_{-1}([L])=\tilde{\varphi}^{\omega^1,\omega^2,\psi}( [Z] + \sum\limits_{Z \prec Z'}  c_{Z,Z'}[Z']), $$ 	where  $c_{Z,Z'}$ are constants in $\mathbb{N} $,   $Z'$ runs over irreducible components of $\tilde{\mathfrak{Z}}$, and $\prec$ is the refine string order coming from $Q^{(2)}$.
\end{corollary}
\end{spacing}


\begin{thebibliography}{10}
	
	\bibitem{MR4337423}
	P.~N. Achar.
	\newblock {\em Perverse sheaves and applications to representation theory},
	volume 258 of {\em Mathematical Surveys and Monographs}.
	\newblock American Mathematical Society, Providence, RI, [2021] \copyright
	2021.
	
	\bibitem{AHJR}
	P.~N. Achar, A.~Henderson, D.~Juteau, and S.~Riche.
	\newblock Weyl group actions on the {S}pringer sheaf.
	\newblock {\em Proc. Lond. Math. Soc. (3)}, 108(6):1501--1528, 2014.
	
	\bibitem{bao2016canonical}
	H.~Bao and W.~Wang.
	\newblock Canonical bases in tensor products revisited.
	\newblock {\em Amer. J. Math.}, 138(6):1731--1738, 2016.
	
	\bibitem{baumann2011canonical}
	P.~Baumann.
	\newblock The canonical basis and the quantum frobenius morphism.
	\newblock {\em arXiv preprint arXiv:1201.0303}, 2011.
	
	\bibitem{MR1299527}
	J.~Bernstein and V.~Lunts.
	\newblock {\em Equivariant sheaves and functors}, volume 1578 of {\em Lecture
		Notes in Mathematics}.
	\newblock Springer-Verlag, Berlin, 1994.
	
	\bibitem{BL}
	R.~Bezrukavnikov and I.~Losev.
	\newblock Etingof's conjecture for quantized quiver varieties.
	\newblock {\em Invent. Math.}, 223(3):1097--1226, 2021.
	
	\bibitem{CDL}
	S.~Cautis, C.~Dodd, and J.~Kamnitzer.
	\newblock Associated graded of {H}odge modules and categorical {$\mathfrak{sl}_2$}
	actions.
	\newblock {\em Selecta Math. (N.S.)}, 27(2):Paper No. 22, 55, 2021.
	
	\bibitem{CKL}
	S.~Cautis, J.~Kamnitzer, and A.~Licata.
	\newblock Coherent sheaves on quiver varieties and categorification.
	\newblock {\em Math. Ann.}, 357(3):805--854, 2013.
	
	\bibitem{davison2024bpsliealgebrasperverse}
	B.~Davison.
	\newblock Bps lie algebras and the less perverse filtration on the
	preprojective coha, 2024.
	
	\bibitem{fang2023tensor}
	J.~Fang and Y.~Lan.
	\newblock Lusztig sheaves and tensor products of integrable highest weight
	modules, 2023.
	
	\bibitem{fang2022correspondence}
	J.~Fang, Y.~Lan, and J.~Xiao.
	\newblock The correspondence between the canonical and semicanonical bases.
	\newblock {\em arXiv preprint arXiv:2210.16758}, 2022.
	
	\bibitem{fang2023lusztig}
	J.~Fang, Y.~Lan, and J.~Xiao.
	\newblock Lusztig sheaves and integrable highest weight modules, 2023.
	
	\bibitem{HLSS}
	L.~Hennecart.
	\newblock Microlocal characterization of {L}usztig sheaves for affine quivers
	and {$g$}-loops quivers.
	\newblock {\em Represent. Theory}, 26:17--67, 2022.
	
	\bibitem{hennecart2024geometric}
	L.~Hennecart.
	\newblock Geometric realisations of the unipotent enveloping algebra of a
	quiver.
	\newblock {\em Adv. Math.}, 441:Paper No. 109536, 59, 2024.
	
	\bibitem{MR2995184}
	S.-J. Kang and M.~Kashiwara.
	\newblock Categorification of highest weight modules via
	{K}hovanov-{L}auda-{R}ouquier algebras.
	\newblock {\em Invent. Math.}, 190(3):699--742, 2012.
	
	\bibitem{MR1458969}
	M.~Kashiwara and Y.~Saito.
	\newblock Geometric construction of crystal bases.
	\newblock {\em Duke Math. J.}, 89(1):9--36, 1997.
	
	\bibitem{MR1074006}
	M.~Kashiwara and P.~Schapira.
	\newblock {\em Sheaves on manifolds}, volume 292 of {\em Grundlehren der
		mathematischen Wissenschaften [Fundamental Principles of Mathematical
		Sciences]}.
	\newblock Springer-Verlag, Berlin, 1990.
	\newblock With a chapter in French by Christian Houzel.
	
	\bibitem{MR1088333}
	G.~Lusztig.
	\newblock Quivers, perverse sheaves, and quantized enveloping algebras.
	\newblock {\em J. Amer. Math. Soc.}, 4(2):365--421, 1991.
	
	\bibitem{lusztig1992canonical}
	G.~Lusztig.
	\newblock Canonical bases in tensor products.
	\newblock {\em Proc. Nat. Acad. Sci. U.S.A.}, 89(17):8177--8179, 1992.
	
	\bibitem{MR1227098}
	G.~Lusztig.
	\newblock {\em Introduction to quantum groups}, volume 110 of {\em Progress in
		Mathematics}.
	\newblock Birkh\"{a}user Boston, Inc., Boston, MA, 1993.
	
	\bibitem{MR1653038}
	G.~Lusztig.
	\newblock Canonical bases and {H}all algebras.
	\newblock In {\em Representation theories and algebraic geometry ({M}ontreal,
		{PQ}, 1997)}, volume 514 of {\em NATO Adv. Sci. Inst. Ser. C: Math. Phys.
		Sci.}, pages 365--399. Kluwer Acad. Publ., Dordrecht, 1998.
	
	\bibitem{MR1758244}
	G.~Lusztig.
	\newblock Semicanonical bases arising from enveloping algebras.
	\newblock {\em Adv. Math.}, 151(2):129--139, 2000.
	
	\bibitem{Malkin2003Tensor}
	A.~Malkin.
	\newblock Tensor product varieties and crystals: the {$ADE$} case.
	\newblock {\em Duke Math. J.}, 116(3):477--524, 2003.
	
	\bibitem{maulik2019quantum}
	D.~Maulik and A.~Okounkov.
	\newblock Quantum groups and quantum cohomology.
	\newblock {\em Ast\'erisque}, (408):ix+209, 2019.
	
	\bibitem{MR1302318}
	H.~Nakajima.
	\newblock Instantons on {ALE} spaces, quiver varieties, and {K}ac-{M}oody
	algebras.
	\newblock {\em Duke Math. J.}, 76(2):365--416, 1994.
	
	\bibitem{MR1604167}
	H.~Nakajima.
	\newblock Quiver varieties and {K}ac-{M}oody algebras.
	\newblock {\em Duke Math. J.}, 91(3):515--560, 1998.
	
	\bibitem{nakajima2001quiver}
	H.~Nakajima.
	\newblock Quiver varieties and finite-dimensional representations of quantum
	affine algebras.
	\newblock {\em J. Amer. Math. Soc.}, 14(1):145--238, 2001.
	
	\bibitem{MR1865400}
	H.~Nakajima.
	\newblock Quiver varieties and tensor products.
	\newblock {\em Invent. Math.}, 146(2):399--449, 2001.
	
	\bibitem{MR3077693}
	H.~Nakajima.
	\newblock Quiver varieties and tensor products, {II}.
	\newblock In {\em Symmetries, integrable systems and representations},
	volume~40 of {\em Springer Proc. Math. Stat.}, pages 403--428. Springer,
	Heidelberg, 2013.
	
	\bibitem{saito2002crystal}
	Y.~Saito.
	\newblock Crystal bases and quiver varieties.
	\newblock {\em Math. Ann.}, 324(4):675--688, 2002.
	
	\bibitem{schiffmann2020cohomological}
	O.~Schiffmann and E.~Vasserot.
	\newblock On cohomological {H}all algebras of quivers: generators.
	\newblock {\em J. Reine Angew. Math.}, 760:59--132, 2020.
	
	\bibitem{webster2015canonical}
	B.~Webster.
	\newblock Canonical bases and higher representation theory.
	\newblock {\em Compos. Math.}, 151(1):121--166, 2015.
	
	\bibitem{zheng2014categorification}
	H.~Zheng.
	\newblock Categorification of integrable representations of quantum groups.
	\newblock {\em Acta Math. Sin. (Engl. Ser.)}, 30(6):899--932, 2014.
	
\end{thebibliography}
\end{document}